%% file: BucGHU23.tex
\documentclass[11pt,reqno]{amsart}

\usepackage[english]{babel}
\usepackage[utf8]{inputenc}
\usepackage[T1]{fontenc}
\usepackage{lmodern} \normalfont %to load T1lmr.fd 
\DeclareFontShape{T1}{lmr}{bx}{sc} { <-> ssub * cmr/bx/sc }{}
\usepackage{microtype}	% for improved spacing between words and letters

\usepackage{amssymb}
\usepackage{amsmath}
\usepackage{amsthm}
\usepackage{mathtools}
\usepackage{mathrsfs}
\usepackage{accents} % correct spacing for accents in sub- and superscript
\usepackage{etoolbox}
\usepackage{siunitx}
\usepackage{dsfont} % for special 1, \amthds{1i
\allowdisplaybreaks % allow equations to break over pages

\usepackage{bm}

\usepackage{tabularx}

\usepackage{graphicx}
\usepackage{color}
\usepackage[dvipsnames]{xcolor}
\usepackage{circuitikz}
\usepackage{tikz}
\usetikzlibrary{calc,positioning,shapes}
\usetikzlibrary{patterns,decorations.pathmorphing,decorations.markings}
\usepackage{pgfplots}
\pgfplotsset{compat=newest}
\usepackage[margin=10pt,font=small,labelfont=bf,labelsep=endash]{caption}
\usepackage{subcaption}
\usepackage{mleftright}

\usepackage{booktabs}
\usepackage{paralist}
\usepackage{soul}
\usepackage{listings}

\numberwithin{equation}{section}

\usepackage{algorithm}
\usepackage{algorithmicx}
\usepackage{algpseudocode}

\usepackage{cite}
\usepackage{multirow} % for multirows in tabular environment
\usepackage{enumitem} % for labelling items in enumerate
\setlist[enumerate]{label=(\roman*)}
\usepackage{xspace}

\usepackage[colorlinks,linkcolor=teal,citecolor=teal,urlcolor=teal]{hyperref}
\usepackage[nameinlink,noabbrev]{cleveref}

\usepackage{tablefootnote}

\theoremstyle{plain}
\newtheorem{theorem}{Theorem}[section]
\newtheorem{proposition}[theorem]{Proposition}
\newtheorem{lemma}[theorem]{Lemma}
\newtheorem{corollary}[theorem]{Corollary}
\newtheorem{remark}[theorem]{Remark}
\newtheorem{definition}[theorem]{Definition}
\newtheorem{assumption}[theorem]{Assumption}
\newtheorem{example}[theorem]{Example}

\input{def}

\newcommand{\MOR}{\textsf{MOR}\xspace}
\newcommand{\ODE}{\textsf{ODE}\xspace}

\newcommand{\ROM}{\textsf{ROM}\xspace}

\newcommand{\FOM}{\textsf{FOM}\xspace}

\newcommand{\MPG}{\textsf{MPG}\xspace}

\newcommand{\GMG}{\textsf{GMG}\xspace}
\newcommand{\LMG}{\textsf{LMG}\xspace}
\newcommand{\SMG}{\textsf{SMG}\xspace}
\newcommand{\PDE}{\textsf{PDE}\xspace}
\newcommand{\PDEs}{\textsf{PDEs}\xspace}
\newcommand{\MSE}{\textsf{MSE}\xspace}
\newcommand{\POD}{\textsf{POD}\xspace}
\newcommand{\ANNs}{\textsf{ANNs}\xspace}

\newcommand{\NCA}{\textsf{NCA}\xspace}
\newcommand{\sfAE}{\textsf{AE}\xspace}%\AE already defined

\newcommand{\seeSymbol}{$\triangleright\,$}

\title[MOR on manifolds: a differential geometric framework]{Model reduction on manifolds:\\ a differential geometric framework}
\author[P.~Buchfink \and S.~Glas \and B.~Haasdonk \and B.~Unger]{Patrick Buchfink${}^{\dagger,\ddagger}$ \and Silke Glas${}^\ddagger$ \and Bernard~Haasdonk${}^\dagger$ \and Benjamin Unger${}^\star$}

\address{${}^{\dagger}$ Institute of Applied Analysis and Numerical Simulation, University of Stuttgart, Pfaffenwaldring 57, 70569 Stuttgart, Germany}
\email{\{patrick.buchfink,haasdonk\}@ians.uni-stuttgart.de}
\address{${}^{\ddagger}$ Department of Applied Mathematics, University of Twente, P.O. Box 217, 7500 AE Enschede, The Netherlands}
\email{s.m.glas@utwente.nl}
\address{${}^{\star}$ Stuttgart Center for Simulation Science (SC SimTech), University of Stuttgart, Universit\"{a}tsstr.~32, 70569 Stuttgart, Germany}
\email{benjamin.unger@simtech.uni-stuttgart.de}
\date{\today}

\begin{document}

\begin{abstract}
	Using nonlinear projections and preserving structure in model order reduction (MOR) are currently active research fields. In this paper, we provide a novel differential geometric framework for model reduction on smooth manifolds, which emphasizes the geometric nature of the objects involved. The crucial ingredient is the construction of an embedding for the low-dimensional submanifold and a compatible reduction map, for which we discuss several options. Our general framework allows capturing and generalizing several existing MOR techniques, such as structure preservation for Lagrangian- or Hamiltonian dynamics, and using nonlinear projections that are, for instance, relevant in transport-dominated problems. The joint abstraction can be used to derive shared theoretical properties for different methods, such as an exact reproduction result. To connect our framework to existing work in the field, we demonstrate that various techniques for data-driven construction of nonlinear projections can be included in our framework.
\end{abstract}

\maketitle

%
%34A26 = Geometric methods in ordinary differential equations
%34C20 = Transformation and reduction of ordinary differential equations and systems, normal forms
%37C05 = Dynamical systems involving smooth mappings and diffeomorphisms
%37N30 = Dynamical systems in numerical analysis
%65P10 = Numerical methods for Hamiltonian systems including symplectic integrators
%
%------------------------------------------%
% SECTION 1: INTRODUCTION %
%------------------------------------------%
\section{Introduction}
\label{sec:intro}
To remedy the computational cost associated with repeated solutions of high-dimensional differential equations, \emph{model order reduction} (\MOR) has become an established tool over the last three decades. For an overview of \MOR we refer to \cite{BenCOW17,AntBG20,QuaMN16,HesRS16,QuaR14}. The essential idea of \MOR approaches can be summarized as follows: Given a high-dimensional initial value problem, which we refer to as the \emph{full-order model} (\FOM), find a low-dimensional surrogate system, referred to as the \emph{reduced-order model} (\ROM), which is computationally efficient to evaluate. A computationally efficient surrogate model can be interesting in various contexts; for instance
(i) if the \FOM has to be evaluated for many different parameters (e.g., for parameter studies, sampling-based uncertainty quantification, optimization, or inverse problems),
(ii) if the \FOM has to be evaluated in realtime (e.g., for model-based control), or
(iii) if the computational resources are too little to run the \FOM (e.g., on embedded devices).
To achieve this goal, classical linear-subspace \MOR strives to identify a problem-specific low-dimensional linear subspace such that the state of the initial value problem approximately evolves within this subspace. While this is possible in many applications, the existence of a low-dimensional subspace with good approximation properties cannot always be guaranteed. Mathematically, this can be analyzed by studying the Kolmogorov $n$-widths \cite{Kol36,Pin85} (or equivalently, as shown in \cite{UngG19}, the Hankel singular values), associated with the \emph{set of all solutions} (see the forthcoming \Cref{sec:mor_on_mnf} for the precise definition). 
For wave-like phenomena in solutions as observed in transport-problems (e.g., the wave-equation or advection-equation)
it is now well-understood that the $n$-widths decay slowly for certain initial conditions \cite{CagMS19,GreU19}, thus requiring a large dimension of the \ROM for a good approximation. To resolve this problem, different paths are pursued in the literature, most of which try to replace the linear-subspace assumption with a nonlinear ansatz. We refer to \cite[Cha.\,1.3.1]{Sch23} for an overview. In more detail, we assume to be given an initial value problem (the \FOM) of the form
\begin{equation}
	\label{eqn:FOM}
	\ddt{\fx}(t; \param) = \ff(t,\fx(t;\param);\param),\qquad \fx(t_0;\param) = \fx_0(\param)\in\R^{\dimMnf},
\end{equation}
with time interval $\mnfTime \vcentcolon=(\pointTimeInit,\pointTimeEnd)$, $\pointTimeInit < \pointTimeEnd < \infty$ parameter set $\paramSet \subseteq \R^{p}$,
corresponding parameter $\param\in\paramSet$, and right-hand side $\ff\colon \mnfTime\times\R^{\dimMnf}\times \paramSet\to\R^{\dimMnf}$, which we want to solve
for the unknown state $\fx\colon\mnfTime\times\paramSet\to\R^{\dimMnf}$.
Roughly speaking, the idea of \MOR is to construct a projection $\embCoord \circ \redCoord$ using two mappings $\redCoord\colon \R^{\dimMnf} \to \R^\dimMnfRed$ and $\embCoord\colon \R^\dimMnfRed \to \R^{\dimMnf}$ with $\dimMnfRed\ll\dimMnf$ and to then derive a low-dimensional surrogate model of the \FOM with these mappings. In classical linear-subspace \MOR, the mappings $\redCoord$ and $\embCoord$ are linear, i.e., $\redCoord(\fx) = \rT{\fW}\fx$ and $\embCoord(\reduce{\fx}) = \fV\reduce{\fx}$ with matrices $\fW,\fV\in\R^{\dimMnf\times\dimMnfRed}$ satisfying $\rT\fW \fV = \tname{\fI}{\dimMnfRed}$. The associated \ROM is given by
\begin{equation}
	\label{eqn:ROM}
	\ddt{\reduce{\fx}}(t; \param) = \rT{\fW}\ff(t,\fV\reduce{\fx}(t;\param);\param),\qquad \reduce{\fx}(t_0;\param) = \rT{\fW}\fx_0(\param)\in\R^{\dimMnfRed},
\end{equation}
which we solve for the reduced state $\reduce{\fx}\colon\mnfTime\times\paramSet\to\R^{\dimMnfRed}$. In contrast, we allow for nonlinear mappings in this manuscript, which motivates us to study the \FOM~\eqref{eqn:FOM} as a differential equation on a manifold.

%-----------------------------------------------------------------------------%
\subsection{Main contributions}

In the present paper, we introduce a differential geometric framework for \MOR as a unifying framework that contains classical linear-subspace approaches, nonlinear projection frameworks (including machine learning approaches such as autoencoders), and structure-preserving \MOR. Our main contributions are:
\begin{enumerate}
	\item We provide a general differential geometric framework for \MOR on manifolds in \Cref{sec:gen_frame}. Although the geometric elements we introduce in \Cref{sec:prelim} are, of course, not novel, to the best of our knowledge, there is no framework unifying this for \MOR. Moreover, we inspect recent approaches of \MOR on manifolds and show that these fit into this framework \Crefpoint{tab:methods_covered}.
	\item On top of the general framework for \MOR on manifolds, we introduce the \emph{manifold Petrov--Galerkin} (\MPG, \seeSymbol\Cref{subsec:MPG}) and \emph{generalized manifold Galerkin} (\GMG, \seeSymbol\Cref{subsec:MPG}) reduction, which generalize the \MOR techniques from \cite{OttMR23,Lee20}. Moreover, the \GMG reduction forms the basis for novel structure-preserving variants on manifolds for 
	\begin{enumerate}
		\item[(a)] Lagrangian systems \Crefpoint{subsec:MORLagrangian}, which we denote by \emph{Lagrangian manifold Galerkin} (\LMG), thus extending the linear-subspace model reduction methods in \cite{Lall2003,Carlberg2015}, and
		\item[(b)] Hamiltonian systems \Crefpoint{subsec:MORHamiltonian}, which we denote by \emph{symplectic manifold Galerkin} (\SMG), extending the \MOR method in \cite{BucGH21}.
	\end{enumerate}
	\item For the respective \MOR methods, we give an overview of techniques existing in the literature to construct the nonlinear mappings $\redCoord$ and $\embCoord$ in a data-driven fashion in \Cref{sec:embedding_generation}.
\end{enumerate}

Moreover, we provide an exact reproduction result for \MOR on manifolds \Crefpoint{thm:exact_repro} and discuss a relaxation of the point projection property~\eqref{eq:projProperty:pointProjection} for autoencoders in \Cref{lem:autoencoder_inverseProperty}.

We emphasize that we start with the differential geometric perspective already at the level of the \FOM. The main reason for this choice is that starting directly with a differential equation on a manifold highlights the different geometric objects that appear. Note that we focus on a general framework and not on an efficient-to-evaluate surrogate model, which calls upon efficient numerical implementations or additional approximation steps commonly referred to as hyper-reduction.

%-----------------------------------------------------------------------------%
\subsection{State-of-the-art}
In the following we provide an overview of the various aspects of \MOR that fit into this geometric framework.

%-----------------------------------------------------------------------------%
\subsubsection{MOR and manifolds}
Using (smooth) manifolds in the \MOR community is a concept that has been introduced previously. For parametric linear models, interpolation of the linear subspaces or the reduced system matrices on certain manifolds is discussed, for instance, in \cite{AmsCCF09,Son12,MasGSSA19,Zim21} and has recently been extended to a non-intrusive setting in \cite{OulA21}.
To reduce a high-dimensional parameter space during the training phase, the concept of \emph{active manifolds} \cite{BonF21} was developed as a generalization of the so-called \emph{active subspace} \cite{Con15}, which can be interpreted as the dual concept of the Kolmogorov $n$-widths \cite{UngG19}.
Moreover, \emph{lifting techniques} may be used to obtain a nonlinear projection of the original system, e.g., \cite{Gu11,KraW19} or for a non-intrusive setting \cite{QiaKPW20}.

%-----------------------------------------------------------------------------%
\subsubsection{Nonlinear mappings and transport MOR}
If a (localized) quantity is transported through the spatial domain of a \PDE over time, such as a shock wave, then it is often not possible to construct a low-dimensional linear subspace that well-approximates the solution since the Kolmogorov $n$-widths do not decay exponentially. Examples studied in the literature are, for instance, the advection equation \cite{OhlR16,RimML18}, the wave equation \cite{GreU19}, Burgers' equation \cite{CagMS19}, a pulsed detonation combustor \cite{SchRM19}, a wildland fire model \cite{BlaSU21b}, and a rotating detonation engine \cite{MenBALK20}.

Several nonlinear approaches have been presented to overcome the slowly decaying Kolmogorov $n$-widths, many revolving around the symmetry reduction framework \cite{KirA92,RowM00,BeyT04,OhlR13}. We mention here exemplary the shifted proper orthogonal decomposition \cite{ReiSSM18,BlaSU20,BlaSU22}, the Lagrangian reference frame method \cite{MojB17,CopCHC22}, a registration method \cite{Tad20,FerTZ22}, and front transport reduction \cite{KraBHR22}. The central idea underlying these methods is to either first transform the state with a suitable transformation such that the resulting transformed \FOM is easy to approximate or to encode this transformation directly in the \MOR ansatz space.

While the previous approaches are all inspired by the underlying physics of the problem at hand by exploiting the symmetries inherent to the initial value problem, the approaches can be generalized by considering arbitrary nonlinear mappings, for instance, obtained via machine learning paradigms. The natural method for dimensionality reduction is a (shallow) autoencoder \cite{Kas16,HarM17,Lee20,KimCWZ22,OttMR23}. In particular, the work \cite{Lee20} uses terminology from differential geometry and has inspired our work to a large extent. Another parameterization of the nonlinear mappings that are currently investigated is given by polynomials; see, for instance, \cite{Barnett2022,Geelen2022,Benner2022,Issan2023,Sharma2023}.

%-----------------------------------------------------------------------------%
\subsubsection{Structure-preserving MOR for Lagrangian and Hamiltonian systems}
Classical linear-subspace \MOR for Lagrangian systems is discussed in \cite{Lall2003,Carlberg2015}. Notably, the authors of \cite{Lall2003} already mention that the same methods can be used with nonlinear embeddings $\embCoord$, albeit without explicitly formulating the required differential geometric objects. Moreover, we show how a reduced Lagrangian system can be interpreted as a projection of the Euler--Lagrangian vector field using the \GMG reduction \Crefpoint{thm:equivalence_solve_lmg}.

Structure-preserving \MOR for Hamiltonian systems is discussed in \cite{PenM2016,AfkH17,GonWW2017,YilGB23} using linear subspaces 
and in \cite{BucGH21,Sharma2023} for manifolds (in coordinates). A Hamiltonian-preserving Neural Galerkin scheme is presented in \cite{SchSBP23}. Moreover, some of the ideas for structure-preserving \MOR for port-Hamiltonian systems \cite{SchJ14}, a generalization of Hamiltonian systems to open systems, can be used for structure-preserving \MOR for Hamiltonian systems. We refer to \cite[Rem.~8.2]{MehU23} for an overview.

Besides classical \MOR schemes that rely on a given large-scale dynamical system, non-intrusive methods aim to learn a potentially low-dimensional representation from system measurements directly. In the context of learning Hamiltonian systems, we exemplary mention \cite{GreDY19,GruT23,ShaWK22,YilGBB23}.

%-----------------------------------------------------------------------------%
\subsection{Structure of the paper}
To render the manuscript self-contained, we start our exposition by reviewing all necessary concepts from differential geometry \Crefpoint{sec:prelim}. Readers familiar with these concepts might skip this section and directly start with \Cref{sec:mor_on_mnf}, where we introduce our general \MOR framework for initial value problems on manifolds. Additional structure preservation is detailed in \Cref{sec:structurePreservation}, which is based on additional geometrical structures \Crefpoint{sec:DifGeoPartTwo}. A discussion on specific data-driven construction approaches for the required nonlinear mappings is presented in \Cref{sec:embedding_generation} and followed by conclusions \Crefpoint{sec:conclusions}.

%-----------------------------------------------------------------------------%
\subsection{Notation}
\label{sec:notation}
We use the \emph{index notation}, which differentiates between upper indices $\tidx{\velocity}{i}{}$ and lower indices $\tidx{\covec}{}{i}$. Let us emphasize that indices that concern the index notation are underlined. The position of the index indicates the type of the geometric object.
Moreover, we utilize the \emph{Einstein summation convention}, which implies the summation over an index if the index appears twice (once as a lower index and once as an upper index). For an $\dimMnf$-dimensional vector space $\vectorspace$, this notation is used to abbreviate
(i) the linear combination of a basis $\{ \tidx{E}{}{i} \}_{i=1}^{\dimMnf} \subseteq \vectorspace$ with coefficients $\{ \tidx{\velocity}{i}{} \}_{i=1}^{\dimMnf} \subseteq \R$,
(ii) the linear combination of a dual basis $\{ \tidx{F}{i}{} \}_{i=1}^{\dimMnf} \subseteq \vectorspace^\ast$ with coefficients $\{ \tidx{\covec}{}{i} \}_{i=1}^{\dimMnf} \subseteq \R$,
or (iii) the dual product of the respective coefficients,
\begin{align}\label{eq:einstein_summation_convention}
	&\text{(i) } \tidx{\velocity}{i}{} \tidx{E}{}{i}
\vcentcolon= \sum_{i=1}^{\dimMnf} \tidx{\velocity}{i}{} \tidx{E}{}{i} \in \vectorspace,&
	&\text{(ii) } \tidx{\covec}{}{i} \tidx{F}{i}{}
\vcentcolon= \sum_{i=1}^{\dimMnf} \tidx{\covec}{}{i} \tidx{F}{i}{} \in \vectorspace^\ast,&
	&\text{(iii) } \tidx{\velocity}{i}{} \, \tidx{\covec}{}{i}
\vcentcolon= \sum_{i=1}^{\dimMnf} \tidx{\velocity}{i}{} \, \tidx{\covec}{}{i} \in \R.&
\end{align}
Moreover, we use $\stack{\tidx{\velocity}{i}{}}{1 \leq i \leq \dimMnf} \in \R^{\dimMnf}$ to stack scalars $\tidx{\velocity}{i}{} \in \R$ as a vector in $\R^{\dimMnf}$.
Further notation is introduced in \Cref{subsec:bold_notation}.

%------------------------------------------%
% SECTION 2: A PRIMER ON DIFFERENTIAL GEOMETRY %
%------------------------------------------%
\section{A primer on differential geometry}
\label{sec:prelim}
We start this work by recalling several important definitions and results from the theory of smooth manifolds to render this manuscript self-contained. Our presentation is largely based on the monograph \cite{Lee12}. In particular, all material within this section that is not explicitly referenced is adopted from \cite{Lee12}. 

To motivate the forthcoming definitions, we briefly discuss the tools required 
\begin{enumerate}
	\item to formulate a differential equation on a manifold and
	\item to define the submanifold and mappings needed
\end{enumerate}
to perform model reduction on manifolds.
For the differential geometric formulation of the \FOM,
we stepwise introduce the structure of a \emph{smooth manifold} starting from \emph{topological spaces}.
The topology allows us to characterize \emph{continuous functions} and \emph{smooth manifolds} \Crefpoint{subsec:smooth_manifolds}.
Then, having the needed structure at hand, we continue to define \emph{continuously differentiable functions} on smooth manifolds \Crefpoint{subsec:functions}.
Subsequently, we introduce the \emph{tangent space} at a point on the manifold \Crefpoint{subsec:tangent_space} to be able to formulate the \emph{differential of a function} \Crefpoint{subsec:differential}, which is used to generalize the time-derivative of the state to the manifold setting. 
In order to describe the evolution of an initial value problem,
we set the right-hand side to be a \emph{vector field} \Crefpoint{subsec:tangent_bundle_vector_field}.
With these preparations, a differential equation on a manifold can be formulated \Crefpoint{subsec:curves_dyn_sys}. We refer to \Cref{tab:fom_notation_new} for a comparison of a dynamical system on a vector space and on a smooth manifold. Furthermore, for the model reduction framework, we discuss embedded submanifolds \Crefpoint{subsec:submnf}.

\begin{table}[ht]
	\centering
	\small
	\caption{Formulation of a dynamical system in the time interval $\mnfTime$ on a vector space (left) in comparison to a differential geometric formulation (right).}
	\label{tab:fom_notation_new}
	\begin{tabular}{rl@{\hspace{4em}}rl}
		\toprule
		\multicolumn{2}{l}{\textbf{dynamical system on a vector space}} & \multicolumn{2}{l}{\textbf{dynamical system on a manifold}}\\
		\midrule
		$\R^\dimMnf$ & $\dimMnf$-dim.~vector space & $\mnf$ & $\dimMnf$-dim.~smooth manifold\\
		&& $\TpMnf$, $\T\mnf$& tangent space, tangent bundle\\
		$\ff\colon \R^{\dimMnf} \to \R^\dimMnf$ & right-hand side & $\vf\colon \mnf \to \T\mnf$ & vector field\\
		$\fx\colon \mnfTime\to\R^{\dimMnf}$ & solution curve & $\curve\colon \mnfTime\to\mnf$ & solution curve\\
		$\ddt \fx(t) \in \R^\dimMnf$ & time-derivative & $\evalField[\pointTime]{\ddt \curve} \in \TpMnf[{{\evalFun[\pointTime]{\curve}}}]$ & velocity\\\midrule 
		\multicolumn{2}{l}{$\left\{\quad\begin{aligned}
					\ddt\fx(t) &= \ff(\fx(t)) \in \R^{\dimMnf}\\
					\fx(\pointTimeInit) &= \fx_0 \in \R^{\dimMnf}
				\end{aligned}\right.$} & 
		\multicolumn{2}{l}{$\left\{\quad\begin{aligned}
					\evalField[\pointTime]{\ddt \curve} &= \evalField[{
				\evalFun[\pointTime]{\curve}
			}]{\vf} \in \TpMnf[{{\evalFun[\pointTime]{\curve}}}]\\
				\evalFun[\pointTimeInit]{\curve} &= \pointInit \in \mnf
				\end{aligned}\right.$}\\\bottomrule
	\end{tabular}
\end{table}

\subsection{Chart and smooth manifold}
\label{subsec:smooth_manifolds}
For two topological spaces $\calM$ and $\mnfAlt$ \Crefpoint{appx:topology}, a map $F\colon \calM \to \mnfAlt$ is called a \emph{homeomorphism} if (i) it is bijective (and thus the inverse~$\inv{F}\colon \mnfAlt \to \calM$ exists) and (ii) both, $F$ and $\inv{F}$, are continuous.
Moreover, $\calM$ is called \emph{locally homeomorphic to} $\R^{\dimMnf}$ for $\dimMnf\in\N$ if for every point $\point \in \calM$ there exists an open set $\chartdomain \subseteq \calM$ with $\point \in \chartdomain$
and a homeomorphism $F\colon U \to \evalFun[U]{F} \subseteq \R^{\dimMnf}$.
A topological space $\mnf$ is called a \emph{topological manifold} of \emph{dimension} $\dimMnf$ if it is locally homeomorphic to $\R^{\dimMnf}$ (and additionally Hausdorff and second-countable, see e.g.\ \cite[Cha.~1 and App.~A]{Lee12}).
We denote the dimension with $\dim(\mnf) = N$.

Let $\mnf$ be a topological manifold of dimension $\dimMnf$.
A \emph{chart} is a tuple $\chart$ where the \emph{chart domain} $\chartdomain \subseteq \mnf$ is an open set and the \emph{chart mapping} $\chartmap\colon\chartdomain\to\evalFun[\chartdomain]{\chartmap}\subseteq\R^{\dimMnf}$ is a homeomorphism.
For two charts $(\chartdomain,\chartmap)$ and $(\chartdomainAlt,\chartmapAlt)$ with $\chartdomain\cap\chartdomainAlt\neq\emptyset$, we can define the \emph{transition mappings}
\begin{align*}
	\chartmap \circ \inv\chartmapAlt \colon \evalFun[\chartdomain\cap\chartdomainAlt]{\chartmapAlt}\to \evalFun[\chartdomain\cap\chartdomainAlt]{\chartmap}
\qquad\text{and}\qquad
	\chartmapAlt \circ \inv\chartmap \colon \evalFun[\chartdomain\cap\chartdomainAlt]{\chartmap}\to\evalFun[\chartdomain\cap\chartdomainAlt]{\chartmapAlt},
\end{align*}
which are homeomorphisms as composition of homeomorphisms \Crefpoint{Fig:manifold}. 
\begin{figure}[ht]
	\centering
	\def\svgwidth{0.85\textwidth}
	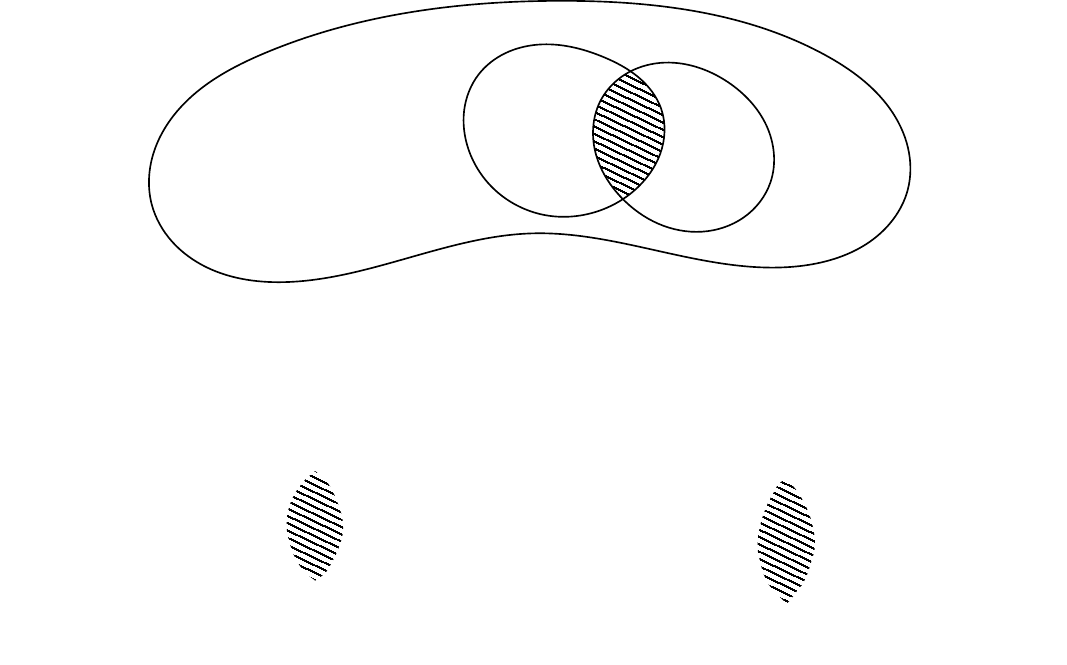
	\caption{Two intersecting chart domains $\chartdomain,\chartdomainAlt \subseteq \mnf$ with respective chart mappings $\chartmap, \chartmapAlt$ and transition mappings $\chartmap \circ \inv\chartmapAlt, \chartmapAlt \circ \inv\chartmap $ on $\evalFun[\chartdomain\cap\chartdomainAlt]{\chartmap}\subseteq\R^{\dimMnf}$ and $\evalFun[\chartdomain\cap\chartdomainAlt]{\chartmapAlt}\subseteq\R^{\dimMnf}$.}
	\label{Fig:manifold}
\end{figure}
The charts $\chart$ and $\chartAlt$ are called \emph{$\calC^k$-compatible} for $k \in \N$ or $k = \infty$ if either $\chartdomain\cap\chartdomainAlt=\emptyset$ or 
\begin{align*}
	\chartmap \circ \inv\chartmapAlt \in\calC^k(\evalFun[\chartdomain\cap\chartdomainAlt]{\chartmapAlt},\evalFun[\chartdomain\cap\chartdomainAlt]{\chartmap})\qquad\text{and}\qquad
	\chartmapAlt \circ \inv\chartmap \in\calC^k(\evalFun[\chartdomain\cap\chartdomainAlt]{\chartmap},\evalFun[\chartdomain\cap\chartdomainAlt]{\chartmapAlt}),
\end{align*}
where differentiability is defined in the classical sense since $\evalFun[\chartdomain\cap\chartdomainAlt]{\chartmap},\evalFun[\chartdomain\cap\chartdomainAlt]{\chartmapAlt}\subseteq\R^{\dimMnf}$.
A collection of charts $\calA = \{(\tname{\chartdomain}{i}, \tname{\chartmap}{i})\mid i\in I\}$ with some index set $I$ is called an \emph{atlas for $\mnf$} if $\mnf = \bigcup_{i\in I} \tname{\chartdomain}{i}$.
The atlas is called \emph{of class $\calC^k$} (or a \emph{$\calC^k$-atlas}) if all charts in $\calA$ are mutually $\calC^k$-compatible. We call a $\calC^k$-atlas $\calA$ \emph{maximal} if all charts that are $\calC^k$-compatible with any chart in $\calA$ are already elements of $\calA$. If $\calA$ is a maximal $\calC^k$ atlas for $\mnf$, then the tuple $(\mnf,\calA)$ is called a \emph{$\calC^k$-manifold} and, in particular, a \emph{smooth manifold} if $k=\infty$. As common in the literature, we omit the explicit mentioning of the maximal atlas whenever possible and say that $\mnf$ is a $\calC^k$-manifold, implicitly assuming a maximal $\calC^k$ atlas to be available.

\subsection{Diffeomorphism and partial derivative}
\label{subsec:functions}
Assume now that we have smooth manifolds $\mnf$ and $\mnfAlt$ of dimension $\dimMnf$ and $\dimMnfAlt$.
A mapping $F\colon\mnf\to\mnfAlt$ is called \emph{of class~$\calC^k$}, or in short notation $F\in\calC^k(\mnf,\mnfAlt)$, if for every~$\point\in\mnf$, there exist charts $\chart$ containing $\point$ and $\chartMnfAlt$ containing ${\evalFun{F}}$ such that $\chartmapMnfAlt\circ F\circ \inv\chartmap\in\calC^k(\evalFun[\chartdomain]{\chartmap},\evalFun[\chartdomainMnfAlt]{\chartmapMnfAlt})$ in the classical sense since $\evalFun[\chartdomain]{\chartmap} \subseteq \R^{\dimMnf}$ and $\evalFun[\chartdomainMnfAlt]{\chartmapMnfAlt} \subseteq \R^{\dimMnfAlt}$.
Note that due to the $\calC^k$-compatibility of the charts, this definition of differentiability does not depend on the choice of the chart.
Throughout the document, a \emph{smooth} mapping is synonymous with mappings of the class $\calC^\infty$.
We restrict ourselves in this work to smooth manifolds and smooth mappings to simplify the presentation.
A smooth bijective map $F \in \calC^\infty(\mnf, \mnfAlt)$ which has a smooth inverse is called a \emph{smooth diffeomorphism (from $\mnf$ to $\mnfAlt$)}.

For calculations, we formulate the derivative in the index notation \Crefpoint{sec:notation}. 
In more detail, we denote for $1 \leq i \leq \dimMnfAlt$ the \emph{$i$-th component function of the chart mapping} as~$\tidx{\chartmap}{i}{}\colon \chartdomain \to \R$
and the \emph{$i$-th component function (of $F$)} with $\tidx{F}{i}{} \vcentcolon= \tidx{\chartmapMnfAlt}{i}{} \circ F\colon \chartdomain \to \R$.
Then, the \emph{$i$-th partial derivative of the $j$-th component of $F \in \calC^1(\mnf, \mnfAlt)$ at $\point \in \mnf$ (in~$\chart$ and~$\chartAlt$)} is defined by
\begin{align}\label{eq:derivative}
	\evalField{\fracdiff{\tidx{F}{j}{}}{\tidx{\chartmap}{i}{}}}
\vcentcolon= \Big( \partial_i
	(\tidx{F}{j}{} \circ \inv\chartmap)
\Big)
(\evalFun[\point]{\chartmap})
\qquad
\text{ for }
	1 \leq i \leq \dimMnf,
\end{align}
where $\partial_i(\cdot)$ describes the $i$-th partial derivative of functions mapping between Euclidean vector spaces.
For scalar-valued functions $f \in \calC^1(\mnf, \R)$, we omit the index, i.e., $\tidx{f}{1}{} \equiv f$ and thus~$\evalField{\fracdiff{\tidx{f}{1}{}}{\tidx{\chartmap}{i}{}}}
= \evalField{\fracdiff{f}{\tidx{\chartmap}{i}{}}}$.
For the derivative of the chart mapping, we obtain
\begin{align}\label{eq:diff_chart_map}
	&\chartmap \in \calC^{\infty}(\mnf, \R^{\dimMnf})&
	&\text{with}&
	&\evalField{
		\fracdiff{
			\tidx{\chartmap}{j}{}
		}{
			\tidx{\chartmap}{i}{}
		}
	}
	=
	\partial_i
	\evalFun[
		\evalFun{\chartmap}
	]
	{
		(\tidx{\chartmap}{j}{} \circ \inv\chartmap)
	}
	= \tidx{\delta}{j}{i}
	\vcentcolon= \begin{cases}
		1,& i=j,\\
		0,& i\neq j,
	\end{cases}
\end{align}
due to $(\tidx{\chartmap}{j}{} \circ \inv\chartmap)
= \tidx{(\chartmap \circ \inv\chartmap)}{j}{}
= \tidx{(\id_{\R^{\dimMnf}})}{j}{}$.
The function $\tidx{\delta}{j}{i}$ is known as the \emph{Kronecker delta}.

\subsection{Tangent and tangent space}
\label{subsec:tangent_space}
Consider a smooth manifold $\mnf$ of dimension $\dimMnf$.
The tangent space of $\mnf$ can be defined in multiple alternative ways (see, e.g., \cite[Sec.~1.6]{Abraham1987} for an overview).
In the present work, we present the \emph{derivation approach}
and closely follow \cite[Sec.~1.7]{Bishop1968}.
For an arbitrary but fixed point $\point \in \mnf$,
we consider the set
\begin{align*}
	\Finfty \vcentcolon= \{ f \in \calC^{\infty}(\chartdomain, \R) \,|\, \chartdomain \subseteq \mnf \text{ open with } \point \in \chartdomain \}.
\end{align*}
Then, a \emph{tangent at $\point \in \mnf$} is a function on this set $\velocity\colon \Finfty \to \R$ which is (i)~linear and (ii)~fulfills the product rule, i.e., for every $f,g \in \Finfty$ and $a,b \in \R$, it holds%
\footnote{%
Note that for the sum/product of functions $f:U_f \to \R$ and $g:U_g \to R$ from $\Finfty$, the domain of the products/sums is shrinked to the intersection $U_f \cap U_g$, which is still open and $m \in U_f \cap U_g$ and thus $(f + g), (f \cdot g) \in \Finfty$.}
\begin{align*}
	&\text{(i) }
\evalFun[af + bg]{\velocity}
=
a \evalFun[f]{\velocity}
+
b \evalFun[g]{\velocity} \in \R,&
	&\text{(ii) }
\evalFun[f\cdot g]{\velocity}
= \evalFun[f]{\velocity} \cdot \evalFun{g}
+  \evalFun{f} \cdot \evalFun[g]{\velocity} \in \R.
\end{align*}
In a broader context, the properties (i) and (ii) define a \emph{derivation}.
The set of all tangents at $\point \in \mnf$
\begin{align} \label{eq:tangent_space}
	\TpMnf
\vcentcolon= \{
	\velocity\colon \Finfty \to \R
	\,|\,
	\velocity \text{ is a tangent at } \point
\}
\end{align}
defines the \emph{tangent space at $\point$}, which can be shown to be an $\dimMnf$-dimensional vector space. Thus, we also refer to elements $\velocity \in \TpMnf$ as \emph{tangent vectors} at $\point$. The $i$-th partial derivative \eqref{eq:derivative} of a scalar-valued function $f \in \Finfty$ can be used to define elements in $\TpMnf$,
\begin{align*}
	&\basisTanAt{i} \in \TpMnf&
&\text{with}&
&\basisTanAt{i}\colon \Finfty \to \R,\;
f \mapsto\evalField{\fracdiff{f}{\tidx{\chartmap}{i}{}}}.
\end{align*}
Moreover, $\left(\basisTanAt{1},\ldots,\basisTanAt{\dimMnf}\right)$ is an ordered basis of $\TpMnf$,
and, thus,
we can represent each tangent vector
\begin{align*}
	\velocity \in \Tp\mnf \qquad\text{with}\qquad\velocity = \tidx{\velocity}{i}{} \evalField{\basisTan{i}},
\end{align*}
where we refer to $\tidx{\velocity}{i}{} \in \R$, $1\leq i\leq \dimMnf$, as the \emph{components (of $\velocity$)}
and where we implicitly sum over $1 \leq i \leq \dimMnf$ by the Einstein summation convention \eqref{eq:einstein_summation_convention}.
Note that for this formalism to work, the index $i$ in the denominator of $\basisTanAt{i}$ counts as a lower index.
In the case of a vector space $\vectorspace$,
the tangent space $\Tp{\vectorspace}$ can be identified with the vector space $\Tp{\vectorspace} \cong \vectorspace$ for all $\point \in \vectorspace$ \cite[p.~59]{Lee12}, in particular $\Tp{\R^{k}} \cong \R^{k}$ for $k\in\N$.

\subsection{Differential and chain rule}
\label{subsec:differential}
Consider smooth manifolds $\mnf$, $\mnfAlt$, and $\mnfAltAlt$ of dimension $\dimMnf$, $\dimMnfAlt$, and $\dimMnfAltAlt$ with charts $\chart$, $\chartMnfAlt$, and $\chartMnfAltAlt$ and a point $\point \in \chartdomain$.
The \emph{differential of a smooth map $F \in \calC^{\infty} (\mnf, \mnfAlt)$ at $\point$} is a linear map
\begin{align}
\label{eq:diff}
	\evalField{\diff{F}} \in \calC^{\infty} (\TpMnf, \TpMnfAlt[{\evalFun{F}}]),
\quad
\tidx{\velocity}{i}{}\, \evalField{\basisTan{i}}
\mapsto
\tidx{\velocity}{i}{}\, \evalField{\fracdiff{\tidx{F}{j}{}}{\tidx{x}{i}{}}}
 \evalField[{\evalFun{F}}]{\basisTanMnfAlt{j}},
\end{align}
which maps between the respective tangent spaces using the $i$-th partial derivative \eqref{eq:derivative} of the $j$-th component function $\tidx{F}{j}{}$ of $F$,
where we sum over $1 \leq i \leq \dimMnf$ and $1 \leq j \leq \dimMnfAlt$ by the Einstein summation convention \eqref{eq:einstein_summation_convention}.
The \emph{chain rule} is an important property of the differential:
For two smooth mappings~$F \in \calC^\infty(\mnf, \mnfAlt)$, $G \in \calC^\infty(\mnfAlt, \mnfAltAlt)$, it holds
\begin{equation}
	\label{eq:chain_rule}
\evalField[\point]{\diff{(G \circ F)}}
	= \evalField[{\evalFun[\point]{F}}]{\diff{G}} \circ
\evalField[\point]{\diff{F}} \colon \TpMnf \to \Tp[{\evalFun[\point]{(G \circ F)}}]\mnfAltAlt.
\end{equation}
In respective charts $\chart$, $\chartMnfAlt$, $\chartMnfAltAlt$ with $\point \in \chartdomain$, $\evalFun[\point]{F} \in \chartdomainMnfAlt$ and $\evalFun[\point]{(G \circ F)} \in \chartdomainMnfAltAlt$, the chain rule reads
\begin{equation*}
\evalField[\point]{\fracdiff{
	\tidx{(G \circ F)}{i}{}
}{
	\tidx{\chartmap}{j}{}
}}
= \evalField[{\evalFun[\point]{F}}]{
	\fracdiff{
		\tidx{G}{i}{}
	}{
		\tidx{\chartmapMnfAlt}{k}{}
	}
}
\evalField[\point]{
	\fracdiff{
		\tidx{F}{k}{}
	}{
		\tidx{\chartmap}{j}{}
	}
}
\qquad
\text{for all}
\quad
\begin{cases}
	1 \leq j \leq \dimMnf,\\
	1 \leq i \leq \dimMnfAltAlt,
\end{cases}
\end{equation*}
where the right-hand side sums over $1 \leq k \leq \dimMnfAlt$ by the Einstein summation convention~\eqref{eq:einstein_summation_convention}.

\subsection{Tangent bundle and vector field}
\label{subsec:tangent_bundle_vector_field}
The \emph{tangent bundle} is the disjoint union of all tangent spaces
\begin{align} \label{eq:tangent_bundle}
	\T\mnf \vcentcolon= \disjointUnion_{\point \in \mnf} \Tp\mnf
\vcentcolon= \{ (\point, \velocity) \;\big|\; \point \in \mnf, \; \velocity \in \Tp\mnf \},
\end{align}
which bundles all points $\point \in \mnf$ and corresponding tangent vectors $\velocity \in \Tp\mnf$ in one set. The \emph{tangent bundle} itself is a smooth manifold of dimension $2\dimMnf$.
In the scope of the present work, we typically use $(\point, \velocity)\in \T\mnf$ to denote elements in $\T\mnf$.
Whenever we have a mapping into a tangent bundle, then we use the notation $(\cdot)\big|_{m}$ to denote the second part of the mapping.
For a given smooth mapping $F \in \calC^\infty(\mnf, \mnfAlt)$,
the \emph{differential (on the tangent bundle)}
\begin{equation}\label{eq:diff_tan_bundle}
	\diff{F} \in \calC^{\infty} (\T\mnf, \T\mnfAlt),\,
(\point, \velocity) \mapsto (\evalFun[\point]{F}, \evalField[\point]{\diff{F}}(\velocity))
\end{equation}
collects the differentials $\evalField[\point]{\diff{F}} \in \calC^{\infty}(\TpMnf, \TpMnfAlt[{{\evalFun[\point]{F}}}])$ at $\point$ for all points $\point \in \mnf$.
For a given chart $\chart$ of $\mnf$,
the differential of the chart mapping $\diff{\chartmap} \in \calC^{\infty} (\T\chartdomain, \T\R^{\dimMnf})$ defines a \emph{natural chart} of $\T\chartdomain$ by identifying $\T\R^{\dimMnf}$ with $\R^{2\dimMnf}$. It maps
\begin{align}\label{eq:chartmapTanSpace}
	\evalFun[\left(
		\point, \tidx{\velocity}{i}{}\, \basisTanAt{i}
	\right)]{\diff{\chartmap}}
= \left( \evalFun{\chartmap}, \stack{\tidx{\velocity}{i}{}}{1 \leq i \leq \dimMnf} \right) \in \R^{2\dimMnf}
\end{align}
since for a chart mapping%
\footnote{This chart mapping may seem redundant as $\chartmapAlt \equiv \id_{\R^\dimMnf}$.
However, we use it to illustrate how $\Tp[{\evalFun{\chartmap}}] \R^{\dimMnf}$ is identified with $\R^\dimMnf$ by using $\chartmapAlt$ to denote the basis vectors $
\evalField[{\evalFun{\chartmap}}]{
	\fracdiff{}{\tidx{\chartmapMnfAlt}{j}{}}
} \in \Tp[{\evalFun{\chartmap}}] \R^{\dimMnf}$.}
$\chartmapMnfAlt \in \calC^\infty (\R^\dimMnf, \R^\dimMnf)$ of $\R^{\dimMnf}$, it holds with \eqref{eq:diff_chart_map} and \eqref{eq:diff} that
\begin{align*}
	\evalFun[\tidx{\velocity}{i}{}\, \basisTanAt{i}]{\evalField{\diff \chartmap}}
= \tidx{\velocity}{i}{}\,
\evalField{
	\fracdiff{\tidx{\chartmap}{j}{}}{\tidx{\chartmap}{i}{}}
}
\evalField[{\evalFun{\chartmap}}]{
	\fracdiff{}{\tidx{\chartmapMnfAlt}{j}{}}
}
=
\tidx{\velocity}{i}{}\,
\tidx{\delta}{j}{i}\,
\evalField[{\evalFun{\chartmap}}]{
	\fracdiff{}{\tidx{\chartmapMnfAlt}{j}{}}
}
= \tidx{\velocity}{i}{}\, \evalField[{\evalFun{\chartmap}}]{
	\fracdiff{}{\tidx{\chartmapMnfAlt}{i}{}}
} \in \Tp[{{\evalFun{\chartmap}}}]{\R^\dimMnf},
\end{align*}
which we identify with $\stack{\tidx{\velocity}{i}{}}{1 \leq i \leq \dimMnf} \in \R^{\dimMnf}$.

Since $\mnf$ and $\T\mnf$ are both smooth manifolds,
we are able to define smooth mappings from $\mnf$ to $\T\mnf$ based on \Cref{subsec:functions}.
A \emph{smooth vector field} is a mapping $\vf \in \calC^{\infty} (\mnf, \T\mnf)$ with $\natProj \circ \vf= \id_{\mnf}$ with $\natProj\colon \T\mnf \to \mnf$, $(\point, \velocity) \mapsto \point$.
It assigns each point $\point \in \mnf$ an element $\evalFun{\vf} \vcentcolon= (\point, \evalField[\point]{\vf} ) \in \T\mnf$ in the tangent bundle,
where we denote the \emph{vector field at $\point \in \mnf$} with $\evalField[\point]{\vf} \in \TpMnf$.
The \emph{set of all smooth vector fields on $\mnf$} is denoted with $\smoothVfs$.

\subsection{Curve and initial value problem}
\label{subsec:curves_dyn_sys}
For a given smooth manifold $\mnf$ and an interval $\mnfTime\vcentcolon=(\pointTimeInit,\pointTimeEnd)$ with
$\pointTimeInit < \pointTimeEnd < \infty$, we call the mapping $\curve \in \calC^{\infty}(\mnfTime, \mnf)$ a \emph{smooth curve}.
We refer to elements $\pointTime \in \mnfTime$ as \emph{time points}.
By custom, we use for the derivative w.r.t.~time the notation $\ddt(\cdot)$.
The \emph{velocity of a curve $\curve$ at $\pointTime \in \mnfTime$} is
\begin{align*}
	\evalField[\pointTime]{\ddt \curve}
\vcentcolon= \left(
	\evalField[\pointTime]{
		\ddt \tidx{\curve}{i}{}
	}
\right)
\basisTanAt[{{\evalFun[\pointTime]{\curve}}}]{i}
\in \TpMnf[{{\evalFun[\pointTime]{\curve}}}],
\end{align*}
i.e., an element in the tangent space based on the (classical) derivative of the component functions $\tidx{\curve}{i}{}\colon \R \supseteq \mnfTime \to \R$ of the curve.\footnote{%
	Alternatively, the velocity of a curve can be understood in terms of the derivative introduced in \Cref{subsec:differential}
	with $\evalField[\pointTime]{\ddt \curve} = \evalField[\pointTime]{\diff \curve}$.
	In the presented notation, this would require to understand $\mnfTime$ as a smooth manifold with the chart $\chartMnfTime$,
	chart mapping $\chartmapMnfTime \equiv \id_{\mnfTime}\colon \mnfTime \to \R$
	and the derivative \smash{$\evalField[\pointTime]{\ddt \tidx{\curve}{i}{}} = \evalField[\pointTime]{\fracdiff{\tidx{\curve}{i}{}}{\tidx{\chartmapMnfTime}{1}{}}}$}.
}

For a smooth vector field $\vf \in \smoothVfs$, we call $\curve \in \calC^\infty(\mnfTime, \mnf)$ an \emph{integral curve of $\vf$} with \emph{initial value} $\pointInit \in \mnf$, if
\begin{equation}\label{eq:integral_curve}
	\left\{\quad\begin{aligned}
		\evalField[\pointTime]{\ddt \curve}
&= \evalField[{{\evalFun[\pointTime]{\curve}}}]{\vf}
\in \TpMnf[{{\evalFun[\pointTime]{\curve}}}], & t\in\mnfTime\\
	\evalFun[\pointTimeInit]{\curve} &= \pointInit \in \mnf.
\end{aligned}\right.
\end{equation}
We refer to \eqref{eq:integral_curve} as an \emph{initial value problem (on $\mnf$)}.
For a given chart $\chart$, the system~\eqref{eq:integral_curve} can be solved via an $\dimMnf$-dimensional initial value problem on $\R^{\dimMnf}$
\begin{align}
	\label{eq:dyn_sys_index_notation}
	&\evalField[\pointTime]{\ddt \tidx{\curve}{i}{}}
	= \tidx{\left(\evalField[{{\evalFun[\pointTime]{\curve}}}]{\vf}\right)}{i}{} \in \R
	\qquad \text{for } 1 \leq i \leq \dimMnf,&
	&\evalFun[\pointTimeInit]{\tidx{\curve}{i}{}}
= \evalFun[\pointInit]{\tidx{\chartmap}{i}{}} \in \R.
\end{align}
Due to the assumption of a smooth vector field,
we know that there exists a unique integral curve by the fundamental theorem on flows \cite[Thm.~9.12]{Lee12}, if the final time $\pointTimeEnd$ is small enough.
If we assume that there exists a time interval such that all integral curves exist for a set $\tname{M}{0} \subseteq \mnf$ with the starting points $\pointInit \in \tname{M}{0}$,
the \emph{flow of $\vf$} can be defined as
\begin{align*}
	\flowAt[\pointTime]\colon \tname{M}{0} \to \mnf,\quad
\pointInit \mapsto \evalFun[\pointTime; \pointInit]{\curve}.
\end{align*}

\subsection{Bold notation}
\label{subsec:bold_notation}
We introduce a notation that collects all previously introduced types of differential geometric objects (like points, functions, tangent vectors) in a fixed chart
and thereby reduces to computations in $\R$-vector spaces.
We refer to this notation as the \emph{bold notation}\footnote{Be aware that bold symbols may be used for other purposes in other scripts on differential geometry.}.
For a given smooth manifold $\mnf$ with a chart $\chart$,
we use
\begin{align*}
	&\chartmap \in \calC^{\infty} (\chartdomain, \R^{\dimMnf}),&
	&\evalField{\diff \chartmap} \in \calC^{\infty} (\Tp{\chartdomain}, \R^{\dimMnf}) \quad \text{with } \point \in \chartdomain,&
	&\diff \chartmap \in \calC^{\infty} (\T\chartdomain, \R^{2\dimMnf})&
\end{align*}
to map the different types of objects accordingly,
where we identify $\Tp[{{\evalFun{\chartmap}}}]\R^{\dimMnf}$ with $\R^{\dimMnf}$ for $\evalField{\diff \chartmap}$ and $\T\R^{\dimMnf}$ with $\R^{2\dimMnf}$ for ${\diff \chartmap}$.
Let us state clearly that (i) this formulation loses geometrical information
(as it treats different types of objects as a vector in $\R^\dimMnf$)
and (ii) it only works for one fixed chart (since the explicit dependence on the chart is neglected).
However, this formulation can be helpful for readers new to the field of differential geometry with more background in classical numerical analysis and engineering.
The notation for the different types of differential geometric objects for two smooth manifolds $\mnf, \mnfAlt$ with charts $\chart$, $\chartAlt$, respectively, are summarized in \Cref{tab:translation_bold_notation}.

\begin{table}
	\centering
	\caption{Bold notation for two smooth manifolds $\mnf,\mnfAlt$ with charts $\chart, \chartAlt$, respectively.}
	\label{tab:translation_bold_notation}
	\small
	\begin{tabular}{lll}
		\toprule
		\textbf{type}
& \textbf{element}
& \textbf{bold notation}\\
		\midrule
		point
& $\point \in \chartdomain \subseteq \mnf$
& $\pointCoord \vcentcolon= \evalFun{\chartmap} \in \R^{\dimMnf}$\\
		mapping
& $F \in \calC^{\infty} (\chartdomain, \chartdomainAlt)$
& $\fF \vcentcolon= \chartmapAlt \circ F \circ \inv\chartmap: \R^\dimMnf \supseteq \chartmap(\chartdomain) \to \chartmapAlt(\chartdomainAlt) \subseteq \R^{\dimMnfAlt}$\\
		tangent vector
& $\velocity = \tidx{\velocity}{i}{} \basisTanAt{i} \in \Tp{\chartdomain}$
& $\velocityCoord
\vcentcolon= \stack{\tidx{\velocity}{i}{}}{1\leq i\leq \dimMnf} \in \R^{\dimMnf}$\\
		Jacobian matrix\tablefootnote{%
To be more precise, the Jacobian matrix is the coordinate matrix of the linear mapping described by the differential in coordinates $
\evalField[{{\evalFun{F}}}]{\diff \chartmapMnfAlt}
\circ
\evalField{\diff{F}}
\circ
\inv{\evalField{\diff \chartmap}}:
\R^{\dimMnf} \to \R^{\dimMnfAlt}
$.
Moreover, we use in the last column of this row a notation for stacking scalars as matrices similarly to stacking scalars as vectors from \Cref{sec:notation}.
}
& $\evalField{\diff{F}} \in \calC^{\infty} (\Tp{\chartdomain}, \Tp[{\evalFun{F}}]{\chartdomainMnfAlt})$
& $	\evalField[\pointCoord]{\D \fF} \vcentcolon=
\stack{
	\evalField{\fracdiff{\tidx{F}{i}{}}{\tidx{\chartmap}{j}{}}}
}{\substack{
	1\leq i \leq \dimMnfAlt,\\
	1\leq j \leq \dimMnf
}}
\in \R^{\dimMnfAlt \times \dimMnf}$\\
		\midrule
		dynamical system
& $ \left\{\quad\begin{aligned}
	\evalField[\pointTime]{\ddt \curve} &= \evalField[{{\evalFun[\pointTime]{\curve}}}]{\vf} \in \Tp[{{\evalFun[\pointTime]{\curve}}}]{\chartdomain},\\
	\curve(\pointTimeInit) &= \pointInit \in \mnf
\end{aligned}\right.
$
& $ \left\{\quad\begin{aligned}
	\evalField[\pointTime]{\ddt \curveCoord} &= \evalField[{{\evalFun[\pointTime]{\curveCoord}}}]{\vfCoord} \in \R^{\dimMnf},\\
	\curveCoord(\pointTimeInit) &= \pointInitCoord \in \R^{\dimMnf}
\end{aligned}\right.
$\\
		\bottomrule
	\end{tabular}
\end{table}

\subsection{Embedding and embedded submanifold}
\label{subsec:submnf}
Consider two smooth manifolds~$\mnfRed$ and~$\mnf$ of dimension $\dimMnfRed$ and $\dimMnf$, respectively.
A smooth mapping $F \in \calC^{\infty}(\mnfRed, \mnf)$ is called an \emph{immersion} if the respective differential $\evalField[\pointRed]{\diff{F}}\colon \TpMnfRed \to \TpMnf[{{\evalFun[\pointRed]{F}}}]$ is injective at each point $\pointRed \in \mnfRed$.
Moreover, $F$ is called a \emph{smooth embedding} if it is a smooth immersion and a homeomorphism onto its image $\evalFun[\smash{\mnfRed}]{F} \subseteq \mnf$.
For a given smooth embedding $\emb \in \calC^{\infty}(\mnfRed, \mnf)$, the image $\mnfSub$ is an $\dimMnfRed$-dimensional smooth manifold, which is called an \emph{embedded (or regular) submanifold} of $\mnf$.
We denote the tangent space of $\mnfSub$ at $\evalFun[\pointRed]{\emb}$ with $\TpMnfSub \vcentcolon= \evalFun[\smash{\TpMnfRed}]{\evalField[\pointRed]{\diff{\emb}}}$.
From the assumptions, it follows automatically that the embedding $\emb$ is a smooth diffeomorphism onto its image \cite[Prop.~5.2]{Lee12}.

\begin{lemma}\label{lem:emb_subspaces_from_point_red}
	Consider smooth manifolds $\mnfRed, \mnf$ and smooth mappings $\emb \in \calC^{\infty}(\mnfRed, \mnf)$
	and $\redMapPoint \in \calC^{\infty}(\mnf, \mnfRed)$
	with $\redMapPoint \circ \emb \equiv \id_{\mnfRed}$.
	Then, $\emb$ is a smooth embedding and $\mnfSub \subseteq \mnf$ is an embedded submanifold.
\end{lemma}
\begin{proof}
	\Crefpoint{appx:proof_emb_subspaces_from_point_red}.
\end{proof}

%------------------------------------------%
% SECTION 3: MOR ON MANIFOLDS %
%------------------------------------------%
\section{Model order reduction on manifolds}
\label{sec:mor_on_mnf}
With the geometric sundries at hand, we can now introduce \emph{model order reduction on manifolds.} We start with the general framework for model order reduction \Crefpoint{sec:gen_frame}. Then, we detail conditions such that exact reproduction can be achieved \Crefpoint{subsec:exact_reproduction}. Finally, we present an example fitting this framework, the so-called manifold Petrov--Galerkin \Crefpoint{subsec:MPG}. 

\subsection{General framework}\label{sec:gen_frame}
This section sits at the heart of this paper and introduces the general framework upon which the remainder is built. We start this section by defining the \FOM on manifolds \Crefpoint{subsec:gen_FOM}. We then focus on the goal that \MOR strives to achieve and what assumptions are required to reach this goal \Crefpoint{subsubsec:goal_MOR}. Subsequently, we define the reduction map, which is needed to define the reduced-order model \Crefpoint{subsubsec:reduction_map_ROM}. We conclude the general framework with a workflow for \MOR on manifolds \Crefpoint{subsubsec:MORworkflow}.

\subsubsection{Full-order model} \label{subsec:gen_FOM}
In the scope of the present work, we consider high-dimensional parametric initial value problems. More precisely, assume that we are given a time interval $\mnfTime \vcentcolon= (\pointTimeInit,\pointTimeEnd)$ with initial time $\pointTimeInit$ and final time $\pointTimeEnd>\pointTimeInit$, a parameter set $\paramSet \subseteq \R^{p}$, an $\dimMnf$-dimensional smooth manifold $\mnf$ with large $\dimMnf$,
a (possibly parametric) smooth vector field $\vf\colon\paramSet\to\smoothVfs$, and a (possibly parametric) initial value $\pointInit\colon\paramSet\to\mnf$. We consider for $\param\in\paramSet$ the initial value problem
\begin{equation}
	\label{eq:FOM}
	\left\{\quad\begin{aligned}
		\evalField[\pointTime; \param]{\ddt \curve}
		&= \evalField[{
		\evalFun[\pointTime; \param]{\curve}
	}]{\vf(\param)} \in \TpMnf[{{\evalFun[\pointTime; \param]{\curve}}}], & t\in\mnfTime\\
		\evalFun[\pointTimeInit; \param]{\curve}
		&= \pointInit(\param) \in \mnf,
		\end{aligned}\right.
\end{equation}
which we want to solve for the integral curve $\evalFun[\cdot; \param]{\curve} \in \calC^{\infty}(\mnfTime, \mnf)$. We refer to~\eqref{eq:FOM} as the \FOM and to $\evalFun[\param]{\vf}$ as the \emph{\FOM vector field}. 

\begin{remark}[Parameter dependency]
	In the following, we may suppress the explicit notation of the parameter dependence for the sake of brevity. This is possible since the parameter is fixed for each \FOM evaluation. We indicate the parameter dependence only if it is relevant in a specific context.
\end{remark}

\subsubsection{Goal of model order reduction} \label{subsubsec:goal_MOR}
The goal of \MOR can be formulated as to be able to well-approximate the \emph{set of all solutions}
\begin{align}\label{eq:solution_manifold}
	\solMnf \vcentcolon=
\left\{
\evalFun[\pointTime; \param]{\curve} \in \mnf
\mid
(\pointTime, \param) \in \mnfTime \times \paramSet
\right\}
\subseteq \mnf
\end{align}
computationally efficiently.
Sometimes, the set of all solutions is referred to as the \emph{solution manifold}. However, this set does not necessarily have the structure of a manifold.
For example, \cite[Ex.~2.9]{Haasdonk2017} describes a case where the solution might be arbitrarily complex in the parameter (including discontinuous behavior).
The crucial assumption for \MOR to be reasonable is the following.

\begin{assumption}
	\label{ass:mor}
	Given a metric $\metricMnfSymbol\colon \mnf \times \mnf \to \Rpos$, we assume that there exists a low-dimensional embedded submanifold $\mnfSub \subseteq \mnf$ defined by an $\dimMnfRed$-dimensional manifold $\mnfRed$ and a smooth embedding $\emb \in \calC^{\infty}(\mnfRed, \mnf)$ with $\dim(\mnfRed) = \dimMnfRed \ll \dimMnf = \dim(\mnf)$ such that the set of solutions $\solMnf$ can be approximated well, i.e., 
	\begin{displaymath}
		\metricMnfSymbol(\emb(\mnfRed),\solMnf) \vcentcolon= \sup_{\point\in\solMnf} \inf_{\pointRed\in\mnfRed} \metricMnfSymbol(\point,\emb(\pointRed))
	\end{displaymath}		
	is small.
\end{assumption}

We refer to $\mnf$ as the \emph{full(-order) manifold} and to $\mnfRed$ as the \emph{reduced(-order) manifold}. Let us emphasize that the goal is to approximate the set $\solMnf \subseteq \mnf$, not the full manifold $\mnf$. We refer to \Cref{fig:mnf_solMnf_mnfSub} for a schematic illustration of the relation between the full manifold $\mnf$, the set of all solutions $\solMnf$, and the approximating embedded submanifold $\mnfSub$.

\begin{figure}[ht]
	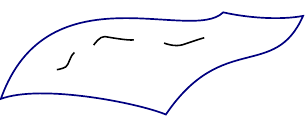
	\caption{Schematic illustration of the full manifold $\mnf$ (dark blue), the set of all solutions $\solMnf$ (black), and the approximating embedded submanifold $\mnfSub$ (red).
	The set of solutions is schematically depicted as three separate trajectories that may occur due to a possible discontinuous behavior in the parameter $\param$.}
	\label{fig:mnf_solMnf_mnfSub}
\end{figure}

%-----------------------------------------------------------------------------%
\subsubsection{Reduction map and reduced-order model} \label{subsubsec:reduction_map_ROM}
Assume that we have identified an $\dimMnfRed$-dimensional embedded submanifold $\mnfSub\subseteq \mnf$ with smooth embedding $\emb \in \calC^{\infty}(\mnfRed, \mnf)$ and that \Cref{ass:mor} is satisfied. To find a \ROM, we want to replace $\evalFun[\pointTime]{\curve}$ in~\eqref{eq:FOM} with the approximation~$
\evalFun[{\evalFun[\pointTime]{\curveRed}}]{\emb}$ based on a reduced integral curve $\curveRed\in \calC^{\infty}(\mnfTime, \mnfRed)$. Note that even if we would have an exact reproduction, i.e., $\evalFun[\pointTime]{\curve} = \evalFun[{\evalFun[\pointTime]{\curveRed}}]{\emb}$ for all $\pointTime\in\mnfTime$, the initial value problem~\eqref{eq:FOM} in the reduced integral curve $\curveRed$ would be overdetermined, in the sense that we have (locally in each chart)~$\dimMnf$ equations for~$\dimMnfRed$ unknowns. Thus, we must also reduce the initial value problem and give the following definition.

\begin{definition}[Reduction map]
	\label{def:reductionMapping}
	A map $\redMap \in \calC^{\infty}(\T\mnf, \T\mnfRed)$ is called \emph{reduction map} for a smooth embedding~$\emb \in \calC^{\infty}(\mnfRed, \mnf)$ if it satisfies the \emph{projection property}
	\begin{align}\label{eq:projProperty}
		\redMap \circ \diff{\emb} = \id_{\T\mnfRed}.
	\end{align}
	As in \Cref{subsec:tangent_bundle_vector_field}, we split the reduction map 
\begin{align*}
	\redMap \in \calC^{\infty}(\T\mnf, \T\mnfRed),\qquad (\point, \velocity) \mapsto \left(
            \evalFun[\point]{\redMapPoint},
            \evalFun[\velocity]{\redMapTan{\point}}
        \right)
\end{align*}
with $\redMapPoint \in \calC^{\infty}(\mnf, \mnfRed)$ and $\redMapTan{\point} \in \calC^{\infty} ( \TpMnf[\point], \TpMnfRed[\redMapPoint(m)])$ for $\point\in\mnf$.
We refer to $\redMapPoint$ as a \emph{point reduction} for $\emb$ and to $\redMapTan{\point}$ as a \emph{tangent reduction} for $\emb$. 
\end{definition}

Note that~\eqref{eq:projProperty} immediately implies that $\diff{\emb} \circ \redMap \in \calC^{\infty} ( \T\mnf, \T\mnfSub )$ is idempotent and thus a projection. Moreover,~\eqref{eq:projProperty} implies that a point reduction and a tangent reduction for $\emb$ satisfy
\begin{subequations}
	\label{eq:projProperty:split}
	\begin{align}
		\label{eq:projProperty:pointProjection}\redMapPoint \circ \emb &= \id_{\mnfRed},\\
		\label{eq:projProperty:tangentProjection}\redMapTan{\pointSub} \circ \evalField[\pointRed]{\diff{\emb}} &= \id_{\TpMnfRed} \qquad \text{for all } \pointRed \in \mnfRed,
	\end{align}
\end{subequations}
which we refer to as the \emph{point projection property} and the \emph{tangent projection property}, respectively. The relation between the embedding $\emb$ and the reduction map $\redMap$ is illustrated in \Cref{fig:redMapping}.

\begin{figure}[ht]
	\centering
	\begin{subfigure}{.85\linewidth}
		\centering
		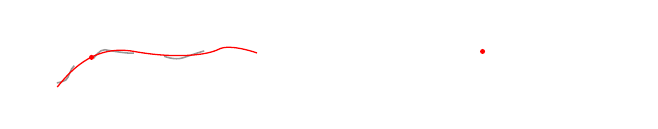
		\caption{Schematic illustration of the relation of the embedding $\emb$ and the point reduction $\redMapPoint$.}
		\label{fig:red_emb}
	\end{subfigure}\\[.3em]
	\begin{subfigure}{.85\linewidth}
		\centering
		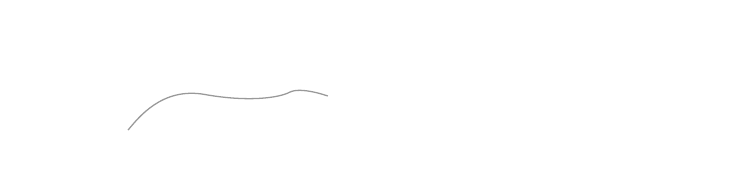
		\caption{Schematic illustration of the relation of the tangent spaces involved in \MOR on manifolds. The reduced tangent space $\TpMnfRed$ is displayed orthogonally to $\mnfRed$ for a better visualization.}
		\label{fig:submnf_tan_space}
	\end{subfigure}\\[-.7em]
	\caption{Schematic illustration of the relation between the embedding $\emb$ and the reduction map $\evalFun[\point, \velocity]{\redMap}= \left(\evalFun[\point]{\redMapPoint},\evalFun[\velocity]{\redMapTan{\point}}\right)$ with $\point\in\mnf$.}
	\label{fig:redMapping}
\end{figure}

\begin{example}[Linear-subspace MOR]\label{rem:subspace_MOR}
	Projection-based linear-subspace \MOR with a reduced-basis matrix $\fV \in \R^{\dimMnf \times \dimMnfRed}$ and a projection matrix $\fW \in \R^{\dimMnf \times \dimMnfRed}$ is contained as a special case of the presented formulation with $\mnf = \chartdomain = \R^{\dimMnf}$, $\mnfRed = \chartdomainRed = \R^{\dimMnfRed}$, $\chartmap = \id_{\R^{\dimMnf}}$, $\chartmapRed = \id_{\R^{\dimMnfRed}}$ and
	\begin{align*}
		&\evalFun[\pointCoord]{\redCoord}
\vcentcolon= \rT\fW \pointCoord,&
		&\evalField[\pointCoord]{\redMapCoord}(\velocityCoord)
\vcentcolon=  \rT\fW \fv,
		&{\pointSubCoord} \vcentcolon= \fV \pointRedCoord.
	\end{align*}
	This exactly covers the case where $\embCoord$ and $\redCoord$ are linear.
	The projection property \eqref{eq:projProperty} then relates to the \emph{biorthogonality} of $\fW$ and $\fV$
	\begin{align*}
		\redCoord \circ \embCoord \equiv&\; \id_{\R^{\dimMnfRed}}&
		&\iff&
&\rT\fW \fV = \tname{\fI}{\dimMnfRed} \in \R^{\dimMnfRed \times \dimMnfRed},\\
		\evalField[{{\evalFun[\pointRedCoord]{\embCoord}}}]{\redMapCoord}
\circ \evalField[\pointRedCoord]{\diff{\embCoord}} \equiv&\; \id_{\R^{\dimMnfRed}}&
&\iff&
&\rT\fW \fV = \tname{\fI}{\dimMnfRed} \in \R^{\dimMnfRed \times \dimMnfRed},
	\end{align*}
	which is often assumed in linear-subspace \MOR.
\end{example}

\begin{definition}[Reduced-order model]
	Consider the \FOM~\eqref{eq:FOM}, a smooth embedding $\emb \in \calC^{\infty}(\mnfRed, \mnf)$, and a reduction map $\redMap \in \calC^{\infty}(\T\mnf, \T\mnfRed)$ for $\emb$ with point and tangent reduction for $\emb$ given by $\evalFun[\point, \velocity]{\redMap}= \left(\evalFun[\point]{\redMapPoint},\evalFun[\velocity]{\redMapTan{\point}}\right)$.
	We define the \emph{\ROM vector field} as
	$\vfRed\colon \paramSet \to \smoothVfsRed$ via
	\begin{align*}
		\evalField[\pointRed]{\vfRed(\param)}
\vcentcolon= \evalFun[{\evalField[\pointSub]{\vf(\param)}}]{\redMapTan{\pointSub}}\in \TpMnfRed.
	\end{align*}
	Then, for $\param\in\paramSet$, we call the initial value problem on $\mnfRed$
	\begin{equation}
		\label{eq:rom}
		\left\{\quad\begin{aligned}
			\evalField[\pointTime; \param]{\ddt \curveRed} 
			&= \evalField[{
				\evalFun[\pointTime; \param]{\curveRed}
			}]{\vfRed(\param)} \in \TpMnfRed[{{\evalFun[\pointTime; \param]{\curveRed}}}]\\
		\evalFun[\pointTimeInit; \param]{\curveRed}
		&= \reduce{\curve}_0(\param) \vcentcolon= \evalFun[{\evalFun[\param]{\pointInit}}]{\redMapPoint} \in \mnfRed
		\end{aligned}\right.
	\end{equation}
	the \emph{\ROM} for~\eqref{eq:FOM} under the reduction map $\redMap$ with solution $\curveRed(\cdot; \param) \in \calC^\infty (\mnfTime, \mnfRed)$.
\end{definition}

We emphasize that both, the point and the tangent reduction, are relevant for the \ROM, since the point reduction is used to map the initial value $\pointInit$, while the tangent reduction maps the \FOM vector field to the tangent space of the reduced manifold $\mnfRed$.
Moreover, we see that it is not sufficient to define $\redMapPoint$ and $\smash{\redMapTan{\pointSub}}$ only in the image of $\emb$ and $\smash{\evalField[\pointRed]{\diff{\emb}}}$, respectively, since the initial value and the evaluated \FOM vector field may be elements of $\mnf\setminus\evalFun[\smash{\mnfRed}]{\emb}$ and~$\TpMnf[\pointSub] \setminus \TpMnfSub$, respectively.

%-----------------------------------------------------------------------------%
\subsubsection{MOR workflow}
\label{subsubsec:MORworkflow}

With \Cref{ass:mor} at hand, \MOR (in the scope of this work) can be summarized in three steps:
\begin{enumerate}
	\item \stepApproximation: Given the \FOM~\eqref{eq:FOM}, find a reduced manifold $\mnfRed$ and a smooth embedding $\emb \in \calC^{\infty} (\mnfRed, \mnf)$ such that
$\metricMnfSymbol(\emb(\mnfRed),\solMnf)$ is small.
	\item \stepReduction: Identify a reduction map $\redMap \in \calC^{\infty}(\T\mnf, \T\mnfRed)$ for $\emb$ and construct the \ROM~\eqref{eq:rom}.
	\item \stepReconstruction: Solve the \ROM~\eqref{eq:rom} for ${\curveRed}$ and approximate the \FOM solution curve $\curve$ with
		\begin{equation}
			\label{eqn:solutionCurveApproximation}
			\evalFun[\pointTime; \param]{\curve}
\approx
\evalFun[{
	\evalFun[\pointTime; \param]{\curveRed}
}]{\emb}
\qquad
\text{for } (\pointTime, \param) \in \mnfTime \times \paramSet.
		\end{equation}
\end{enumerate}

In the remainder of the manuscript, we discuss all three steps, starting with the \stepReconstruction~step in the subsequent subsection. Possible constructions of the reduction map in the \stepReduction~step are discussed in \Cref{subsec:MPG,sec:structurePreservation}. The construction of the embedding $\emb$ in the \stepApproximation~step is analyzed in a data-driven framework in \Cref{sec:embedding_generation}.

\subsection{Exact reproduction}\label{subsec:exact_reproduction}

A desirable property in the \stepReconstruction~step is to answer the question when the approximation in~\eqref{eqn:solutionCurveApproximation} is exact, which we refer to as \emph{exact reproduction}. 
Clearly, if for a given parameter $\param\in\paramSet$, the \FOM solution $\curve$ evolves on $\emb(\mnfRed)$, i.e., $\evalFun[t; \param]{\curve}\in\emb(\mnfRed)$ for all $\pointTime \in \mnfTime$, then we can define the smooth curve
\begin{equation}
	\label{eqn:curveRedAlt}
	\curveRedAlt \vcentcolon= \emb^{-1}(\evalFun[\cdot;\param]{\curve}) \in \calC^\infty (\mnfTime,\mnfRed),
\end{equation}
since, by assumption, $\emb$ is a diffeomorphism onto its image. With this choice, we immediately obtain
\begin{align}\label{eq:exact_reproduction_tangent}
		\evalField[{
			\evalFun[{
				\evalFun[\pointTime; \param]{\curveRedAlt}
			}]{\emb}
		}]{\vf(\param)}
		= \evalField[{
	\evalFun[\pointTime; \param]{\curve}
}]{\vf(\param)}
		= \evalField[\pointTime; \param]{\ddt \curve}
		= \evalField[\pointTime; \param]{\ddt (\emb \circ \beta)}
		= \evalFun[{
	\evalField[\pointTime; \param]{\ddt \curveRedAlt}
}]{
	\evalField[{
		\evalFun[\pointTime; \param]{\curveRedAlt}
	}]{\diff\emb}
},
\end{align}
where the last equality follows from the chain rule~\eqref{eq:chain_rule}. 
It remains to prove that the \ROM~\eqref{eq:rom} is able to recover the reduced curve $\curveRedAlt$, which we show in the following.

\begin{theorem}[Exact reproduction of a solution]\label{thm:exact_repro}
	Assume that the \FOM~\eqref{eq:FOM} is uniquely solvable 
	and consider a reduction map $\redMap \in \calC^{\infty}(\T\mnf, \T\mnfRed)$ for the smooth embedding $\emb \in \calC^{\infty} (\mnfRed, \mnf)$ and a parameter $\param\in\paramSet$.
	Assume that the \ROM~\eqref{eq:rom} is uniquely solvable and $\evalFun[\pointTime;\param]{\curve}\in\emb(\mnfRed)$ for all $\pointTime\in\mnfTime$. Then the \ROM solution $\evalFun[\cdot; \param]{\curveRed}$ exactly recovers the solution $\evalFun[\cdot; \param]{\curve}$ of the \FOM~\eqref{eq:FOM} for this parameter, i.e.,
	\begin{align}\label{eq:exact_repro}
		&\evalFun[{\evalFun[\pointTime; \param]{\curveRed}}]{\emb}
=  \evalFun[\pointTime; \param]{\curve}&
& \text{for all } \pointTime \in \mnfTime.
	\end{align}
\end{theorem}

\begin{proof}
	Since $\evalFun[\pointTime;\param]{\curve}\in\emb(\mnfRed)$ for all $\pointTime\in\mnfTime$, we can construct $\curveRedAlt$ as in~\eqref{eqn:curveRedAlt}. It remains to show that $\curveRedAlt$ satisfies the \ROM~\eqref{eq:rom}. First, we obtain
	\begin{align*}
		 \reduce{\curve}_0(\param)
= \evalFun[\pointInit(\param)]{\redMapPoint}
= \evalFun[{\evalFun[\pointTimeInit; \param]{\curve}}]{\redMapPoint}
= \evalFun[{
	\evalFun[\pointTimeInit; \param]{\curveRedAlt}
}]{(\redMapPoint \circ \emb)}
= \evalFun[\pointTimeInit; \param]{\curveRedAlt},
	\end{align*}
	where the last equality is due to the projection property~\eqref{eq:projProperty:split} for the point reduction.
	Second, $\curveRedAlt$ suffices the initial value problem of the \ROM since the tangent projection property~\eqref{eq:projProperty:tangentProjection} implies with \eqref{eq:exact_reproduction_tangent}
	\begin{align*}
		\evalFun[{
	\evalField[{
		\evalFun[{
			\evalFun[\pointTime; \param]{\curveRedAlt}
		}]{\emb}
	}]{\vf(\param)}
		}]{\redMapTan{{
			{\evalFun[{
				\evalFun[\pointTime; \param]{\curveRedAlt}
			}]{\emb}}
		}}}
		= \evalFun[{
	\evalField[\pointTime; \param]{\ddt \curveRedAlt}
}]{\Big(\redMapTan{{\evalFun[{
			\evalFun[\pointTime; \param]{\curveRedAlt}
		}]{\emb}}}
	\circ
	\evalField[{
		\evalFun[\pointTime; \param]{\curveRedAlt}
	}]{\diff{\emb}}
\Big)}
	= \evalField[\pointTime; \param]{\ddt \curveRedAlt}.\tag*{\qedhere}
	\end{align*}
\end{proof}

In the following, we give an example of for which the exact reproduction can be achieved for a specific choice of $\mnfRed$ and $\emb$.

\begin{corollary}[Canonical form]\label{cor:exact_repro_emb}
	For a given \FOM~\eqref{eq:FOM} on $\mnf$,
	assume that $\mnfRed = \mnfTime \times \paramSet$ is an $(\paramDim + 1)$-dimensional smooth manifold,
	that the \FOM is uniquely solvable,
	that the \FOM solution $\curve\colon \mnfTime \times \paramSet =\vcentcolon \mnfRed \to \mnf$ is a smooth embedding,
	and that there exists a reduction map $\redMap \in \calC^{\infty}(\T\mnf, \T\mnfRed)$ for the smooth embedding $\emb\equiv\curve$.
	Then, the \ROM~\eqref{eq:rom} reproduces the \FOM solution exactly with the reduced integral curve $\evalFun[\pointTime; \param]{\curveRed} = (\pointTime, \param)$ such that the flow of the \ROM is $\flowRedAt[\pointTimeAlt](\pointTime, \param) = (\pointTime + \pointTimeAlt, \param)$.
	Moreover, the \ROM in bold notation reads
	\begin{align*}
		\evalField[\pointTime; \param]{\ddt \curveRedCoord}	&= \fe_1 \in \R^{\paramDim + 1}, &
		 \curveRedCoord_0(\param) &= (\pointTimeInit, \param),
	\end{align*}
	where $\fe_1 \in \R^{\paramDim + 1}$ denotes the first unit vector.
\end{corollary}

\begin{proof}
	With the assumptions of \Cref{cor:exact_repro_emb}, the choice $\emb \equiv \curve$ guarantees that the assumptions of \Cref{thm:exact_repro} are fulfilled and $\evalFun[\pointTime; \param]{\curveRedAlt} = (\pointTime, \param)$ is a valid choice for the curve in~\eqref{eqn:curveRedAlt}, which was used in the proof of \Cref{thm:exact_repro} as the \ROM solution candidate. For the remaining statement, we observe
	\begin{align*}
		\evalField[\pointTime; \param]{\ddt \tidx{\curveRedAlt}{1}{}} &= 1, &
		\evalField[\pointTime; \param]{\ddt \tidx{\curveRedAlt}{i}{}} &= 0 \qquad \text{for } 1 < i \leq \paramDim + 1,
	\end{align*}
	and $ \reduce{\curve}_0(\param) = \evalFun[\pointTimeInit; \param]{\curveRedAlt} = (\pointTimeInit, \param)$, which completes the proof.
\end{proof}

\begin{example}
	A particular example describing the situation from \Cref{cor:exact_repro_emb} is given by the linear advection equation with constant coefficients and periodic boundary conditions, see for instance \cite[Ex.~5.12]{BlaSU20}.
\end{example}

%-----------------------------------------------------------------------------%
\subsection{Manifold Petrov--Galerkin (MPG)}
\label{subsec:MPG}

Now we want to address one example of how to construct a reduction map \Crefpoint{def:reductionMapping}, i.e., how to do the \stepReduction~step from the general \MOR workflow described in \Cref{subsubsec:MORworkflow}. Note that this specific choice of reduction map has been independently developed in \cite{OttMR23}.
Moreover, we emphasize that the specific choice of the reduction map is crucial for the approximation quality of the \ROM; see, for instance, \cite{OttPR22}. Nevertheless, our goal here is not to present an optimal choice but rather an example of leveraging the smooth embedding $\emb$ to construct a reduction map using the previously introduced framework in \Cref{sec:gen_frame}.

Assume that we have completed the \stepApproximation~step from the general \MOR workflow, i.e., we have already identified a reduced manifold $\mnfRed$ together with a smooth embedding~$\emb$. Since $\emb$ is a homeomorphism onto its image, we know that $\inv{\emb}\colon \mnfSub \to \mnfRed$ exists. Under for \MOR reasonable assumptions, the extension lemma for smooth functions (see for instance \cite[Lem.~2.26]{Lee12}) guarantees that we can find a smooth extension $\redMapPoint$ of $\inv{\emb}$, which by construction satisfies the point projection property~\eqref{eq:projProperty:pointProjection}. 
We refer to \Cref{fig:red_emb} for an illustration of the relation between~$\inv{\emb}$ and~$\redMapPoint$. Differentiating the point projection property \eqref{eq:projProperty:pointProjection} with the chain rule \eqref{eq:chain_rule} implies
\begin{align}\label{eq:red_emb_duality}
	\evalField[{\evalFun[\pointRed]{\emb}}]{\diff{\redMapPoint}}\circ \evalField[\pointRed]{\diff{\emb}}
	= \evalField[\pointRed]{\diff{(\id_{\mnfRed})}}
	= \id_{\TpMnfRed}\colon \TpMnfRed \to \TpMnfRed,
\end{align}
i.e., $\evalField[{\evalFun[\pointRed]{\emb}}]{\diff{\redMapPoint}}$ is a left-inverse to $\evalField[\pointRed]{\diff{\emb}}$. In particular, we have proven the following duality result.

\begin{theorem}[MPG reduction map]
	\label{thm:red_emb_duality}
	Consider a smooth embedding $\emb$ and a point reduction~$\redMapPoint$ for $\emb$. Then, the differential of the point reduction $\redMapPoint$ is a left inverse to the differential of the embedding $\emb$.
	Consequently,
	\begin{equation}
		\label{eq:mpg_projection} 
		\redMapMPG \colon \T\mnf\to \T\mnfRed \qquad (\point, \velocity) \mapsto \left(\evalFun[\point]{\redMapPoint},\evalFun[\velocity]{\evalField[\point]{\diff{\redMapPoint}}}\right)
	\end{equation}
	is a smooth reduction map for $\emb$, which we call the \emph{\MPG reduction map} for $(\redMapPoint,\emb)$.
\end{theorem}

We refer to the \ROM~\eqref{eq:rom} obtained with the \MPG reduction map from~\Cref{thm:red_emb_duality} as the \emph{\MPG-\ROM} for $(\redMapPoint,\emb)$. In index and bold notation, the tangent projection property~\eqref{eq:red_emb_duality} reads
	\begin{align}\label{eq:red_emb_duality_coord}
		&\evalField[{\evalFun[\pointRed]{\emb}}]{\fracdiff{
			\tidx{\redMapPoint}{i}{}
		}{
			\tidx{\chartmap}{k}{}
		}}
		\evalField[\pointRed]{\fracdiff{
			\tidx{\emb}{k}{}
		}{
			\tidx{\chartmapRed}{j}{}
		}}
		= \tidx{\delta}{i}{j},&
		& \evalField[{\pointSubCoord}]{\D\redCoord}
		\evalField[\pointRedCoord]{\D\embCoord} = \tname{\fI}{\dimMnfRed} \in \R^{\dimMnfRed \times \dimMnfRed}.
	\end{align}
	It can be interpreted as that the columns of $\evalField[\pointRedCoord]{\D\embCoord}$ span an $\dimMnfRed$-dimensional reduced vector space that changes with the reduced coordinates $\pointRedCoord \in \R^\dimMnfRed$,
	whereas the rows of $\evalField[{\pointSubCoord}]{\D\redCoord}$ span an $\dimMnfRed$-dimensional vector space dual to the reduced vector space.

\begin{example}[Linear-subspace MOR]
	\label{rem:lin_mor_mpg}
	If $\emb$ and $\redMapPoint$ are linear as in \Cref{rem:subspace_MOR}, then the \MPG-\ROM~\eqref{eq:rom} with the \MPG reduction map from \Cref{thm:red_emb_duality} is the \ROM obtained in classical linear-subspace \MOR via Petrov--Galerkin projection
	\begin{align*}
		&\evalField[\pointSubCoord]{\mpgReductionCoordSymbol}
= \evalField[\pointSubCoord]{\fD \redCoord}
= \rT\fW,&
		&\evalField[\pointTime]{\ddt \curveRedCoord}
= \rT\fW \evalField[{
	\evalFun[\pointTime]{\curveRedCoord}
}]{\vfCoord},
	\end{align*}
	which is the motivation for the terminology \MPG.
\end{example}

%------------------------------------------%
% SECTION 4: MANIFOLDS WITH STRUCTURE %
%------------------------------------------%
\section{Manifolds with structure}
\label{sec:DifGeoPartTwo}

As a next step, we want to discuss structure-preserving \MOR on manifolds \Crefpoint{sec:structurePreservation}.
Beforehand, we specify the relevant structures on the \FOM level in the present section.
The idea is to equip the underlying full manifold $\mnf$ with additional structure to formulate a \FOM vector field $\vf$, which guarantees physical properties, e.g., that the \FOM solutions preserve energy over time.
We introduce additional structure on~$\mnf$ \Crefpoint{sec:structure_diff_geo}, which allows us to formulate Lagrangian systems \Crefpoint{subsec:lagrange} and Hamiltonian systems~\Crefpoint{subsec:hamiltonian} on manifolds.
Both systems admit a \FOM vector field, which guarantees that the \FOM solutions preserve the corresponding energy over time.

\subsection{Additional structure on \texorpdfstring{$\mnf$}{the manifold}} \label{sec:structure_diff_geo}

To keep this work self-contained, we proceed by detailing more concepts of differential geometry. We discuss
the cotangent space and covectors \Crefpoint{subsec:cotangent_space}, tensors \Crefpoint{subsec:tensors}, tensor fields \Crefpoint{subsec:tensor_fields}, structured tensor fields \Crefpoint{subsec:StructuredTensorFields}, and pullbacks of covectors, tensor fields, and functions \Crefpoint{subsec:pullback}. 

\subsubsection{Cotangent space, covectors, and cotangent bundle}
\label{subsec:cotangent_space}
The dual of the tangent space at~$\point \in \mnf$ \eqref{eq:tangent_space} is the \emph{cotangent space at $\point \in \mnf$}
\begin{align*}
	\coTp\mnf \vcentcolon= \{ \covec \mid \covec\colon \Tp\mnf \to \R \text{ linear}\},
\end{align*}
which is again an $\dimMnf$-dimensional vector space.
Elements in the cotangent space are called \emph{cotangent vectors} or simply \emph{covectors}.
Covectors can be constructed from scalar-valued functions:
For each scalar-valued function $f \in \calC^{\infty}(\mnf, \R)$,
its differential at $m$, $\evalField{\diff{f}} \in \calC^{\infty}(\TpMnf, \Tp[\evalFun{f}]{ \R})$, defines a linear functional on $\TpMnf$ if we identify $\Tp[\evalFun{f}]{ \R}$ with~$\R$.
Thus, the differential at $m$ of a scalar-valued function is a covector $\evalField{\diff{f}} \in \coTpMnf$.
For a given chart $\chart$ of $\mnf$,
this construction can be used to define a basis of $\coTpMnf$:
For each $i \in \{1, \dots, \dimMnf\}$,
the $i$-th component function of the chart mapping $\tidx{\chartmap}{i}{} \in \calC^{\infty}(\chartdomain, \R)$ is a scalar-valued function and thus $\evalField{\diff{\tidx{\chartmap}{i}{}}} \in \coTpMnf$.
Moreover, with \eqref{eq:diff_chart_map}, \eqref{eq:diff}, and identifying $\Tp[{{\evalFun[\point]{\chartmap}}}]{\R}
\cong \R$, it holds for all basis vectors of the tangent space $\basisTanAt{j} \in \TpMnf$, $1 \leq j \leq \dimMnf$ the dual relationship
\begin{align*}
	\evalFun[
		\basisTanAt{j}
	]{
		\evalField{\diff{\tidx{\chartmap}{i}{}}}
	}
	=
	\evalField{
		\fracdiff{
			\tidx{\chartmap}{i}{}
		}{
			\tidx{\chartmap}{j}{}
		}
	}
	= \tidx{\delta}{i}{j}
	\in \R \cong \Tp[{{\evalFun[\point]{\chartmap}}}]{\R}
\end{align*}
The differentials $\{ \evalField{\diff{\tidx{\chartmap}{i}{}}} \}_{1 \leq i \leq \dimMnf}$ define a basis of $\coTpMnf$ and we can represent each covector $\covec \in \coTp\mnf$ as
\begin{align*}
	&\covec = \tidx{\covec}{}{i}\; \evalField{\basisCoTan{i}} \in \coTp\mnf,
\end{align*}
with \emph{(covector) components} $\tidx{\covec}{}{i} \in \R$,
where the right-hand side sums over $1 \leq i \leq \dimMnf$ by Einstein summation convention \eqref{eq:einstein_summation_convention}.
By the duality of the bases of $\Tp\mnf$ and $\coTp\mnf$,
it holds for each covector $\covec \in \coTpMnf$ and vector $\velocity \in \TpMnf$ that
\begin{align*}
	\evalFun[\velocity]{\covec}
= \evalFun[
	\tidx{\velocity}{i}{}
	\basisTanAt{i}
]{
	\left(
		\tidx{\covec}{}{j}\, \diff{\tidx{\chartmap}{j}{}}
	\right)
}
=\tidx{\covec}{}{j}\, 
\tidx{\velocity}{i}{}\,
\evalFun[\basisTanAt{i}]{\basisCoTanAt{j}}
= \tidx{\covec}{}{i}\,
\tidx{\velocity}{i}{}\,
\in \R.
\end{align*}
Analogously to the tangent bundle \eqref{eq:tangent_bundle}, a \emph{cotangent bundle $\coT\calM$} can be formulated as the disjoint union of $\coTpMnf$, which can be shown to be a smooth manifold of dimension $2\dimMnf$.

\subsubsection{Tensors}
\label{subsec:tensors}
A generalization of vectors and covectors are the so-called \emph{tensors}.
For a vector space $\vectorspace$ and its dual $\vectorspace^\ast$, 
the \emph{space of $(r,s)$-tensors} given by
\begin{align*}
	\Trs\vectorspace \vcentcolon=
\underbrace{\vectorspace \otimes \cdots \otimes \vectorspace}_{r\text{ times}} \otimes
\underbrace{\vectorspace^\ast \otimes \cdots \otimes \vectorspace^\ast}_{s\text{ times}}.
\end{align*}
In the present work we consider tensors on the tangent and cotangent space, i.e., $\vectorspace = \TpMnf$ and $\vectorspace^\ast = \coTpMnf$.
Special cases are $\Trs[1,0]{\Tp\mnf} = \Tp\mnf$ and $\Trs[0,1]{\Tp\mnf} = \coTp\mnf$.
An element $\tensor \in \Trs{\Tp\mnf}$ of a general $(r,s)$-tensor space is called an \emph{$r$-times contravariant $s$-times covariant tensor}. This element can be represented by
\begin{align*}
	\tensor = \tidx{\tensor}{\tname{i}{1} \dots \tname{i}{r}}{\tname{j}{1} \dots \tname{j}{s}}\;\;
\evalField{\basisTan{\tname{i}{1}}} \otimes \cdots \otimes \evalField{\basisTan{\tname{i}{r}}}
\otimes
\evalField{\basisCoTan{\tname{j}{1}}} \otimes \cdots \otimes \evalField{\basisCoTan{\tname{j}{s}}}
\end{align*}
with \emph{components} $\tidx{\tensor}{\tname{i}{1} \dots \tname{i}{r}}{\tname{j}{1} \dots \tname{j}{s}} \in \R$ for $1 \leq \tname{i}{1}, \dots, \tname{i}{r}, \tname{j}{1}, \dots, \tname{j}{s} \leq \dimMnf$, where the right-hand side sums over each index $1\leq \tname{i}{1}, \dots, \tname{i}{r}, \tname{j}{1}, \dots, \tname{j}{s} \leq \dimMnf$ by the Einstein summation convention~\eqref{eq:einstein_summation_convention}.
The position of the index (upper or lower index) indicates which type (co- or contravariant) the respective index belongs to.
To extend the bold notation from \Cref{subsec:bold_notation} for tensors,
we stack the components with
\begin{align*}
	\tensorCoord
\vcentcolon=
\stack{
	\tidx{\tensor}
		{\tname{i}{1} \dots \tname{i}{r}}
		{\tname{j}{1} \dots \tname{j}{s}}
}{
	1\leq \tname{i}{1}, \dots, \tname{i}{r},
	\tname{j}{1}, \dots, \tname{j}{s} \leq \dimMnf
}
\in \R^{
	\underbrace{\scriptstyle{\dimMnf \times \dimMnf \times \cdots \times \dimMnf}}_{{r + s} \text{ times}}
}.
\end{align*}

\subsubsection{Tensor field and bundle of $(r,s)$-tensors}
\label{subsec:tensor_fields}
A so-called \emph{tensor field} is a mapping which assigns each point $\point \in \mnf$ a tensor in the corresponding $(r,s)$-tensor space $\Trs{\Tp\mnf}$ analogous to the smooth vector field introduced in \Cref{subsec:tangent_bundle_vector_field}.
To this end we define the \emph{bundle of $(r,s)$-tensors} as the disjoint union of all $(r,s)$-tensor spaces
\begin{align*}
	\Trs{\T\mnf} \vcentcolon= \disjointUnion_{\point \in \mnf} \Trs{\Tp\mnf}
	= \{ (\point, \tensor) \;\big|\; \point \in \mnf, \; \tensor \in \Trs{\Tp\mnf} \}.
\end{align*}
Similarly as before, we obtain the special cases $\Trs[1,0]{\T\mnf} = \T\mnf$ and $\Trs[0,1]{\T\mnf} = \coT\mnf$.
An \emph{$(r,s)$-tensor field} is defined as a map
\begin{align*}
	&\tf\colon \mnf \to \Trs{\T\mnf}, \;\point \mapsto (\point, \evalField{\tf})&
	&\text{such that } \evalField{\tf} \in \Trs{\Tp\mnf}.
\end{align*}
For a given chart $\chart$ of $\mnf$, we denote the $((r+s) \cdot \dimMnf)$ functions $\smash
{\tidx{\tf}{\tname{i}{1} \dots \tname{i}{r}}{\tname{j}{1} \dots \tname{j}{s}}\colon \chartdomain \to \R}$,
$1\leq \tname{i}{1}, \dots, \tname{i}{r}, \tname{j}{1}, \dots, \tname{j}{s} \leq \dimMnf$, with 
$\evalFun{\tidx{\tf}{\tname{i}{1} \dots \tname{i}{r}}{\tname{j}{1} \dots \tname{j}{s}}}
\vcentcolon=
\tidx{\left( \evalField{\tf} \right)}{\tname{i}{1} \dots \tname{i}{r}}{\tname{j}{1} \dots \tname{j}{s}}$
as the \emph{component functions} to stress the dependence on the point $\point$.
To extend the bold notation from \Cref{subsec:bold_notation} for tensor fields,
we stack the component functions with
\begin{align*}
	\tfCoord
\vcentcolon= \stack{ \tidx{\tf}{\tname{i}{1} \dots \tname{i}{r}}{\tname{j}{1} \dots \tname{j}{s}} \circ \inv\chartmap }{1\leq \tname{i}{1}, \dots, \tname{i}{r}, \tname{j}{1}, \dots, \tname{j}{s} \leq \dimMnf}:
\R^{\dimMnf} \supseteq \evalFun[\chartdomain]{\chartmap} \to \R^{
	\underbrace{\scriptstyle{\dimMnf \times \dimMnf \times \cdots \times \dimMnf}}_{{r + s} \text{ times}}
}.
\end{align*}

An $(r,s)$-tensor field $\tf$ is called \emph{smooth} if all of its component functions are smooth, i.e., $
\tidx{
	\tf
}{
	\tname{i}{1} \dots \tname{i}{r}
}{
	\tname{j}{1} \dots \tname{j}{s}
} \in \calC^{\infty}(\chartdomain, \R)$.
The set of all smooth $(r,s)$-tensor fields is the so-called \emph{smooth section of the $(r,s)$-tensor bundle} $\TrsFieldMnf$.
A special case are the smooth vector fields $\smoothVfs = \TrsFieldMnf[1,0]$.

\subsubsection{Structured tensor fields and musical isomorphisms}
\label{subsec:StructuredTensorFields}
Tensor fields may possess additional properties, which we refer to as \emph{structure}. In the following, we introduce two important examples of tensor fields with special structures, namely Riemannian metrics and symplectic forms.

For a smooth $(0,2)$-tensor field $\tf \in \TrsFieldMnf[0,2]$ with its component functions $\tidx{\tf}{}{ij} \in \calC^{\infty} (\chartdomain, \R)$ for $1 \leq i,j \leq \dimMnf$ in a given chart $\chart$,
the tensor field $\tf$ is called
\begin{itemize}
	\item \emph{symmetric}, if $\tidx{(\evalField{\tf})}{}{ij} = \tidx{(\evalField{\tf})}{}{ji}$ for each $\point \in \chartdomain$ and for all $1 \leq i,j \leq \dimMnf$;
	\item \emph{skew-symmetric} or \emph{$2$-form}, if $\tidx{(\evalField{\tf})}{}{ij} = -\tidx{(\evalField{\tf})}{}{ji}$ for each $\point \in \chartdomain$ and for all $1 \leq i,j \leq \dimMnf$;
	\item \emph{nondegenerate}, if $\stack{\smash{\tidx{(\evalField{\tf})}{}{ij}}}{1\leq i,j \leq \dimMnf} \in \R^{\dimMnf \times \dimMnf}$ is nondegenerate for each $\point \in \chartdomain$;\vspace*{.2em}
	\item \emph{positive definite}, if $\stack{\smash{\tidx{(\evalField{\tf})}{}{ij}}}{1\leq i,j \leq \dimMnf} \in \R^{\dimMnf \times \dimMnf}$ is positive definite for all $\point \in \chartdomain$;\vspace*{.2em}
	\item a \emph{closed $2$-form}, if $\tf$ is a $2$-form and for each $\point \in \chartdomain$
	\begin{align}\label{eq:closedness}
		\evalField{\fracdiff{ \tidx{\tf}{}{jk} }{ \tidx{\chartmap}{i}{} }}
+ \evalField{\fracdiff{ \tidx{\tf}{}{ki} }{ \tidx{\chartmap}{j}{} }}
+ \evalField{\fracdiff{ \tidx{\tf}{}{ij} }{ \tidx{\chartmap}{k}{} }}
= 0
\qquad \text{ for all } 1 \leq i \leq j \leq k \leq \dimMnf.
	\end{align}
\end{itemize}

Combining some of the previous properties, we obtain the following concepts. A smooth $(0,2)$-tensor field $\tf \in \TrsField{0,2}{\mnf}$ on $\mnf$ is called 
\begin{itemize}
	\item a \emph{Riemannian metric on $\mnf$} if $\tf$ is symmetric and positive definite;
	\item a \emph{symplectic form on $\mnf$} if $\tf$ is skew-symmetric, nondegenerate, and closed.
\end{itemize}
If $\tf, \symplForm \in \TrsFieldMnf[0,2]$ are a Riemannian metric and a symplectic form on $\mnf$, respectively, then we call $(\mnf,\tf)$ and $(\mnf,\symplForm)$ a \emph{Riemannian manifold} and \emph{symplectic manifold}, respectively. Note that the nondegeneracy of a symplectic form implies that a symplectic manifold has even dimension.

Both the Riemannian metric and the symplectic form are nondegenerate tensor fields.
This allows to formulate the inverse $(2,0)$-tensor field $\inv\tf \in \TrsFieldMnf[2,0]$ such that ${\tidx{(\inv{\evalField{\tf}})}{ik}{} \tidx{(\evalField{\tf})}{}{kj} = \tidx{\delta}{i}{j}}$,
where the components are typically denoted with \smash{$\tidx{({\evalField{\tf}})}{ik}{} \vcentcolon= \tidx{(\inv{\evalField{\tf}})}{ik}{}$} for the sake of brevity.
Moreover, the nondegeneracy allows to formulate an isomorphism between the tangent and the cotangent bundle.
Loosely speaking, this means that the indices in the index notation can be switched from covariant (superindices) to contravariant (subindices) and vice versa.
This is typically referred to as \emph{musical isomorphisms}
\begin{align}\label{eq:isoFlat}
	&\flat_{\tf} \in \calC^{\infty} (\T\mnf, \cotangentT\mnf),
\quad
\left(	
	\point,\,
	\tidx{\velocity}{i}{} \, \basisTanAt{i}
\right)
\mapsto
\left(
	\point,\,
	\tidx{(\evalField{\tf})}{}{ij} \; \tidx{\velocity}{j}{} \; \basisCoTanAt{i}
\right),\\
	\label{eq:isoSharp}
	&\sharp_{\tf} \in \calC^{\infty} (\cotangentT\mnf, \T\mnf),
\quad
\left(
	\point,\,
	\tidx{\covec}{}{i} \, \basisCoTanAt{i}
\right)
\mapsto
\left(
	\point,\,
	\tidx{(\evalField{\tf})}{ij}{} \; \tidx{\covec}{}{j} \; \basisTanAt{i}
\right).
\end{align}
Due to the nondegeneracy of $\tf$,
the two mappings are inverses of each other, i.e.,
\begin{align}\label{eq:iso_sharp_inv}
	\sharp_\tf \circ \flat_\tf
	\equiv \id_{\T\mnf}.
\end{align}
By a slight abuse of notation, we use the same symbols from \eqref{eq:isoFlat} and \eqref{eq:isoSharp} also to map between (co)tangent spaces $\flat_{\tf}\colon \TpMnf \to \coTpMnf$ and $\sharp_{\tf}\colon \coTpMnf \to \TpMnf$ (instead of the respective bundles).

%-----------------------------------------------------------------------------%
\subsubsection{Pullback of covectors, tensor fields, and functions}
\label{subsec:pullback}
Consider two smooth manifolds $\mnf$, $\mnfAlt$
and a smooth map $F \in \calC^{\infty}(\mnf, \mnfAlt)$.
Let $\chart$ and $\chartMnfAlt$ be charts of $\mnf$ and $\mnfAlt$ respectively
such that $\point \in \chartdomain$ and $\evalFun{F} \in \chartdomainMnfAlt$.
The differential \eqref{eq:diff} of $F$ can be used to define the \emph{pointwise pullback (of covectors) by $F$ at $\point$} via
\begin{align}
	\label{def:diff_pullback}
	\evalField{\dualdiff{F}} \in \calC^{\infty}\left(\coTpMnfAlt[{\evalFun{F}}],  \coTpMnf \right),
	\quad
	\tidx{\covec}{}{i}
	\;
	\evalField[{\evalFun{F}}]{\basisCoTanMnfAlt{i}}
	\mapsto
	\evalField{\fracdiff{\tidx{F}{i}{}}{\tidx{x}{j}{}}}
	\tidx{\covec}{}{i}
	\;
	\evalField{\basisCoTan{j}}.
\end{align}

For a smooth $(0,s)$-tensor field $\tf \in \TrsFieldMnfAlt[0,s]$,
the \emph{pullback of $\tf$ by $F$}, denoted by $\pullback{F}{\tf} \in \TrsFieldMnf[0,s]$, is a smooth tensor field (see \cite[Prop.~11.26]{Lee12})
with component functions%
\footnote{The pullbacks from \eqref{def:diff_pullback} and \eqref{eq:pullback} can be related in the case of smooth covector fields $\covf \in \TrsFieldMnfAlt[0,1]$, i.e., $s=1$, with $\evalField{\left(\pullback{F}{\covf}\right)} = \evalField{\dualdiff{F}}{\evalField[{\evalFun{F}}]{\covf}} \in \coTpMnf$.}
\begin{equation}\label{eq:pullback}
\begin{aligned}
	\tidx{\left(
		\evalField{\pullback{F}{\tf}}
	\right)}{}{\tname{j}{1} \dots \tname{j}{s}}
\vcentcolon= \tidx{(
	\evalField[{\evalFun{F}}]{\tf}
)}{}{\tname{\ell}{1}\dots \tname{\ell}{s}}
\cdot
\trafo{F}{\point}{\tname{\ell}{1}}{\chartmap}{\tname{j}{1}}
\cdots
\trafo{F}{\point}{\tname{\ell}{s}}{\chartmap}{\tname{j}{s}}.
\end{aligned}
\end{equation}
A scalar-valued smooth function $h \in \calC^{\infty} (\mnfAlt, \R)$ can be interpreted as a $(0,0)$-tensor field.
Then, as a special case of \eqref{eq:pullback}, the \emph{pullback of (a function) $h$ by $F$} is a smooth function $\pullback{F}h \in \calC^{\infty} (\mnf, \R)$ with
\begin{align}\label{eq:pullback_function}
	\evalFun{(\pullback{F}h)} = \evalFun[{\evalFun{F}}]{h} = \evalFun{(h \circ F)}.
\end{align}

By \Cref{subsec:cotangent_space},
the differential of a smooth scalar-valued function $G \in \calC^{\infty}(\mnfAlt, \R)$ defines a covector $\evalField[{\evalFun[\point]{F}}]{\diff{G}} \in \coTpMnfAlt$.
Then an analogue to the chain rule \eqref{eq:chain_rule} is
\begin{align}\label{eq:chain_rule_cotan}
	&\evalField[\point]{\diff{(\pullback{F} G)}}
	= \evalField[\point]{\dualdiff{F}}
\evalField[{\evalFun[\point]{F}}]{\diff{G}}
	\in \coTpMnf,
\end{align}
which uses the pullback of a function \eqref{eq:pullback_function} on the left-hand side and applies the pointwise pullback $\evalField[\point]{\dualdiff{F}} \in \calC^{\infty}(\coTpMnfAlt, \coTpMnf)$ to the covector $\evalField[{\evalFun[\point]{F}}]{\diff{G}} \in \coTpMnfAlt$ on the right-hand side of the equation.

%--------------------------------------------------------------------%
\subsection{Lagrangian systems}
\label{subsec:lagrange}
This subsection defines Lagrangian systems formulated on a manifold and additionally introduces further structure required for the \MOR part discussed in the forthcoming \Cref{subsec:MORLagrangian}.
Consider a $\dimMnfQ$-dimensional smooth manifold $\mnfQ$ with chart $\chartMnfQ$.
As mentioned in \Cref{subsec:tangent_bundle_vector_field}, the tangent bundle $\T\mnfQ$ is a $2\dimMnfQ$-dimensional smooth manifold
and the differential $\diff{\chartmapMnfQ} \in \calC^{\infty} (\chartdomainTanSpaceMnfQ, \R^{2\dimMnfQ})$ defines a natural chart \eqref{eq:chartmapTanSpace}.
We abbreviate this chart with $\chartmapTanSpaceMnfQ \vcentcolon= \diff{\chartmapMnfQ}$ for brevity.
By \eqref{eq:chartmapTanSpace}, it holds
\begin{align*}
	\chartmapTanSpaceMnfQ\colon
\chartdomainTanSpaceMnfQ \to \R^{2\dimMnfQ},\;
\left(
	\pointQ,
	\tidx{\velocity}{i}{}\; \evalField[\pointQ]{\basisMnfQTan{i}}
\right)
\mapsto \left(
	\evalFun[\pointQ]{\chartmapMnfQ},
	\stack{\tidx{\velocity}{i}{}}{1 \leq i \leq \dimMnfQ}
\right).
\end{align*}
It will be relevant to differentiate between the first $\dimMnfQ$ and the latter $\dimMnfQ$ entries of $\chartmapTanSpaceMnfQ$ for a point $\pointMnfQTanSpace = (\pointQ, \velocity)\in \T\mnfQ$, which will be denoted with
\begin{align*}
	&\evalFun[\pointMnfQTanSpace]{\tidx{\chartmapTanSpaceMnfQ}{i}{}}
= \evalFun[\pointQ]{\tidx{\chartmapMnfQ}{i}{}},&
	&\evalFun[\pointMnfQTanSpace]{\tidx{\chartmapTanSpaceMnfQ}{\dimMnfQ + i}{}}
= \tidx{\velocity}{i}{},&
	&\text{for } 1 \leq i \leq \dimMnfQ.
\end{align*}
To lift a smooth curve $\curveMnfQ \in \calC^{\infty}(\mnfTime, \mnfQ)$ to its tangent bundle, we define
\begin{align*}
\lift{\curveMnfQ} \in \calC^{\infty} (\mnfTime, \T\mnfQ),\;
\pointTime \mapsto \left(
	\evalFun[\pointTime]{\curveMnfQ},
	\evalField[\pointTime]{\ddt \curveMnfQ}
\right).
\end{align*}

We denote a \emph{Lagrangian system} as the tuple $(\mnfQ, \Lag)$ of a smooth manifold $\mnfQ$ and a smooth function $\Lag \in \calC^\infty(\T\mnfQ, \R)$,
which we refer to as the \emph{Lagrangian function}.
The associated second-order differential equation on the manifold is given by the \emph{Euler--Lagrange equation}
\begin{align}\label{eq:euler_lagrange}
	&\evalField[{{{\evalFun[\pointTime]{\lift{\curveMnfQ}}}}}]{
		\fracdiff{\Lag}
			{\chartmapMnfQPosi{i}}
	}
- \evalField[\pointTime]{
	\ddt \left(
	\evalField[
		{{\evalFun[\cdot]{\lift{{\curveMnfQ}}}}}
	]{
		\fracdiff{\Lag}{\tidx{\chartmapTanSpaceMnfQ}{\dimMnfQ+i}{}}}
	\right)
}
= 0&
	&\text{for } 1\leq i \leq \dimMnfQ,&
	&{\evalFun[\pointTimeInit]{\lift\curveMnfQ}}
= \begin{pmatrix}
	\pointQInit\\
	\velocityInit
\end{pmatrix},
\end{align}
with \emph{initial value} $(\pointQInit, \velocityInit) \in \T\mnfQ$, which has to be solved for $\curveMnfQ \in \calC^\infty(\mnfTime, \mnfQ)$.
In bold notation, the equation in coordinates reads for the Lagrangian $\LagCoord\vcentcolon= \Lag \circ \inv\chartmapTanSpaceMnfQ\colon \R^{2 \dimMnfQ} \supseteq \chartmapTanSpaceMnfQ(\chartdomainTanSpaceMnfQ) \to \R$, $(\pointQCoord, \velocityCoord) \mapsto \LagCoord(\pointQCoord, \velocityCoord)$
\begin{align*}
	\evalField[\left({{\evalFun[\pointTime]{\curveMnfQCoord}}}, \ddt {{\evalFun[\pointTime]{\curveMnfQCoord}}}\right)]{\D_{\pointQCoord} \LagCoord}
	- \ddt
\evalField[\pointTime]{
	\left(
		\evalField[
			\left(
				{{\evalFun[\cdot]{\curveMnfQCoord}}},
				\ddt {{\evalFun[\cdot]{\curveMnfQCoord}}}
			\right)
		]{\D_{\velocityCoord} \LagCoord}
	\right)
}
= \fzero
 \in \R^{\dimMnfQ},
\end{align*}
where $\D_{\pointQCoord}(\cdot)$ denotes the derivative with respect to the first $\dimMnfQ$ coordinates (named $\pointQCoord$ here) and $\D_{\velocityCoord}(\cdot)$ the derivative for the last $\dimMnfQ$ coordinates (named $\velocityCoord$ here).

Since the Euler--Lagrange equations are obtained from a variation of an action functional, it is well-known that the solution curve is guaranteed to conserve a scalar-valued function (see, e.g., \cite[Sec.~3.5]{Abraham1987} and \cite[Prop.~7.3.1]{MarR99}):

\begin{theorem}
	\label{thm:euler_lagrange_energy}
	The \emph{energy}
	\begin{align*}
		\energy\colon \T\mnfQ \to \R,\qquad
\pointMnfQTanSpace = \left( \pointQ, \tidx{\velocity}{i}{} \basisMnfQTanAt{i} \right)
\mapsto
\tidx{\velocity}{j}{}
\evalField[\pointMnfQTanSpace]{
	\fracdiff{\Lag}
		{\chartmapMnfQVeli{j}}
}
- \evalFun[\pointMnfQTanSpace]{\Lag}
	\end{align*}
	is conserved along the lift of the solution curve $\curveMnfQ$ of the Euler--Lagrange equations, i.e.,
	\begin{align*}
\evalField[\pointTime]{
	\ddt \evalFun[{
		\evalFun[\cdot]{\lift{\curveMnfQ}}
	}]{\energy}
} = 0
\qquad
\text{for all } \pointTime \in \mnfTime.
	\end{align*}
\end{theorem}

The Lagrangian is called \emph{regular} if the smooth $(0,2)$-tensor field defined by the second-order derivative of the Lagrangian w.r.t.\ the velocity 
\begin{align}\label{eq:def_metricV}
	\evalField[\pointMnfQTanSpace]{\metricV}
\vcentcolon= \evalField[\pointMnfQTanSpace]{
	\sfracdiff{\Lag}
		{\chartmapMnfQVeli{i}}
		{\chartmapMnfQVeli{j}\,}
}
\basisTMnfQCoTanVelAt{i}
\otimes
\basisTMnfQCoTanVelAt{j}
\end{align}
at each point $\pointMnfQTanSpace \in \T\mnfQ$ is nondegenerate \Crefpoint{subsec:StructuredTensorFields}.
In this case, we can formulate the \emph{Euler--Lagrangian vector field} $\vfLag \in \smoothVfs[\T\mnfQ]$ such that at a point $\pointMnfQTanSpace = \left( \pointQ, \tidx{\velocity}{i}{} \basisTanAt[\pointQ]{i} \right) \in \T\mnfQ$, it holds
\begin{align}\label{eq:def_vfLag}
	\evalField[\pointMnfQTanSpace]{\vfLag}
\vcentcolon= 
\tidx{\velocity}{i}{}
\basisTMnfQTanPosAt{i}
+
\tidx{{
	(\evalField[\pointMnfQTanSpace]{\metricV})
}}{(\dimMnfQ + i)(\dimMnfQ + j)}{}
\left(
	\evalField[\pointMnfQTanSpace]{
		\fracdiff{\Lag}
			{\chartmapMnfQPosi{j}}
	}
	-
	\evalField[\pointMnfQTanSpace]{
		\sfracdiff{\Lag}
			{\chartmapMnfQVeli{j}}
			{\chartmapMnfQPosi{k}\,}
	}
	\tidx{\velocity}{k}{}
\right)
\basisTMnfQTanVelAt{i},
\end{align}
where we use the convention from \Cref{subsec:StructuredTensorFields} to use upper indices to denote the corresponding inverse tensor field.
This vector field can be used to formulate the Lagrangian system:
Let $\curve \in \calC^\infty(\mnfTime, \T\mnfQ)$ be an integral curve of $\vfLag$ with starting point $(\pointQInit, \velocityInit) \in \T\mnfQ$.
Then, solving the Euler--Lagrange equations \eqref{eq:euler_lagrange} for $\curveMnfQ$ is equivalent to finding the integral curve $\curve$ of $\vfLag$
with
$\evalFun[\pointTime]{\curve} = \evalFun[\pointTime]{\lift{\curveMnfQ}}$.
In bold notation, the system for $\curve$ reads
\begin{align}\label{eq:lag_fo_coord}
	\evalField[\pointTime]{\ddt \curveCoord}
=
\begin{pmatrix}
	\evalFun[\pointTime]{\curveVCoord}\\
	\inv{
		\evalField[
			{\evalFun[\pointTime]{\curveCoord}}
		]{\metricVCoord}
	} \left(
		\evalField[{
			\evalFun[\pointTime]{\curveCoord}
		}]{\fD_{\pointQCoord} \LagCoord}
		- 
		{\evalField[{
			\evalFun[\pointTime]{\curveCoord}
		}]{\fD^2_{\velocityCoord\pointQCoord} \LagCoord}}
		\evalFun[\pointTime]{\curveVCoord}
	\right)
\end{pmatrix}
\in \evalFun[\chartdomainTanSpaceMnfQ]{\chartmapTanSpaceMnfQ} \subseteq \R^{2\dimMnfQ}.
\end{align}
Here we denote by $\evalField[\pointMnfQTanSpaceCoord]{\fD^2_{\velocityCoord\pointQCoord} \LagCoord} \in \R^{\dimMnfQ \times \dimMnfQ}$ the mixed derivative w.r.t.~$\velocityCoord$ and $\pointQCoord$
and the solution curve is split
$\evalFun[\pointTime]{\curveCoord} = (
	\evalFun[\pointTime]{\curveQCoord},
	\evalFun[\pointTime]{\curveVCoord}
)  \in \evalFun[\chartdomainTanSpaceMnfQ]{\chartmapTanSpaceMnfQ} \subseteq \R^{2 \dimMnfQ}$
in a part for $\pointQCoord$ and a part for $\velocityCoord$.
The system~\eqref{eq:lag_fo_coord} is typically referred to as the \emph{first-order formulation} for the Lagrangian system.

%--------------------------------------------------------------------%
\subsection{Hamiltonian systems} \label{subsec:hamiltonian}
In this subsection, we derive a formulation of Hamiltonian systems on a manifold,\footnote{Hamiltonian systems may result from Lagrangian systems via a Legendre transformation, but this is not the subject of the current work, so we refer to \cite[Sec.~3.6]{Abraham1987}.} providing the structure to perform \MOR in the forthcoming \Cref{subsec:MORHamiltonian}.
Let us recall from \Cref{subsec:cotangent_space}
that the differential of a smooth function $G \in \calC^{\infty}(\mnf, \R)$ at a point $\point \in \mnf$ defines a covector $\evalField{\diff{G}} \in \coTpMnf$.
Extending this idea, the differential $\diff{G} \in \calC^{\infty} (\T\mnf, \T\R)$ defines a smooth covector field $\diff{G} \in \TrsFieldMnf[0,1]$
with component functions $\tidx{(\evalField{\diff{G}})}{}{i} = \evalField{\fracdiff{G}{\tidx{\chartmap}{i}{}}}$.

For a given symplectic manifold $(\mnf, \symplForm)$ and a smooth function $\Ham \in \calC^\infty(\mnf, \R)$ referred to as the \emph{Hamiltonian (function)},
the \emph{Hamiltonian vector field}
\begin{align*}
	&\vfHam
\vcentcolon= \isoSharp[\symplForm]{\diff\Ham} \in \TrsFieldMnf[1,0],&
	&\text{or in index notation: }
	\tidx{(\evalField{\vfHam})}{i}{}
= \tidx{(\evalField{\symplForm})}{ij}{}\;
\tidx{(\evalField{\diff\Ham})}{}{j}
\end{align*}
is uniquely defined due to the nondegeneracy of $\symplForm$.
A \emph{Hamiltonian system} $(\mnf, \symplForm, \Ham)$ is an initial value problem 
\eqref{eq:integral_curve}
with an integral curve $\curve \in \calC^\infty(\mnfTime, \mnf)$ of $\vfHam$ with starting point $\pointInit \in \mnf$, i.e.,
\begin{align}\label{eq:hamiltonian_system}
	&\evalField[\pointTime]{\ddt \curve}
= \evalField[{{\evalFun[\pointTime]{\curve}}}]{\vfHam}
\in \TpMnf[{{\evalFun[\pointTime]{\curve}}}]&
&\text{and}&
&\evalFun[\pointTimeInit]{\curve} = \pointInit \in \mnf.
\end{align}
We denote a Hamiltonian system in bold notation%
\footnote{As the Jacobian $\evalField[\pointCoord]{\fD\HamCoord} \in \R^{1 \times \dimMnf}$ is a row vector,
we need to transpose it for the multiplication to match dimensions.}
with
\begin{align}\label{eq:hamiltonian_system_coord}
	&\evalField[\pointTime]{\ddt \curveCoord}
= \invb{
	\evalField[{{
		\evalFun[\pointTime]{\curveCoord}
	}}]\symplFormCoord}
\rT{\evalField[{{
	\evalFun[\pointTime]{\curveCoord}}
}]{\D\HamCoord}} \in \R^{\dimMnf},&
	&\evalFun[\pointTimeInit]{\curveCoord} = \pointInitCoord \in \R^{\dimMnf}.
\end{align}
This special construction of the vector field guarantees that the Hamiltonian is conserved along the solution curve, since
\begin{align*}
	\evalField[\pointTime]{\ddt (\Ham \circ \curve)}
\stackrel{\eqref{eq:chain_rule}}{=} \tidx{\left(
	\evalField[{{\evalFun[\pointTime]{\curve}}}]{\diff \Ham}
\right)}{}{i}
\tidx{\left(
	\evalField[\pointTime]{\ddt \curve}
\right)}{i}{}
\stackrel{\eqref{eq:hamiltonian_system}}{=}
\tidx{\left(
	\evalField[{{\evalFun[\pointTime]{\curve}}}]{\diff \Ham}
\right)}{}{i}
\tidx{\left(
	\evalField[{{\evalFun[\pointTime]{\curve}}}]{\symplForm}
\right)}{ij}{}
\tidx{\left(
	\evalField[{{\evalFun[\pointTime]{\curve}}}]{\diff \Ham}
\right)}{}{j}
= 0,
\end{align*}
where the last step uses that for skew-symmetric tensors $\tensor \in \Trs[2,0]{\TpMnf}$, it holds $\tidx{\covec}{}{i}\, \tidx{\tensor}{ij}{}\, \tidx{\covec}{}{j} = - \tidx{\covec}{}{i}\, \tidx{\tensor}{ij}{}\, \tidx{\covec}{}{j} = 0$ for all covectors $\covec \in \coTpMnf$.

For two given symplectic manifolds $(\mnf, \symplForm)$ and $(\mnfAlt, \symplFormAlt)$,
we call a smooth diffeomorphism $F \in \calC^{\infty} (\mnfAlt, \mnf)$ a \emph{symplectomorphism} if $\pullback{F}\symplForm = \symplFormAlt$.
It can be shown that the flow of a Hamiltonian system $\flowAt: \mnf \to \mnf$ is a symplectomorphism.

The theorem of Darboux (see e.g. \cite[Thm.~3.2.2]{Abraham1987}) guarantees that for each point $\point \in \mnf$,
there exists a chart $\chart$ with $\point \in \chartdomain$ which is \emph{canonical},
i.e., the symplectic form in these coordinates can be represented with $\symplFormCoord \equiv \rT\JtN$ by the \emph{canonical Poisson tensor}%
\footnote{Note that in contrast to existing work in the field of structure-preserving \MOR of Hamiltonian systems,
we speak of the symplectic form $\symplFormCoord = \rT\JtN$ instead of $\JtN$. This yields the same Hamiltonian vector field $\evalField[\pointCoord]{\vfHamCoord}
= \JtN \rT{\evalField[\pointCoord]{\D \HamCoord}}$
and does not change the reduction formulas later, but it helps to understand the more general case of noncanonical coordinates $\symplFormCoord \neq \rT\JtN$.}
\begin{align}\label{eq:poisson_tensor}
	&\JtN = \begin{pmatrix}
		\tname{\fzero}{\dimMnf} & \tname{\fI}{\dimMnf}\\
		 -\tname{\fI}{\dimMnf} & \tname{\fzero}{\dimMnf}
	\end{pmatrix} \in \R^{2\dimMnf \times 2\dimMnf}&
	&\text{for which}&
	&\rT\JtN = -\JtN = \inv\JtN,
\end{align}
where $\tname{\fI}{\dimMnf}, \tname{\fzero}{\dimMnf} \in \R^{\dimMnf \times \dimMnf}$ are the identity matrix and matrix of all zeros, respectively.
In the case of $\mnf = \R^{2\dimMnf}$ with $\evalField[\pointCoord]{\symplFormCoord} = \rT\JtN$ for all $\point \in \mnf$, we call $(\R^{2\dimMnf}, \rT\JtN, \HamCoord)$ a \emph{canonical Hamiltonian system}.

%------------------------------------------%
% SECTION 5: STRUCTURE-PRESERVIGN MOR %
%------------------------------------------%
\section{Structure-preserving MOR on manifolds}
\label{sec:structurePreservation}

With the general model reduction framework presented in \Cref{sec:mor_on_mnf} at hand, we now discuss how the general framework can be specialized to preserve important features of the initial value problem on the manifold. In more detail, we first introduce the \emph{generalized manifold Galerkin} (\GMG) reduction map in \Cref{subsec:GMG} and then use it to discuss the structure-preserving \MOR of
\begin{itemize}
	\item Lagrangian systems in \Cref{subsec:MORLagrangian}, and 
	\item Hamiltonian systems in \Cref{subsec:MORHamiltonian}.
\end{itemize}

%-----------------------------------------------------------------------------%
\subsection{Generalized manifold Galerkin}
\label{subsec:GMG}
Assume that the manifold $\mnf$ of dimension $\dimMnf$ is endowed with a nondegenerate $(0,2)$-tensor field $\tf \in \TrsFieldMnf[0,2]$, as defined in \Cref{subsec:StructuredTensorFields}. As in \Cref{subsec:MPG}, we assume that we have already constructed an embedded submanifold $\mnfSub \subseteq \mnf$ defined by a smooth embedding $\emb \in \calC^{\infty}(\mnfRed, \mnf)$, i.e., we have completed the \stepApproximation~step from the general \MOR workflow in \Cref{subsubsec:MORworkflow}. The straightforward way to define a reduced tensor field is to use the pullback from \Cref{subsec:pullback}. Hence, we make the following assumption.

\begin{assumption}
	\label{ass:reducedPseudoRiemannian}
	Given the nondegenerate $(0,2)$-tensor field $\tf \in \TrsFieldMnf[0,2]$, the smooth embedding $\emb \in \calC^{\infty}(\mnfRed, \mnf)$ is such that the reduced tensor field
	\begin{equation*}
		\tfRed \vcentcolon= \pullback{\emb}{\tf} \in \TrsFieldMnfRed[0,2],
	\end{equation*}
	is nondegenerate.
\end{assumption}

Note that the reduced tensor field in bold notation reads
\begin{align}
	\label{eq:pmetric_red_coord}
	\evalField[\pointRedCoord]{\tfRedCoord}
= 
\rT{\evalField[\pointRedCoord]{\D\embCoord}}\,
\evalField[{\pointSubCoord}]{\tfCoord}
\evalField[\pointRedCoord]{\D\embCoord}
\in \R^{\dimMnfRed \times \dimMnfRed},
\end{align}
which immediately illustrates that \Cref{ass:reducedPseudoRiemannian} may not be satisfied, in general. For instance, if we take $\mnf = \R^2$ and $\mnfRed = \R$, the tensor field to be a constant skew-symmetric matrix and a linear embedding, then \Cref{ass:reducedPseudoRiemannian} is violated. See also the forthcoming \Cref{ex:degenerateTensorField}. On the other hand, if the tensor field is a Riemannian metric on $\mnf$, i.e., symmetric and positive definite, then the reduced tensor field is also a Riemannian metric.

We immediately obtain the following relation between the full and reduced musical isomorphisms discussed in \Cref{subsec:StructuredTensorFields}.
\begin{lemma}
	\label{lem:gmg_pmetricRed}
	Under \Cref{ass:reducedPseudoRiemannian}, it holds
	\begin{align}
		\label{eq:lem_gmg_pmetricRed}
		\evalField[\pointRed]{\dualdiff{\emb}}
		\circ
		\flat_{\tf}
		\circ \evalField[\pointRed]{\diff{\emb}}
		=
		\flat_{\tfRed} \in \calC^{\infty} (\TpMnfRed, \coTpMnfRed).
	\end{align}
\end{lemma}

\begin{proof}
	We prove the statement in index notation.
	Using~\eqref{eq:isoFlat}, \eqref{eq:pullback}, \eqref{eq:diff}, and \eqref{def:diff_pullback}, we obtain for all $\pointRed \in \mnfRed$, all $\velocityRed \in \TpMnfRed$ and all $1 \leq i \leq \dimMnfRed$
	\begin{align*}
		\tidx{(\evalFun[\velocityRed]{\flat_{\tfRed}})}{}{i}
		&= \tidx{(\evalField[\pointRed]{\tfRed})}{}{ij}\;
\tidx{\velocityRed}{j}{}
		= \tidx{
			(\evalField[{\evalFun[\pointRed]{\emb}}]{\tf})
		}{}{\tname{\ell}{1} \tname{\ell}{2}}\;
\trafo{\emb}{\pointRed}{\tname{\ell}{1}}{\chartmap}{i}
\trafo{\emb}{\pointRed}{\tname{\ell}{2}}{\chartmap}{j}
\tidx{\velocityRed}{j}{}\\
		&= \trafo{\emb}{\pointRed}{\tname{\ell}{1}}{\chartmap}{i}
\tidx{
	(\evalField[{\evalFun[\pointRed]{\emb}}]{\tf})
}{}{\tname{\ell}{1} \tname{\ell}{2}}\;
\trafo{\emb}{\pointRed}{\tname{\ell}{2}}{\chartmap}{j}
\tidx{\velocityRed}{j}{}
=
\tidx{\left(
	\evalFun[\velocityRed]{
		(
		\evalField[\pointRed]{\dualdiff{\emb}}
		\circ
		\flat_{\tf}
		\circ
		\evalField[\pointRed]{\diff{\emb}}
		)
	}
\right)}{}{i}.\qedhere
	\end{align*}
\end{proof}

The additional structure allows us to construct an alternative reduction mapping to the \MPG reduction map~\eqref{eq:mpg_projection}, which we refer to as the \emph{generalized manifold Galerkin} (\GMG)
\begin{equation}
	\label{eq:gmg_reduction}
	\gmgReduction\colon \T\mnf \supseteq \vecbunSub \to \T\mnfRed,\qquad
(\point,\velocity)
\mapsto
\left(
	\redMapPoint(\point),\evalFun[{\velocity}]{\left(
		\sharp_{\tfRed}
		\circ
		\evalField[{{\evalFun{\redMapPoint}}}]{\dualdiff{\emb}}
		\circ
		\flat_{\tf}
	\right)}
\right),
\end{equation}
which is defined on the vector bundle
\begin{align*}
	\vecbunSub \vcentcolon= \dot{\bigcup}_{\point \in \mnfSub} \TpMnf.
\end{align*}
The domain $\vecbunSub \subseteq \T\mnf$ of the \GMG reduction map is in general smaller than in the original definition of a reduction map \Crefpoint{def:reductionMapping}. Nevertheless, all previous results are valid for reduction maps $\redMap: \vecbunSub \to \T\mnfRed$ as the reduction map is only used in the \ROM to project $\evalField[\pointSub]{\vf} \in \TpMnf[\pointSub]$ which is part of $\vecbunSub$.
We avoided introducing $\vecbunSub$ earlier for a better readability.
The restriction of the domain for the \GMG is necessary as $\evalField[\pointRed]{\dualdiff{\emb}}: \coTpMnf[{{\evalFun[\pointRed]{\emb}}}] \to \coTpMnfRed$ is defined on $\coTpMnf[{{\evalFun[\pointRed]{\emb}}}]$ only.

By construction, \eqref{eq:projProperty:pointProjection}, \Cref{lem:gmg_pmetricRed}, and \eqref{eq:iso_sharp_inv}, we obtain
\begin{align*}
	\evalField[\pointSub]{\gmgReduction} \circ \evalField[\pointRed]{\diff{\emb}}
	&=
	\sharp_{\tfRed}
	\circ
	\evalField[(\red \circ \emb)(\pointRed)]{\dualdiff{\emb}}
	\circ
	\flat_{\tf}
	\circ
	\evalField[\pointRed]{\diff{\emb}}
	= \sharp_{\tfRed} \circ \flat_{\tfRed}
	= \id_{\TpMnfRed},
\end{align*}
which proves the following result.

\begin{theorem}\label{thm:gmg}
	The \GMG reduction~\eqref{eq:gmg_reduction} is a reduction map for $\emb$.
\end{theorem}

The corresponding \ROM~\eqref{eq:rom} obtained with the \GMG reduction map is called \emph{\GMG-\ROM}. In bold notation, the associated reduced vector field for the \FOM vector field $\vf \in \smoothVfs$ reads with \eqref{eq:projProperty:pointProjection}
\begin{align}
	\evalField[\pointSubCoord]{
		\gmgReductionCoord
	}
	\left( \evalField[\pointSubCoord]{\vfCoord} \right)
	&=
	\invb{
		\rT{\evalField[(\redMapPointCoord \circ \embCoord)(\pointRedCoord)]{\D\embCoord}}\,
		\evalField[{\pointSubCoord}]{\tfCoord}
		\evalField[(\redMapPointCoord \circ \embCoord)(\pointRedCoord)]{\D\embCoord}
	}
	\rT{\evalField[(\redMapPointCoord \circ \embCoord)(\pointRedCoord)]{\D\embCoord}}\,
	\evalField[{\pointSubCoord}]{\tfCoord}
	\evalField[{\pointSubCoord}]{\vfCoord}\nonumber\\
	\label{eq:gmg_projection_coord}
	&=
	\invb{
		\rT{\evalField[\pointRedCoord]{\D\embCoord}}\,
		\evalField[{\pointSubCoord}]{\tfCoord}
		\evalField[\pointRedCoord]{\D\embCoord}
	}
	\rT{\evalField[\pointRedCoord]{\D\embCoord}}\,
	\evalField[{\pointSubCoord}]{\tfCoord}
	\evalField[{\pointSubCoord}]{\vfCoord}
	\in \R^{\dimMnfRed}.
\end{align}

To motivate the name \GMG, we consider the special case that $\mnf = \R^{\dimMnf}$, $\mnfRed = \R^{\dimMnfRed}$ are vector spaces over $\R$ (with identity charts $\chartmap \equiv \id_{\R^{\dimMnf}}$, $\chartmapRed \equiv \id_{\R^{\dimMnfRed}}$) and the nondegenerate tensor field $\tf$ is a Riemannian metric that is constant in coordinates, i.e., $\evalField[\pointCoord]{\tfCoord} = \tfCoord = \mathrm{const}$. We then obtain with \eqref{eq:gmg_projection_coord}
\begin{align*}
		\evalField[\pointSubCoord]{
			\gmgReductionCoord
		}
		\left( \evalField[\pointSubCoord]{\vfCoord} \right)
		&=
		\invb{
			\left(
				\rT{\evalField[\pointRedCoord]{\D\embCoord}}
				\sq{\tfCoord}
			\right)
			\left(
				\sq{\tfCoord}
				\evalField[\pointRedCoord]{\D\embCoord}
			\right)
		}
		\left(
			\rT{\evalField[\pointRedCoord]{\D\embCoord}}
			\sq{\tfCoord}
		\right)
		\sq{\tfCoord}
		\evalField[{\pointSubCoord}]{\vfCoord}\\
		&= 
		\left(
			\sq{\tfCoord}
			\evalField[\pointRedCoord]{\D\embCoord}
		\right)^\dagger
		\sq{\tfCoord}
		\evalField[{\pointSubCoord}]{\vfCoord},
\end{align*}
where $(\cdot)^\dagger$ denotes the Moore--Penrose pseudoinverse. In particular, we recover the \emph{manifold Galerkin projection} introduced in \cite[Rem~3.4]{Lee20}, which allows interpreting the reduced vector field as the optimal projection w.r.t.~the Riemannian metric $\tf$; see \cite[Sec.~3.2]{Lee20}.

\begin{example}[Special case: linear-subspace MOR]
	\label{ex:lin_mor_gmg}
	In the case of $\red,$ $\emb$ being linear as in \Cref{rem:subspace_MOR} with $\tfCoord \equiv \mathrm{const}$ and $\rT\fV \tfCoord \fV = \tname{\fI}{\dimMnfRed}$,
	the \GMG reduction is exactly the \ROM obtained via standard Galerkin projection
	\begin{align*}
		&\evalField[\pointSubCoord]{\gmgPullbackCoordSymbol}(\velocityCoord)
= \invb{\rT\fV \tfCoord \fV} \rT\fV \tfCoord \velocityCoord
= \rT\fV \tfCoord \velocityCoord,
		&\evalField[\pointTime]{\ddt \curveRedCoord}
= \rT\fV \tfCoord \evalField[{
	\evalFun[\pointTime]{\curveRedCoord}
}]{\vfCoord}.
	\end{align*}
\end{example}

As discussed in \Cref{sec:DifGeoPartTwo},
the \FOM vector field may possess additional structure in specific applications, such as Lagrangian or Hamiltonian dynamics.
In the following, we show that the \GMG reduction can be used to formulate structure-preserving \MOR (on manifolds) for Lagrangian and Hamiltonian systems
by choosing a specific nondegenerate tensor field.

%-----------------------------------------------------------------------------%
\subsection{MOR on manifolds for Lagrangian systems}
\label{subsec:MORLagrangian}
As in \Cref{subsec:lagrange}, consider a $\dimMnfQ$-dimensional smooth manifold $\mnfQ$ with a chart $\chartMnfQ$ and the corresponding chart $\chartTanSpaceMnfQ$ of the tangent bundle $\T\mnfQ$ \Crefpoint{subsec:lagrange}.
The manifold to be reduced in the context of Lagrangian systems is the tangent bundle $\T\mnfQ$. To be consistent with the notation introduced before, we thus set $\mnf \vcentcolon= \T\mnfQ$ with even dimension $\dimMnf \vcentcolon= \dim(\mnf) \vcentcolon= 2 \dimMnfQ$. Instead of working directly on $\mnf$, we still aim for a construction on $\mnfQ$ by employing that the differential of a smooth map \Crefpoint{subsec:tangent_bundle_vector_field} is a mapping between the associated tangent spaces.

\begin{definition}[Lifted embedding and lifted point reduction]\label{def:lifted_emb}
	Consider an embedded submanifold $\mnfQSub\subseteq\mnfQ$ defined by a $\dimMnfQRed$-dimensional manifold $\mnfQRed$ and a smooth embedding $\embQ\colon \mnfQRed \to \mnfQSub$. Then, we call
	\begin{align*}
		\embTanSpace \vcentcolon= \diff{\embQ} \colon \T\mnfQRed \to \T{\left( \mnfQSub \right)},\qquad
 \left( \pointQRed, \velocityRed \right) \mapsto \left(
	\evalFun[\pointQRed]{\embQ},
	\evalField[\pointQRed]{\diff{\embQ}} \left( \velocityRed \right)
\right)
	\end{align*}
	the \emph{lifted embedding} for $\embQ$.
	Analogously, for a point reduction $\redMnfQ\colon \mnfQ \to \mnfQRed$,
	we define the \emph{lifted point reduction}
	\begin{align*}
		\redTanSpace \vcentcolon= \diff{\redMnfQ} \colon
\T\mnfQ \to \T\mnfQRed,\qquad
(\pointQ, \velocity) \mapsto
\left( \evalFun[\pointQ]{\redMnfQ}, \evalField[{\pointQ}]{\diff{\redMnfQ}} (\velocity) \right).
	\end{align*}
\end{definition}

Let us emphasize that $\redTanSpace$ is indeed a point reduction on $\mnf = \T\mnfQ$ for the lifted embedding~$\emb$, which is a straightforward consequence of \Cref{thm:red_emb_duality}.
For $\smash{\Big( \pointQRed, \tidx{\velocityRed}{k}{} \basisMnfQRedTanAt{k} \Big)} \in \T\mnfQRed$ with a chart $\chartMnfQRed$ for $\mnfQRed$ and $\chartTanSpaceMnfQRed$ for $\T\mnfQRed$, we immediately obtain
\begin{align*}
		\evalField[\left( \pointQRed, \tidx{\velocityRed}{k}{} \basisMnfQRedTanAt{k} \right)]{
			\fracdiff{\tidx{\embTanSpace}{i}{}}{\tidx{\chartmapTanSpaceMnfQRed}{j}{}}
		}
	= \begin{cases}
		\evalField[{\pointQRed}]{
			\fracdiff{\tidx{\embQ}{i}{}}
				{\tidx{\chartmapMnfQRed}{j}{}}
		},
		&1\leq i\leq \dimMnfQ,\; 1\leq j\leq \dimMnfQRed,\\
		0,
		&1\leq i\leq \dimMnfQ,\; 1\leq j-\dimMnfQRed \leq \dimMnfQRed,\\
		\evalField[{\pointQRed}]{
			\sfracdiff{\tidx{\embQ}{i - \dimMnfQ}{}}
				{\tidx{\chartmapMnfQRed}{k}{}}
				{\tidx{\chartmapMnfQRed}{j}{}\,}
		} \tidx{\velocityRed}{k}{},
		&1\leq i - \dimMnfQ \leq \dimMnfQ,\; 1\leq j \leq \dimMnfQRed,\\
		\evalField[{\pointQRed}]{
			\fracdiff{\tidx{\embQ}{i-\dimMnfQ}{}}
				{\tidx{\chartmapMnfQRed}{j-\dimMnfQRed}{}}
		},
		&1\leq i - \dimMnfQ \leq \dimMnfQ,\; 1\leq j-\dimMnfQRed \leq \dimMnfQRed,\\
	\end{cases}
\end{align*}
which reads in bold notation
\begin{align*}
		&\evalFun[{\pointQRedCoord, \velocityRedCoord}]{\embTanSpaceCoord}
= \begin{pmatrix}
	\evalFun[\pointQRedCoord]{\embQCoord}\\
	\evalField[\pointQRedCoord]{\fD \embQCoord}\velocityRedCoord
\end{pmatrix}
\in \R^{2 \dimMnfQ},&
		&\evalField[(\pointQRedCoord, \velocityRedCoord)]{\fD \embTanSpaceCoord}
= \begin{pmatrix}
	\evalField[\pointQRedCoord]{\fD \embQCoord} &
\fzero \\
	\evalField[\pointQRedCoord]{\fD \left(
		\evalField[(\cdot)]{\fD \embQCoord} \velocityRedCoord \right)} &
\evalField[\pointQRedCoord]{\fD \embQCoord}
\end{pmatrix}
\in \R^{2 \dimMnfQ \times 2 \dimMnfQRed}.
\end{align*}

\begin{example}
	\label{ex:linear_embTanSpaceCoord}
	For a linear embedding $\evalFun[\pointQRedCoord]{\embQCoord} = \fV \pointQRedCoord$ as in \Cref{rem:subspace_MOR},
	the lifted embedding from \Cref{def:lifted_emb} is described by a block-diagonal basis matrix
	\begin{align*}
		\evalFun[\pointQRedCoord, \velocityRedCoord]{\embTanSpaceCoord}
= \begin{pmatrix}
	\fV & \fzero\\
	\fzero & \fV
\end{pmatrix}
\begin{pmatrix}
	\pointQRedCoord\\
	\velocityRedCoord
\end{pmatrix},
	\end{align*}
	which is frequently used in \MOR for second-order systems (see, e.g., \cite{Ruiner2012}).
\end{example}

With these preparations, let us now assume that we have a Lagrangian system $(\mnfQ, \Lag)$ with initial value $(\pointQInit, \velocityInit) \in \T\mnfQ$ together with embedded submanifold $\mnfQSub \subseteq \mnfQ$ with the embedding $\embQ$ and a point reduction $\redMnfQ$ available. Let $\embTanSpace$ and $\redTanSpace$ denote the corresponding lifted embedding and lifted point reduction as in \Cref{def:lifted_emb}. To preserve the Lagrangian system structure in the \ROM, we do not aim for a projection of the Euler--Lagrange equations~\eqref{eq:euler_lagrange} but rather start by constructing a reduced Lagrangian via
\begin{equation}
	\label{eqn:Lagrangian_red}
	\LagRed \vcentcolon= \pullback{\embTanSpace}{\Lag} = \Lag \circ \embTanSpace \in \calC^{\infty}(\T\mnfQRed)
\end{equation}
and immediately obtain the reduced Lagrangian system $(\mnfQRed,\LagRed)$ with reduced initial value~$(\pointQInitRed, \velocityInitRed) \vcentcolon= \evalFun[\pointQInit, \velocityInit]{\redTanSpace} \in \T\mnfQRed =\vcentcolon \mnfRed$. Note that with this strategy, we immediately obtain the \ROM that itself is a Lagrangian system, which is not automatically guaranteed if we reduce the vector field~\eqref{eq:def_vfLag}. Straightforward calculations (see \Cref{appx:proof_red_euler_lagrange}) show that the Euler-Lagrange equations of the reduced Lagrangian system read
\begin{equation}
	\label{eq:red_euler_lagrange}
	\begin{aligned}
		0 &=
		\evalField[{\evalFun[\pointTime]{\curveRed}}]{
			\fracdiff{\tidx{\embQ}{j}{}}
				{\tidx{\chartmapMnfQRed}{i}{}}
		}
		\Bigg(
		\evalField[
			{\evalFun[
				{\evalFun[\pointTime]{\lift{\curveRed}}}
			]{\embTanSpace}}
		]{
			\fracdiff{\Lag}{\tidx{\chartmapTanSpaceMnfQ}{j}{}}
		}
		-
		% first term: k, l, l
		\evalField[
			{\evalFun[
				{\evalFun[\pointTime]{\lift{\curveMnfQRed}}}
			]{\embTanSpace}}
		]{
			\sfracdiff{\Lag}
				{\tidx{\chartmapTanSpaceMnfQ}{\dimMnfQ + j}{}}
				{\tidx{\chartmapTanSpaceMnfQ}{k}{}\,}
		}
		\evalField[{\evalFun[\pointTime]{\curveMnfQRed}}]{
			\fracdiff{\tidx{\embQ}{k}{}}
				{\tidx{\chartmapMnfQRed}{\ell}{}}
		}
		\evalField[\pointTime]{\ddt\tidx{\curveMnfQRed}{\ell}{}}\\
			&\phantom{\evalField[{\evalFun[\pointTime]{\curveMnfQRed}}]{
				\fracdiff{\tidx{\embQ}{j}{}}
					{\tidx{\chartmapMnfQRed}{i}{}}
			}
			\Bigg(}
		-
		% third term: N+k, l, l
		\evalField[
			{\evalFun[
				{\evalFun[\pointTime]{\lift{\curveMnfQRed}}}
			]{\embTanSpace}}
		]{
			\sfracdiff{\Lag}
				{\tidx{\chartmapTanSpaceMnfQ}{\dimMnfQ + j}{}}
				{\tidx{\chartmapTanSpaceMnfQ}{\dimMnfQ + k}{}\,}
		}
		\evalField[{\evalField[\pointTime]{{\curveMnfQRed}}}]{
			\sfracdiff{\tidx{\embQ}{k}{}}
				{\tidx{\chartmapMnfQRed}{p}{}}
				{\tidx{\chartmapMnfQRed}{\ell}{}\,}
		}
		\evalField[\pointTime]{\ddt\tidx{\curveMnfQRed}{p}{}}
		\evalField[\pointTime]{\ddt\tidx{\curveMnfQRed}{\ell}{}}\\
		% fourth term: N+k, n+l, n+l
		&\phantom{\evalField[{\evalFun[\pointTime]{\curveMnfQRed}}]{
			\fracdiff{\tidx{\embQ}{j}{}}
				{\tidx{\chartmapMnfQRed}{i}{}}
		}
		\Bigg(}
		-
		\evalField[
			{\evalFun[
				{\evalFun[\pointTime]{\lift{\curveMnfQRed}}}
			]{\embTanSpace}}
		]{
			\sfracdiff{\Lag}
				{\tidx{\chartmapTanSpaceMnfQ}{\dimMnfQ + j}{}}
				{\tidx{\chartmapTanSpaceMnfQ}{\dimMnfQ + k}{}\,}
		}
		\evalField[{\evalFun[\pointTime]{\curveMnfQRed}}]{
			\fracdiff{\tidx{\embQ}{k}{}}
				{\tidx{\chartmapMnfQRed}{\ell}{}}
		}
		\evalField[\pointTime]{\sddt\tidx{\curveMnfQRed}{\ell}{}}
		\Bigg)
	\end{aligned}
	\end{equation}
	for $1\leq i \leq \dimMnfRed$ where the right-hand side sums over $1 \leq j,k \leq \dimMnfQ$ and $1 \leq \ell,p \leq \dimMnfQRed$ by the Einstein summation convention \eqref{eq:einstein_summation_convention}. In bold notation, the reduced Euler--Lagrange equations read
	\begin{equation}\label{eq:red_euler_lagrange_coord}
	\begin{aligned}
		\fzero =&\; \rT{\evalField[
			{\evalFun[\pointTime]{\curveMnfQRedCoord}}
		]{
			\fD \embQCoord
		}}
		\Bigg(
			\evalField[{\evalFun[
				{\evalFun[\pointTime]{\curveMnfQRedCoord}},
				{\evalField[\pointTime]{\ddt \curveMnfQRedCoord}}
			]{\embTanSpaceCoord}}]{
				\fD_{\pointQCoord} \LagCoord
			}
			-
			\evalField[{\evalFun[
				{\evalFun[\pointTime]{\curveMnfQRedCoord}},
				{\evalField[\pointTime]{\ddt \curveMnfQRedCoord}}
			]{\embTanSpaceCoord}}]{
				\fD^2_{\velocityCoord\pointQCoord} \LagCoord
			}
			\evalField[
				{\evalFun[\pointTime]{\curveMnfQRedCoord}}
			]{
				\fD \embQCoord
			}
			{\evalField[\pointTime]{\ddt \curveMnfQRedCoord}}\\
			&\;\phantom{
				\rT{\evalField[
					{\evalFun[\pointTime]{\curveMnfQRedCoord}}
				]{
					\fD \embQCoord
				}}
				\Bigg(
			}
			-
			\evalField[{\evalFun[
				{\evalFun[\pointTime]{\curveMnfQRedCoord}},
				{\evalField[\pointTime]{\ddt \curveMnfQRedCoord}}
			]{\embTanSpaceCoord}}]{
				\fD^2_{\velocityCoord\velocityCoord} \LagCoord
			}
			\evalField[{
				{\evalFun[\pointTime]{\curveMnfQRedCoord}}
			}]{
				\fD^2 \embQCoord
			}
			\left(
			\evalField[\pointTime]{\ddt \curveMnfQRedCoord}
			\otimes
			\evalField[\pointTime]{\ddt \curveMnfQRedCoord}
			\right)
			\\
			&\;\phantom{
				\rT{\evalField[
					{\evalFun[\pointTime]{\curveMnfQRedCoord}}
				]{
					\fD \embQCoord
				}}
				\Bigg(
			}
			-
			\evalField[{\evalFun[
				{\evalFun[\pointTime]{\curveMnfQRedCoord}},
				{\evalField[\pointTime]{\ddt \curveMnfQRedCoord}}
			]{\embTanSpaceCoord}}]{
				\fD^2_{\velocityCoord\velocityCoord} \LagCoord
			}
			\evalField[
				{\evalFun[\pointTime]{\curveMnfQRedCoord}}
			]{
				\fD \embQCoord
			}
			{\evalField[\pointTime]{\sddt \curveMnfQRedCoord}}
		\Bigg)
		\in \R^{\dimMnfQRed}.
	\end{aligned}
\end{equation}
By construction, the reduced Lagrangian system fulfills the Euler--Lagrange equations for the reduced Lagrangian $\LagRed$. Thus, \Cref{thm:euler_lagrange_energy} guarantees that along the lift of the solution curve $\curveMnfQRed$, the \emph{reduced energy}
	\begin{align*}
		\energyRed\colon \T\mnfQRed \to \R,\qquad
	\pointMnfQTanSpaceRed = \left(\pointQRed, \tidx{\velocityRed}{i}{} \basisMnfQRedTanAt{i}\right)
	\mapsto
	\tidx{\velocityRed}{j}{}
	\evalField[\pointMnfQTanSpaceRed]{
		\fracdiff{\LagRed}
			{\chartmapMnfQRedVeli{j}}
	}
	- \evalFun[\pointMnfQTanSpaceRed]{\LagRed}
	\end{align*}
	is preserved. Moreover, it holds $\energyRed \equiv \energy \circ \embTanSpace$.

Following the construction in \Cref{subsec:lagrange} and assuming that the reduced Lagrangian~$\LagRed$ is regular, we can formulate a reduced vector field $\vfLagRed \in \TrsFieldMnfRed[1,0]$ for the reduced Euler--Lagrange equations \eqref{eq:red_euler_lagrange_coord}. Indeed, we obtain
for a point $\pointMnfQTanSpaceRed = \left(\pointQRed, \tidx{\velocityRed}{i}{} \basisMnfQRedTanAt[\pointQRed]{i}\right) \in \T\mnfQRed$ as
\begin{equation}\label{eq:red_vf_euler_lagrange}
	\begin{aligned}
		\evalField[\pointMnfQTanSpaceRed]{\vfLagRed}
		&=\vcentcolon
\tidx{\velocityRed}{i}{}
\basisTMnfQRedTanPosAt{i} +
\tidx{(\evalField[\pointMnfQTanSpaceRed]{\sigma})}{i}{\, \ell}
\Bigg(
\tidx{\left(
\evalField[
	{\evalFun[
		\pointMnfQTanSpaceRed
	]{\embTanSpace}}
]{\vfLag}\right)}{\dimMnfQ + \ell}{}
-
\evalField[{\pointQRed}]{
	\sfracdiff{\tidx{\embQ}{\ell}{}}
		{\tidx{\chartmapMnfQRed}{r}{}}
		{\tidx{\chartmapMnfQRed}{p}{}\,}
}
{\tidx{\velocityRed}{p}{}}\,
{\tidx{\velocityRed}{r}{}}
\Bigg)
\basisTMnfQRedTanVelAt{i}
		\end{aligned}
	\end{equation}
 with indices $1\leq k,\ell \leq\dimMnfQ$ and $1\leq i,j,p,r\leq \dimMnfQRed$ and
\begin{displaymath}
	\tidx{(\evalField[\pointMnfQTanSpaceRed]{\sigma})}{i}{\, \ell} \vcentcolon=
\tidx{(\evalField[
	\pointMnfQTanSpaceRed
]{\metricVRed})}{ij}{}\;
\evalField[{\pointQRed}]{
	\fracdiff{\tidx{\embQ}{k}{}}
		{\tidx{\chartmapMnfQRed}{j}{}}
}
\tidx{\left(\evalField[
	{\evalFun[
		\pointMnfQTanSpaceRed
	]{\embTanSpace}}
]{\metricV}\right)}{}{k\ell}
.
\end{displaymath}

In order to relate this reduction to our framework, we show in the following that
the reduced Euler--Lagrangian vector field~\eqref{eq:red_vf_euler_lagrange} can be interpreted as a \GMG reduction \eqref{eq:gmg_reduction} of the Euler--Lagrangian vector field~\eqref{eq:def_vfLag} if an appropriate tensor field is selected.
We refer to this as the \emph{Lagrangian manifold Galerkin} (\LMG).
With the nondegenerate tensor field~$\metricV$ from \eqref{eq:def_metricV},
we define a tensor field~$\metricLMG \in \TrsFieldMnf[0,2]$ on $\mnf = \T\mnfQ$ with
\begin{equation}
	\label{eqn:LMGtensorField}
	\evalField[\pointMnfQTanSpace]{\metricLMG}
\vcentcolon= \tidx{\left(
	\evalField[\pointMnfQTanSpace]{\metricQ}
 \right)}{}{ij}\;
\basisTMnfQCoTanVelAt{i}
\otimes
\basisTMnfQCoTanPosAt{j}
+
\tidx{\left(
	\evalField[\pointMnfQTanSpace]{\metricV}
\right)}{}{ij}\;
\basisTMnfQCoTanPosAt{i}
\otimes
\basisTMnfQCoTanVelAt{j},
\end{equation}
where $\tidx{(
	\evalField[\pointMnfQTanSpace]{\metricQ}
)}{}{ij}$ are additional components.
A typical choice could be $
\tidx{(
	\evalField[\pointMnfQTanSpace]{\metricQ}
)}{}{ij}
= \tidx{(
	\evalField[\pointMnfQTanSpace]{\metricV}
)}{}{ij}$.
In bold notation, the tensor field reads
\begin{align}\label{eq:lmg_metric_coord}
		\evalField[\pointMnfQTanSpaceCoord]{\metricLMGCoord} = \begin{pmatrix}
			\fzero & \evalField[\pointMnfQTanSpaceCoord]{\metricVCoord}\\
			\evalField[\pointMnfQTanSpaceCoord]{\metricQCoord} & \fzero
		\end{pmatrix}.
\end{align}

The associated reduced tensor field is denoted with $\metricLMGRed$ (as in \Cref{subsec:GMG}).
Assuming that $\metricLMGRed$ is nondegenerate,
we define the \LMG reduction map
\begin{equation}
	\label{eqn:lmg_reduction}
	\lmgReduction\colon \T\mnf \supseteq \vecbunSub \to\T\mnfRed, \qquad (\point,\velocity)
\mapsto
\left(
	\redMapPoint(\point),
	\evalFun[{\velocity}]{\left(
		\sharp_{\metricLMGRed}
		\circ
		\evalField[{{\evalFun{\redMapPoint}}}]{\dualdiff{\emb}}
		\circ
		\flat_{\metricLMG}
	\right)}
\right).
\end{equation}
The \LMG reduction map~\eqref{eqn:lmg_reduction} is a particular case of a \GMG reduction map, and thus, we immediately obtain from \Cref{thm:gmg} that $\lmgReduction$ is a reduction map for the lifted embedding~$\emb$.

\begin{theorem}\label{thm:equivalence_solve_lmg}
	Consider the \ROM obtained by reducing the Euler--Lagrange vector field with $\lmgReduction$. Then solving this \ROM for $\curveRed$ is equivalent to solving the reduced Euler--Lagrange equations \eqref{eq:red_euler_lagrange} for $\curveMnfQRed$ with $\evalFun[\pointTime]{\curveRed} = \evalFun[\pointTime]{\lift{\curveMnfQRed}}$.
\end{theorem}

\begin{proof}
	 \Crefpoint{appx:proof_equivalence_solve_lmg}.
\end{proof}

We conclude this section with three remarks.

\begin{remark}
	In the special case of a classical \MOR $\mnfQ = \R^{\dimMnfQ}$, $\mnfQRed=\R^{\dimMnfQRed}$ with a linear embedding $\embQCoord(\pointQRedCoord) = \fV \pointQRedCoord$ as in \Cref{ex:linear_embTanSpaceCoord}, a linear point reduction $\redMnfQCoord(\pointQCoord) = \rT\fV \pointQCoord$, and a quadratic Lagrangian $\LagCoord$, the reduced Euler--Lagrange equations \eqref{eq:red_euler_lagrange_coord} recover the \ROM from \cite{Lall2003}.
	In our framework that relates to the choice $\mnf = \R^{2\dimMnfQ}$, $\mnfRed = \R^{2\dimMnfQRed}$, $\embTanSpace$ as in~\Cref{ex:linear_embTanSpaceCoord}, and $\smash{\redMapLMGCoord (\velocityCoord_{\pointQCoord}, \velocityCoord_{\velocityCoord}) = \rT{\big(
		\rT{\velocityCoord_{\pointQCoord}} \rT{\fV},\;
		\rT{\velocityCoord_{\velocityCoord}} \rT{\fV}
	\big)}}$.
\end{remark}

\begin{remark}
	In \cite{Lall2003}, the authors argue that the reduced Euler--Lagrange equations cannot be obtained from a projection with the embedding $\embQ$ of the first-order system (which is formulated with the Euler--Lagrange vector field \eqref{eq:def_vfLag} in the scope of this work).
	This is no contradiction to our work since we suggest a projection based on the lifted embedding $\embTanSpace$ from \Cref{def:lifted_emb} to obtain the reduced Euler--Lagrange equations via a reduction of the Euler--Lagrange vector field.
\end{remark}

\begin{remark}[Second-order derivatives of $\embQ$]
	The reduced Euler--Lagrange equations require the computation of second-order derivatives of $\embQ$, which might be computationally intensive. Notably, the formulation of the \ROM in structure-preserving \MOR for Hamiltonian systems presented in the following subsection is independent of second-order derivatives of the embedding $\embQ$.
\end{remark}

%-----------------------------------------------------------------------------%
\subsection{MOR on manifolds for Hamiltonian systems}
\label{subsec:MORHamiltonian}
Lastly, we assume to be given a Hamiltonian system $(\mnf, \symplForm, \Ham)$ as \FOM and demonstrate how structure-preserving \MOR on manifolds can be formulated.
The procedure works analogously to the \GMG from \Cref{subsec:GMG},
while choosing the symplectic form $\symplForm$ as the nondegenerate tensor field $\tf = \symplForm$.
First, we assume that the \stepApproximation\ step is completed
and we are given a reduced manifold $\mnfRed$ and
a smooth embedding $\emb \in \calC^{\infty}(\mnfRed, \mnf)$ fulfilling \Cref{ass:reducedPseudoRiemannian}, i.e., $\pullback{\emb}{\symplForm}$ is nondegenerate.
We show at the end of this section \Crefpoint{lem:sympl_from_nondeg} that this assumption is enough that $\symplFormRed \vcentcolon= \pullback{\emb}{\symplForm}$ is a symplectic form and $(\mnfRed, \symplFormRed)$ is a symplectic manifold.
In this case, the embedding $\emb\colon (\mnfRed, \symplFormRed) \to (\mnfSub, \restrict{\symplForm}{\mnfSub})$ is a symplectomorphism.
Second, we use the reduction map
\begin{equation}
	\label{eq:smg_reduction}
	\smgReduction\colon \T\mnf \supseteq \vecbunSub \to \T\mnfRed,\qquad
(\point,\velocity)
\mapsto
\left(
	\redMapPoint(\point),
	\evalFun[{\velocity}]{\left(
		\sharp_{\symplFormRed}
		\circ
		\evalField[{{\evalFun{\redMapPoint}}}]{\dualdiff{\emb}}
		\circ
		\flat_{\symplForm}
	\right)}
\right),
\end{equation}
which we refer to as the \emph{symplectic manifold Galerkin} (\SMG) reduction map.
The \SMG reduction map is a special case of the \GMG reduction map \eqref{eq:gmg_reduction} with $\tf = \symplForm$ and $\tfRed = \symplFormRed$,
and, thus, we obtain from \Cref{thm:gmg} that $\smgReduction$ is a reduction map for $\emb$.
Hence, the \SMG reduction fits our \MOR framework from \Cref{sec:gen_frame} and it defines a \ROM by \eqref{eq:rom},
which we refer to as the \SMG-\ROM.
It remains to show that the \SMG-\ROM indeed is a Hamiltonian system,
which was the motivation for preserving the underlying structure.

\begin{theorem}\label{thm:smg_rom_ham_sys}
	The \SMG-\ROM is a Hamiltonian system $(\mnfRed, \symplFormRed, \HamRed)$ with the \emph{reduced Hamiltonian} $\HamRed \vcentcolon= \pullback{\emb} \Ham = \Ham \circ \emb$.
\end{theorem}

\begin{proof}
	The \ROM vector field with the \SMG reduction \eqref{eq:smg_reduction} reads
with (a) equations~\eqref{eq:projProperty:pointProjection} and~\eqref{eq:iso_sharp_inv},
and (b) equation~\eqref{eq:chain_rule_cotan}
	\begin{equation}\label{eq:smg_gmg_pullback}
	\begin{aligned}
		\evalField[\pointSub]{\smgReduction}(\evalField[\pointSub]{\vfHam})
=&\;
\evalFun[{
	\isoSharp[\symplForm]{
		\evalField[{\evalFun[\pointRed]{\emb}}]{\diff{\Ham}}
	}
}]{
\left(
	\sharp_{\symplFormRed}
	\circ
	\evalField[(\red \circ \emb)(\pointRed)]{\dualdiff{\emb}}
	\circ
	\flat_{\symplForm}
\right)
}\\
\stackrel{\text{(a)}}{=}&\; \isoSharp[\symplFormRed]{
	\evalField[\pointRed]{\dualdiff{\emb}} \left(
		\evalField[{\evalFun[\pointRed]{\emb}}]{\diff{\Ham}}
	\right)
}
\stackrel{\text{(b)}}{=} \isoSharp[\symplFormRed]{
	{\evalField[\pointRed]{\diff \HamRed}}
},
	\end{aligned}
	\end{equation}
	which is exactly the Hamiltonian vector field of the Hamiltonian system $(\mnfRed, \symplFormRed,\HamRed)$.
\end{proof}

Using \eqref{eq:gmg_projection_coord}, the reduced vector field in the \SMG-\ROM in bold notation reads 
\begin{align}\label{eq:smg_projection_coord}
	\evalField[\pointSubCoord]{\smgReductionCoord}\left(
		\evalField[\pointSubCoord]{\vfHamCoord}
	\right)
	&=
	\invb{
		\rT{\evalField[\pointRedCoord]{\D\embCoord}}
		\evalField[{\pointSubCoord}]{\symplFormCoord}
		\evalField[\pointRedCoord]{\D\embCoord}
	}
	\underbrace{
	\rT{\evalField[\pointRedCoord]{\D\embCoord}}
		\evalField[\pointSubCoord]{\symplFormCoord}
		\evalField[\pointSubCoord]{\vfHamCoord}
	}_{
		=
		\rT{\evalField[\pointRedCoord]{\D\embCoord}}
		\rT{\evalField[{\pointSubCoord}]{\D \HamCoord}}
		=
		\rT{\evalField[{\pointRedCoord}]{\D \HamRedCoord}}
	}
	\in \R^{\dimMnfRed}.
 \end{align}

For a canonical Hamiltonian system, our generalization of the \SMG-\ROM is consistent with the definitions existing in the literature, which is shown by the following lemma.

\begin{lemma}
	\label{thm:special_case_smg}
	For a canonical Hamiltonian system $(\R^{2\dimMnf}, \rT\JtN, \HamCoord)$ and reduced symplectic manifold~$(\mnfRed, \symplFormRed) = (\R^{2\dimMnfRed}, \rT\Jtn)$, it holds that
	\begin{enumerate}
		\item the \SMG reduction evaluated at the base point $\pointSubCoord$ equals the symplectic inverse
		\begin{equation*}
		\evalField[\pointSubCoord]{\smgReductionCoord}\left(
			\velocityCoord
		\right)
		=
		\symplInv{\evalField[\pointRedCoord]{\D\embCoord}} \velocityCoord
		\vcentcolon=
		\Jtn\, \rT{\evalField[\pointRedCoord]{\D\embCoord}}\, \rT\JtN \velocityCoord
		\qquad \text{for all } \velocityCoord \in \R^{2\dimMnf},
		\end{equation*}
		\item the \SMG-\ROM is consistent with \cite{BucGH21}, and
		\item if, moreover, the embedding $\embCoord$ is linear,
		the \SMG-\ROM equals the symplectic Galerkin \ROM introduced in \cite{PenM2016,AfkH17}.
	\end{enumerate}
\end{lemma}

\begin{proof}
	By assumption, it holds $\mnf = \R^{2\dimMnf}$, $\symplFormCoord = \rT\JtN$, $\mnfRed = \R^{2\dimMnfRed}$, $\symplFormRedCoord = \rT\Jtn$.
	(i) Inserting the quantities in \eqref{eq:smg_projection_coord} yields the statement.
	(ii) For $\embCoord$ to be a symplectomorphism, i.e.,
	$
	\evalField[\pointRed]{(\pullback{\emb}\symplForm)}\,
	= \evalField[{
		\pointRed
	}]{\symplFormRed}
	$
	for all $\pointRed \in \mnfRed$,
	is with \eqref{eq:pmetric_red_coord} equivalent to
	\begin{align}\label{eq:embCoord_sympl}
		\rT{\evalField[{\pointRedCoord}]{\D\embCoord}}\;
\rT\JtN \evalField[{\pointRedCoord}]{\D\embCoord}
		= \rT\Jtn
		\qquad \text{for all }\pointRedCoord \in \R^{2\dimMnfRed}.
	\end{align}
	Considering $\rT\JtN = -\JtN$ and $\rT\Jtn = -\Jtn$ and multiplying the previous equation on both sides with $(-1)$, gives exactly the definition of a symplectic embedding from \cite[Def.~2]{BucGH21}. Thus, the assumptions on the embedding are equivalent (up to smoothness requirements).
	Moreover, the \SMG-\ROM in \cite{BucGH21} is projected with the symplectic inverse which (by point~(i)) is equivalent to the \SMG reduction map for the case assumed in the present lemma ($\mnf = \R^{2\dimMnf}$, $\symplFormCoord = \rT\JtN$, $\mnfRed = \R^{2\dimMnfRed}$, $\symplFormRedCoord = \rT\Jtn$).

	(iii) If the embedding is linear, then there exists $\fV \in \R^{2\dimMnf \times 2\dimMnfRed}$ such that $\pointSubCoord = \fV \pointRedCoord$. Then, the requirement of $\emb$ to be a symplectomorphism is equivalent to $\rT\fV \JtN \fV = \Jtn$, which is in \cite[Equation~3.2]{PenM2016} formulated as the condition that $\fV$ is a \emph{symplectic matrix}. Moreover, the symplectic inverse of $\fV$ is used to obtain the \ROM, which is, again, by point~(i), equivalent to our approach in this particular case.
\end{proof}

However, our approach extends the existing methods, as it also works (i) on general smooth manifolds (not just $\mnf = \R^{2\dimMnf}$) and (ii) even in the case $\mnf=\R^{2\dimMnf}$ for noncanonical symplectic forms $\symplFormCoord \neq \rT\JtN$. Structure-preserving \MOR for noncanonical Hamiltonian systems (for the particular case of a linear embedding) is discussed in \cite{MaboudiAfkham2018a}.
Compared to that approach, however, we use the noncanonical symplectic form prescribed by the \FOM, which generalizes the symplectic inverse straightforwardly.

It remains to show that assuming nondegeneracy of $\pullback{\emb} \symplForm$ is enough that $\pullback{\emb} \symplForm$ is a symplectic form,
which we show in the following.

\begin{lemma}\label{lem:sympl_from_nondeg}
	Consider a symplectic manifold $(\mnf, \symplForm)$,
	a smooth manifold $\mnfRed$, 
	and a smooth embedding $\emb \in \calC^{\infty} (\mnfRed, \mnf)$
	such that $\symplFormRed \vcentcolon= \pullback{\emb}{\symplForm}$ is nondegenerate.
	Then $\symplFormRed$ is a symplectic form,
	$(\mnfRed, \symplFormRed)$ is a symplectic manifold,
	and $\emb$ is a symplectomorphism.
\end{lemma}

\begin{proof}
	It is sufficient to show that $\symplFormRed = \pullback{\emb}{\symplForm}$ is a symplectic form, which in this case results in showing that $\symplFormRed$ is skew-symmetric and closed.
	The skew-symmetry is inherited for all points $\pointRed \in \mnfRed$ since with \eqref{eq:pullback}
	$$
	\tidx{(\evalField[\pointRed]{\symplFormRed})}{}{\tname{j}{1} \tname{j}{2}}
=
\tidx{(\evalField[\pointSub]{\symplForm})}{}{\tname{\ell}{1} \tname{\ell}{2}}\;
\trafoEmb{\tname{\ell}{1}}{\tname{j}{1}}
\trafoEmb{\tname{\ell}{2}}{\tname{j}{2}}
= -\tidx{(\evalField[\pointSub]{\symplForm})}{}{\tname{\ell}{2} \tname{\ell}{1}}\;
\trafoEmb{\tname{\ell}{2}}{\tname{j}{2}}
\trafoEmb{\tname{\ell}{1}}{\tname{j}{1}}
= -\tidx{(\evalField[\pointRed]{\symplFormRed})}{}{\tname{j}{2} \tname{j}{1}}.
	$$
	Closedness is inherited since the pullback of a closed form is closed again \cite[proof of Prop.~17.2]{Lee12}.
\end{proof}

Note that this is a central difference to reduced Riemannian metrics, which are automatically nondegenerate due to positive definiteness. The following example shows that the reduced tensor field can degenerate if arbitrary embeddings $\emb$ in combination with a symplectic form are considered.

\begin{example}[Example for degenerate $\pullback{\emb}\symplForm$]
	\label{ex:degenerateTensorField}
	For an arbitrary $\dimMnfRed$ with $2\dimMnfRed \leq \dimMnf$, consider $\mnf = \R^{2\dimMnf}$, $\symplFormCoord = \rT\JtN$, $\mnfRed = \R^{2\dimMnfRed}$, and the embedding
	\begin{align*}
		&\pointSubCoord = \fE \pointRedCoord&
		&\text{with}&
		\fE \vcentcolon= \begin{pmatrix}
			\fI_{2\dimMnfRed}\\
			\fzero_{2\dimMnf-2\dimMnfRed}
		\end{pmatrix}
		\in \R^{2\dimMnf \times 2\dimMnfRed}.
	\end{align*}
	In this case, it holds $\evalField[\pointRedCoord]{\D\embCoord} = \fE$ and the reduced tensor field is
	\begin{align*}
		\evalField[\pointRedCoord]{\symplFormRedCoord}
= \rT{\evalField[\pointRedCoord]{\D\embCoord}}\,
\evalField[{\pointSubCoord}]{\symplFormCoord}
\evalField[\pointRedCoord]{\D\embCoord}
= \begin{pmatrix}
	\fI_{2\dimMnfRed} & \fzero_{2\dimMnf-2\dimMnfRed}
\end{pmatrix}
\begin{pmatrix}
	\fzero_{\dimMnf} & \fI_{\dimMnf}\\
	-\fI_{\dimMnf} & \fzero_{\dimMnf}
\end{pmatrix}
\begin{pmatrix}
	\fI_{2\dimMnfRed}\\
	\fzero_{2\dimMnf-2\dimMnfRed}
\end{pmatrix}
= \fzero \in \R^{2\dimMnfRed \times 2\dimMnfRed},
	\end{align*}
	which is clearly degenerate.
\end{example}

%------------------------------------------%
% SECTION 6: SNAPSHOT-BASED GENERATION %
%------------------------------------------%
\section{Snapshot-based generation of embedding and point reduction}
\label{sec:embedding_generation}
Another key task in \MOR is the choice of a particular embedding $\emb$ (the \stepApproximation~step in \Cref{subsubsec:MORworkflow}).
In this section, we thus consider the construction of the embedding in a data-driven setting, which is directly combined with the construction of a point reduction $\redMapPoint$.
We first introduce the data-driven setting \Crefpoint{subsec:snapshotGen} and then detail four techniques to generate an embedding and a corresponding point reduction. In \Cref{tab:methods_covered}, we present an overview of selected methods discussed in the literature and how they fit into our general framework.
Throughout the section, we assume
to be given the $N$-dimensional smooth manifold $\mnf$
and a metric $\metricMnfSymbol\colon \mnf \times \mnf \to \Rpos$.

\begin{table}
	\centering
	\caption{\MOR techniques from different references that are covered by \MPG~\eqref{eq:mpg_projection}, \GMG~\eqref{eq:gmg_reduction}, \LMG~\eqref{eqn:lmg_reduction}, and \SMG~\eqref{eq:smg_reduction} introduced in our work.}
	\label{tab:methods_covered}
	\small
	\begin{tabular}{rlll}
		\toprule
		\textbf{name} & \textbf{ref.} &  & \textbf{details}\\
		\midrule
		QPROM & \cite{Barnett2022,Geelen2022,Issan2023} & \MPG & $\redCoord$ linear, $\embCoord$ quadratic\\
		EncROM & \cite{OttMR23} & \MPG & $\redCoord, \embCoord$ autoencoders\\\midrule
		qmf & \cite{Benner2022} & \GMG & $\tfCoord \equiv \fI_{\dimMnf}$, $\redCoord$ linear, $\embCoord$ quadratic\\
		manifold Galerkin & \cite{Lee20} & \GMG & $\tfCoord$ independent of $\pointCoord$, symmetric, pos.~def.\\\midrule
		& \cite{Lall2003,Carlberg2015} & \LMG & $\tfCoord$ as in \eqref{eq:lmg_metric_coord}, $\redCoord, \embCoord$ linear\\\midrule
		symplectic Galerkin & \cite{PenM2016,AfkH17} & \SMG & $\tfCoord \equiv \rT\JtN$, $\redCoord$, $\embCoord$ linear\\
		SMG & \cite{BucGH21} & \SMG & $\tfCoord \equiv \rT\JtN$, $\redCoord, \embCoord$ autoencoders\\
		SMG-QMCL & \cite{Sharma2023} & \SMG & $\tfCoord \equiv \rT\JtN$, $\redCoord$ linear, $\embCoord$ from manifold cotangent lift\\
		\bottomrule
	\end{tabular}
\end{table}

\subsection{Snapshot-based generation}
\label{subsec:snapshotGen}
In the scope of the present work, we focus on \emph{snapshot-based generation of an embedding and a point reduction}.
Consider a finite subset $\solMnfTrain \subseteq \solMnf$ of the set of all solutions $\solMnf \subseteq \mnf$ from \eqref{eq:solution_manifold}, which is referred to as the \emph{(training-)set of snapshots} and its elements $\pointTrain \in \solMnfTrain$ as \emph{snapshots}.
Typically, the embedding and the point reduction are determined by searching in a given family of functions
$$\redEmbFamily \vcentcolon= \left\{
	(\emb, \red) \in \calC^\infty(\mnfRed, \mnf) \times \calC^\infty(\mnf, \mnfRed)
	\,\big|\,
	\red \text{ is a point reduction for } \emb
	\text{ \eqref{eq:projProperty:pointProjection}}
\right\}$$
by optimizing over a functional $\loss\colon \redEmbFamily \to \Rpos$ that measures the quality of approximation based on the snapshots $\pointTrain \in \solMnfTrain$, i.e.,
\begin{align}\label{eq:red_emb_optim}
	(\emb^\star, \red^\star)
\vcentcolon= \argmin_{(\emb, \red) \in \redEmbFamily}
\evalFun[\emb, \red]{\loss}.
\end{align}
We emphasize that \Cref{lem:emb_subspaces_from_point_red} guarantees that searching within $\redEmbFamily$ automatically yields that $\emb$ is a smooth embedding and $\mnfSub$ is an embedded submanifold. Note that for practical purposes, which we do not further consider, one might want to relax the smoothness assumptions in $\redEmbFamily$.

One well-established functional is the \emph{mean squared error} (\MSE)
\begin{align}\label{eq:lossMSE}
	\evalFun[\emb, \red]{\lossMSE} \vcentcolon=
	\frac{1}{\abs{\solMnfTrain}}
	\sum_{\pointTrain \in \solMnfTrain}
	\big(
		\metricMnf{	
			\pointTrain
		}{
			{{\evalFun[\pointTrain]{(\emb \circ \red)}}}
		}
	\big)^2
	\in \Rpos.
\end{align}
The motivation of minimizing the \MSE is that if $\evalFun[\emb, \red]{\lossMSE} = 0$, it is guaranteed that all snapshots $\pointTrain \in \solMnfTrain$ are in the image of the embedding $\emb$ and thus directly lay on the embedded submanifold, i.e., $\solMnfTrain \subseteq \mnfSub$.
In general, however, the \MSE is not equal to zero.
Nevertheless, then we know that for each addend of \eqref{eq:lossMSE} it holds that
\begin{align}\label{eq:msePointwise}
	\big(
		\metricMnf{	
			\pointTrain
		}{
			{{\evalFun[\pointTrain]{(\emb \circ \red)}}}
		}
	\big)^2
	\leq \abs{\solMnfTrain} \cdot \evalFun[\emb, \red]{\lossMSE}
\end{align}
for all snapshots $\pointTrain \in \solMnfTrain$ due to non-negativity of the respective addends.

In the following we present four examples for snapshot-based generation for the case where $\mnf=\R^{\dimMnf}$, $\mnfRed=\R^{\dimMnfRed}$, $\TpMnf = \R^{\dimMnf}$, $\TpMnfRed = \R^{\dimMnfRed}$ are Euclidean vector spaces with chart mappings $\chartmap \equiv \id_{\R^{\dimMnf}}$, $\chartmapRed \equiv \id_{\R^{\dimMnfRed}}$ and the metric $\metricMnfSymbol$ is defined by a symmetric, positive-definite matrix $\pmetricCoord \in \R^{\dimMnf \times \dimMnf}$ with
\begin{align*}
	&\pmetricNormCoord{\pointCoord}
\vcentcolon= \sqrt{
	\rT{\pointCoord} \fg  \pointCoord
},&
	&\metricMnfCoord{\pointCoord}{\pointAltCoord}
= \pmetricNormCoord{\pointCoord - \pointAltCoord},&
	&\text{for $\pointCoord, \pointAltCoord \in \R^{\dimMnf}$.}
\end{align*}
With this choice, the \MSE~\eqref{eq:lossMSE} in coordinates reads
\begin{align}\label{eq:mse_in_coord}
	\evalFun[\embCoord, \redCoord]{\lossMSECoord}
= \frac{1}{\abs{\solMnfTrain}}
\sum_{\pointTrainCoord \in \solMnfTrain}
\pmetricNormCoord{
	\pointTrainCoord - \evalFun[\pointTrainCoord]{(\embCoord \circ \redCoord)}
}^2.
\end{align}
For each of the four presented approaches, we
\begin{enumerate}
	\item formulate the respective family of functions as a subset of $\redEmbFamily$,
	\item describe how the \MSE functional \eqref{eq:lossMSE} is optimized,
	\item refer to existing work that uses the respective technique.
\end{enumerate}

%-----------------------------------------------------------------------------%
\subsection{Linear subspaces}\label{subsec:classical_mor}
As discussed in \Cref{rem:subspace_MOR}, linear-subspace \MOR is included in our framework if the embedding $\embCoord$ and the point reduction $\redCoord$ are linear maps
\begin{align}\label{eq:mappings_classical}
	&\evalFun[\pointRedCoord]{\embClassicCoord}
\vcentcolon= \fV \pointRedCoord,&
	&\evalFun[\pointCoord]{\redClassicCoord}
	\vcentcolon=
	\rT\fW \pointCoord,
\end{align}
based on the matrices $\fV, \fW \in \R^{\dimMnf \times \dimMnfRed}$
with $\dimMnfRed \ll \dimMnf$. Due to the linearity of the mapping $\embClassicCoord$, we get that $\evalFun[\R^{\dimMnfRed}]{\embClassicCoord}$ is a linear subspace of $\mnf = \R^{\dimMnf}$, which is why we refer to this technique as \emph{linear-subspace \MOR}.
We formulate the respective family of functions by
\begin{align*}
	\redEmbClassicalFamily
\vcentcolon= \left\{
	(\embClassicCoord,
	\redClassicCoord)
	\text{ from \eqref{eq:mappings_classical}}
	\,\Big|\,
	\fV, \fW \in \R^{\dimMnf \times \dimMnfRed}
	\text{ such that }
	\rT\fW \fV = \tname{\fI}{\dimMnfRed}
\right\}.
\end{align*}

\begin{proposition}
	The family of functions $\redEmbClassicalFamily$ is a subset of $\redEmbFamily$.
\end{proposition}
\begin{proof}
	The assumption $\rT\fW \fV = \tname{\fI}{\dimMnfRed}$ implies the point projection property \eqref{eq:projProperty:pointProjection} with
	\begin{align*}
		&\evalFun[\pointRedCoord]{(\redClassicCoord \circ \embClassicCoord)}
		=
		\rT\fW \fV \pointRedCoord
		= \pointRedCoord.\qedhere
	\end{align*}
\end{proof}

In the case of a \emph{Galerkin projection}, i.e., $\rT\fW = \rT\fV \pmetricCoord$,
solving \eqref{eq:red_emb_optim} for $\redEmbClassicalFamily$ simplifies to determine the basis matrix
\begin{align}\label{eq:pod_functional}
	&\fV^\ast
	= \argmin_{\substack{
		\fV \in \R^{\dimMnf \times \dimMnfRed}\\
		\rT\fV \pmetricCoord \fV = \tname{\fI}{\dimMnfRed}
	}}
	\sum_{\pointTrainCoord \in \solMnfTrain}
	\pmetricNormCoord{
		(\tname{\fI}{\dimMnf} - \fV \rT\fV \pmetricCoord ) \pointTrainCoord
	}^2,
\end{align}
which is known as the \emph{proper orthogonal decomposition} (\POD).
An optimal solution can be computed with the truncated singular value decomposition (see, e.g., \cite{Volkwein2013}).

For structure-preserving linear-subspace \MOR techniques, $\redEmbClassicalFamily$ may have to be restricted to a class that preserves the respective structure. For example, for the structure-preserving \MOR for Hamiltonian systems, the \SMG is used, which assumes $\embCoord$ to be a symplectomorphism and $\fW = \JtN \fV \rT\Jtn$ is the so-called symplectic inverse. Such \MOR techniques are used, e.g., in \cite{PenM2016,AfkH17,Buchfink2019}.

%-----------------------------------------------------------------------------%
\subsection{Quadratic manifolds} \label{subsec:MOR_quadratic_manifolds}
Recently, so-called \emph{\MOR on quadratic manifolds} has become an active field of research \cite{Barnett2022,Geelen2022,Benner2022,Issan2023,Sharma2023}.
In our terms,
the embedding and point reduction are set to
\begin{align}\label{eq:mappings_quad}
	&\evalFun[\pointRedCoord]{\embQuadCoord}
\vcentcolon= \tname{\fA}{2} \quadKron{\pointRedCoord} + \tname{\fA}{1} \pointRedCoord + \tname{\fA}{0},&
	\redQuadCoord(\pointCoord)
\vcentcolon= \rT{\tname{\fA}{1}} \left( \pointCoord - \tname{\fA}{0} \right),
\end{align}
with $\tname{\fA}{2} \in \R^{\dimMnf \times {\dimMnfRed(\dimMnfRed+1)/2}}$, $\tname{\fA}{1} \in \R^{\dimMnf \times \dimMnfRed}$, $\tname{\fA}{0} \in \R^{\dimMnf}$. By $\quadKron{(\cdot)}\colon \R^{\dimMnfRed} \to \R^{\dimMnfRed(\dimMnfRed+1)/2}$, $\pointRedCoord \mapsto \quadKron{\pointRedCoord}$, we denote the \emph{symmetric Kronecker product}, which produces all pairwise products of components $\unstack{\pointRedCoord}{i}{} \in \R$ of $\pointRedCoord$ for $1 \leq i \leq \dimMnfRed$ while neglecting redundant entries, i.e.,
\begin{equation*}
\quadKron{\pointRedCoord}
=
\rT{\stack{
	\unstack{\pointRedCoord}{1}{}
	\cdot \unstack{\pointRedCoord}{1}{},\;
	\unstack{\pointRedCoord}{1}{}
	\cdot \unstack{\pointRedCoord}{2}{},\;
	\unstack{\pointRedCoord}{2}{}
	\cdot \unstack{\pointRedCoord}{2}{},\;
	\dots,\;
	\unstack{\pointRedCoord}{\dimMnfRed}{}
	\cdot \unstack{\pointRedCoord}{\dimMnfRed}{}
}{}}
\in \R^{\dimMnfRed(\dimMnfRed+1)/2}.
\end{equation*}
The respective family of functions is
\begin{align}\label{eq:family_quad}
	\redEmbQuadFamily
\vcentcolon= \left\{
	(\embQuadCoord,
	\redQuadCoord)
	\text{ from \eqref{eq:mappings_quad}}
	\,\Big|\,
	\rT{\tname{\fA}{1}} \tname{\fA}{1} = \tname{\fI}{\dimMnfRed}
	\text{ and }
	\rT{\tname{\fA}{1}} \tname{\fA}{2} = \tname{\fzero}{\dimMnfRed \times \dimMnfRed(\dimMnfRed+1)/2}
\right\}.
\end{align}
\begin{proposition}
	The family $\redEmbQuadFamily$ is a subset of $\redEmbFamily$.
\end{proposition}
\begin{proof}
	The assumptions \eqref{eq:family_quad} on $\tname{\fA}{2}$ and $\tname{\fA}{1}$ imply the point projection property \eqref{eq:projProperty:pointProjection},
	\begin{align*}
		\evalFun[\pointRedCoord]{(\redQuadCoord \circ \embQuadCoord)}
	= \rT{\tname{\fA}{1}} \left(
		\tname{\fA}{2} \quadKron{\pointRedCoord}
		+ \tname{\fA}{1} \pointRedCoord
		+ \tname{\fA}{0}
		- \tname{\fA}{0}
	\right)
	\stackrel{\eqref{eq:family_quad}}{=} \pointRedCoord.\tag*{\qedhere}
	\end{align*}
\end{proof}

The matrices $\tname{\fA}{1}$ and $\tname{\fA}{2}$ are obtained in \cite{Barnett2022,Geelen2022,Benner2022,Issan2023,Sharma2023} from the \MSE functional \eqref{eq:mse_in_coord}.
In this setting, the assumptions in \eqref{eq:family_quad} allow to determine the matrices $\tname{\fA}{0}$, $\tname{\fA}{1}$ and $\tname{\fA}{2}$ sequentially:
First, $\tname{\fA}{0}$ is chosen, e.g., as the mean of $\solMnfTrain$.
Then,~\eqref{eq:mse_in_coord} is optimized for~$\tname{\fA}{1}$ (similarly to \eqref{eq:pod_functional}).
Finally,~\eqref{eq:mse_in_coord} is optimized for~$\tname{\fA}{2}$, which results in a (regularized) linear least squares problem.
The preceding papers use different tangent reductions to derive the \ROM,
which can be classified with the framework introduced in the present paper:
\cite{Barnett2022,Geelen2022,Issan2023} use the \MPG reduction map \eqref{eq:mpg_projection},
while \cite{Benner2022} relies on the \GMG reduction map \eqref{eq:gmg_reduction} (but neglects a few higher order terms).
The major difference between using \MPG or \GMG in that context is that the \MPG projects the \FOM vector field with the tangent reduction
$$
\redMapTanMPGCoord{{\evalFun[\pointRedCoord]{\embQuadCoord}}}
= \evalField[
	{{\evalFun[\pointRedCoord]{\embQuadCoord}}}
]{\D\redQuadCoord}
= \rT{\tname{\fA}{1}}
$$ which is constant,
while the \GMG uses the tangent reduction from \eqref{eq:gmg_projection_coord} with $\evalField[\pointCoord]{\tfCoord} = \fI_{\dimMnf}$ for all $\pointCoord \in \R^{\dimMnf}$
and $\evalField[\pointRedCoord]{\D\embCoord} = \tname{\fA}{2} \fB_2(\pointRedCoord) + \tname{\fA}{1}$ with a linear function $\fB_2\colon \R^{\dimMnfRed} \to \R^{\dimMnfRed(\dimMnfRed+1)/2 \times \dimMnfRed}$ describing the derivative of $\quadKron{(\cdot)}$, resulting in
\begin{align*}
	\redMapTanGMGCoord{
		{\evalFun[\pointRedCoord]{\embQuadCoord}}
	}
	&= \invb{
		\rT{\evalField[\pointRedCoord]{\D\embCoord}}\;
		\evalField[
			{\evalFun[\pointRedCoord]{\embQuadCoord}}
		]{\tfCoord}
		\evalField[\pointRedCoord]{\D\embCoord}
	}
	\rT{\evalField[\pointRedCoord]{\D\embCoord}}\;
	\evalField[
		{\evalFun[\pointRedCoord]{\embQuadCoord}}
	]{\tfCoord}\\
	&=
\invb{
	\fI_{\dimMnfRed}
	+ \rTb{\fB_2(\pointRedCoord)} \rT{\tname{\fA}{2}} \tname{\fA}{2} \fB_2(\pointRedCoord)
}
\rTb{
	\tname{\fA}{2} \fB_2(\pointRedCoord) + \tname{\fA}{1}
},
\end{align*}
which is typically nonlinear in $\pointRedCoord$, and, thus, so is the reduced vector field in general.

In \cite{Sharma2023}, structure-preserving \MOR of Hamiltonian systems on quadratic manifolds is investigated.
Two approaches are presented and compared:
(i) The \emph{blockwise quadratic} approach uses an embedding of a comparable structure as \eqref{eq:mappings_quad} in combination with the \MPG tangent reduction.
In contrast, (ii) the \emph{quadratic manifold cotangent lift}
uses the \SMG-ROM.
In order to construct a symplectomorphism from a quadratic embedding,
the so-called proper symplectic decomposition cotangent lift from \cite{PenM2016} (which generates a linear embedding $\embCoord$) is generalized to the case of nonlinear embeddings $\embCoord$ by introducing the so-called \emph{manifold cotangent lift}.
Based on this idea, the authors construct an embedding $\embCoord: \R^{2\dimMnfRed} \to \R^{2\dimMnf}$,
where the first $\dimMnf$ component functions are of the structure \eqref{eq:mappings_quad} and the last $\dimMnf$ component functions are rational functions.
The \SMG (as a special case of the \GMG) is then used for a structure-preserving tangent reduction of the Hamiltonian vector field.

%-----------------------------------------------------------------------------%
\subsection{Nonlinear compressive approximation}
Following the idea of the previous subsection,
the embedding and the point reduction can be defined more generally with
\begin{align}\label{eq:mappings_NCA}
	&\evalFun[\pointRedCoord]{\embNCACoord}
\vcentcolon= \fA_2 \evalFun[\pointRedCoord]{\ff}
+ \fA_1 \pointRedCoord
+ \fA_0,&
	\evalFun[\pointCoord]{\redNCACoord}
\vcentcolon= \rT\fB \left( \pointCoord - \fA_0 \right),
\end{align}
where $\tname{\fA}{2} \in \R^{\dimMnf \times {\dimMnfIntermed}}$,
$\tname{\fA}{1} \in \R^{\dimMnf \times \dimMnfRed}$,
$\tname{\fA}{0} \in \R^{\dimMnf}$,
$\fB \in \R^{\dimMnf \times \dimMnfRed}$
are matrices,
and $\ff \in \calC^\infty (\R^{\dimMnfRed}, \R^{\dimMnfIntermed})$ is a smooth nonlinear mapping for a given $\dimMnfIntermed \in \N$.
Following \cite{Cohen2023},
we refer to this approach as \emph{nonlinear compressive approximation} (\NCA).
The respective family of functions is
\begin{align}\label{eq:family_NCA}
	\redEmbNCAFamily
\vcentcolon= \left\{
	(\embNCACoord,
	\redNCACoord)
	\text{ from \eqref{eq:mappings_NCA}}
	\,\Big|\,
	\rT\fB \fA_1 = \fI_{\dimMnfRed},\;
	\rT\fB \fA_2 = \tname{\fzero}{\dimMnfRed \times \dimMnfIntermed}
\right\}.
\end{align}

\begin{proposition}
	The family $\redEmbNCAFamily$ is a subset of $\redEmbFamily$.
\end{proposition}

\begin{proof}
	The assumptions on $\tname{\fA}{1}$, $\tname{\fA}{2}$, and $\fB$ in \eqref{eq:family_NCA} imply the point projection property \eqref{eq:projProperty:pointProjection} with
	\begin{align*}
		\evalFun[{\pointRedCoord}]{(\redNCACoord \circ \embNCACoord)}
	= \rT\fB \left(
		\fA_2 \evalFun[\pointRedCoord]{\ff}
		+ \fA_1 \pointRedCoord
		+ \fA_0
		- \fA_0
	\right)
	= \pointRedCoord.\tag*{\qedhere}
	\end{align*}
\end{proof}

The \MSE for this approach may be optimized sequentially as in the \MOR on quadratic manifolds discussed in the previous section using $\fB = \fA_1$.
This method is, e.g., used in \cite{Barnett2023}, where $\ff$ is a neural network.

Multiple works investigate \NCA:
First, the quadratic embedding \eqref{eq:mappings_quad} discussed in the previous subsection is a special case of the \NCA when choosing $\evalFun[\pointRedCoord]{\ff} = \quadKron{\pointRedCoord}$.
Similarly, $\ff$ can be chosen as a higher-order polynomial in $\pointRedCoord$ to obtain a more general approximation.
Second, in \cite{Barnett2023}, $\ff$ is learned with an artificial neural network, while a time-discrete setting is considered for the reduction, which is not covered by our methods.
Third, \cite{Cohen2023} analyzes the approximation of a set of traveling wave solutions with (and without) varying support on the \PDE level using
decision trees and random forests in their numerical experiments. Interestingly, the authors show that a linear point reduction is enough to reproduce the set of traveling wave solutions.
Fourth, in \cite{BucGH23}, it is shown that the \NCA has its limitations in terms of the Kolmogorov $(\dimMnfIntermed + \dimMnfRed)$-width since the solution is contained in an~$(\dimMnfIntermed + \dimMnfRed)$-dimensional linear subspace of $\R^{\dimMnf}$.

%-----------------------------------------------------------------------------%
\subsection{Autoencoders}
\label{subsec:autoencoders}
Autoencoders are a well-known technique in nonlinear dimension reduction (see, e.g., \cite[Cha.~14]{Goodfellow2016}).
In the understanding of the present work, \MOR with autoencoders chooses
\begin{align}\label{eq:mappings_ae}
	&\embAeCoord \in \calC^\infty(\R^{\dimMnfRed}, \R^{\dimMnf}),&
	&\redAeCoord \in \calC^\infty(\R^{\dimMnf}, \R^{\dimMnfRed}),
\end{align}
where both functions are artificial neural networks (\ANNs) with network parameters $\networkParam \in \R^{\numNetworkParam}$ (like weights and biases).
Since $\redAeCoord\colon \R^{\dimMnf} \to \R^{\dimMnfRed}$ and $\embAeCoord\colon \R^{\dimMnfRed} \to \R^{\dimMnf}$,
the concatenation $\embAeCoord \circ \redAeCoord $ maps from $\R^\dimMnf$ over $\R^{\dimMnfRed}$ back to $\R^\dimMnf$.
The in-between compression to $\R^{\dimMnfRed}$ is typically referred to as the \emph{bottleneck}, $\redAeCoord$ as the \emph{encoder}, $\embAeCoord$ as the \emph{decoder}, and the concatenation $\embAeCoord \circ \redAeCoord $ as an \emph{autoencoder}. The respective family is
\begin{align*}
	\redEmbAeFamily
\vcentcolon= \left\{
	(\embAeCoord, \redAeCoord) \text{ from \eqref{eq:mappings_ae}}
	\mid
	\networkParam \in \R^{\numNetworkParam} \text{ network parameters}
\right\}.
\end{align*}

Without special assumptions about the architecture of the \ANNs,
it is generally impossible to show the point projection property \eqref{eq:projProperty:pointProjection}.
However, whenever the minimum of the cost functional~\eqref{eq:red_emb_optim} is small, then we show in the following that the point projection property \eqref{eq:projProperty:pointProjection} holds approximately.
We assume to be given a norm $\pmetricRedNormCoord{\cdot}\colon \R^{\dimMnfRed} \to \Rpos$ such that $\embAeCoord$ and $\redAeCoord$ are \emph{Lipschitz continuous}, i.e., there exists a constant $\embLipschitzConst \geq 0$ such that for all points $\pointRedCoord, \pointRedAltCoord \in \R^{\dimMnfRed}$
\begin{align}
	\label{eq:embLipschitz}
	\pmetricNormCoord{	
		{\evalFun[\pointRedCoord]{{\embAeCoord}}}
	-
		{\evalFun[\pointRedAltCoord]{{\embAeCoord}}}
	}
	\leq&\; \embLipschitzConst
	\pmetricRedNormCoord{	
		\pointRedCoord
	-
		\pointRedAltCoord
	}
\end{align}
and a constant $\redLipschitzConst \geq 0$ such that for all points $\pointCoord, \pointAltCoord \in \R^{\dimMnf}$
\begin{align}
	\label{eq:redLipschitz}
	\pmetricRedNormCoord{	
		{\evalFun[\pointCoord]{{\redAeCoord}}}
	-
		{\evalFun[\pointAltCoord]{{\redAeCoord}}}
	}
	\leq&\; \redLipschitzConst
	\pmetricNormCoord{	
		\pointCoord
	-
		\pointAltCoord
	}.
\end{align}

\begin{theorem}\label{lem:autoencoder_inverseProperty}
	For a given tuple $(\embAeCoord, \redAeCoord) \in \redEmbAeFamily$ from the family of functions for \MOR with autoencoders with an \MSE value of $\evalFun[\embAeCoord, \redAeCoord]{\lossMSE} \geq 0$,
	the point projection property \eqref{eq:projProperty:pointProjection} is fulfilled approximately in the sense that for each $\pointRedCoord \in \R^{\dimMnfRed}$
	\begin{align*}
		\pmetricRedNormCoord{
			{\evalFun[\pointRedCoord]{{(\redAeCoord \circ \embAeCoord)}}}
			-
			\pointRedCoord
		}
		&\leq
		\redLipschitzConst
		\sqrt{
			\abs{\solMnfTrain}\evalFun[\embAeCoord, \redAeCoord]{\lossMSE}
		}\\
		&\phantom{\leq} + (\redLipschitzConst \embLipschitzConst + 1)
		\min_{\pointTrainAltCoord \in \solMnfTrain} \pmetricRedNormCoord{
			\pointRedCoord - \evalFun[\pointTrainAltCoord]{{\redAeCoord}}
		}
	\end{align*}
\end{theorem}

Thus, for a bounded set $\reduce{M} \subsetneq \R^{\dimMnfRed}$, a fine sampling in~$\solMnfTrain$ of $\reduce{M}$ and small values of the \MSE functional such that the term $\abs{\solMnfTrain} \evalFun[\embAeCoord, \redAeCoord]{\lossMSE}$ is small, the point projection property \eqref{eq:projProperty:pointProjection} holds approximately on $\reduce{M}$, i.e., $\restrict{\redAeCoord \circ \embAeCoord}{\reduce{M}} \approx \id_{\reduce{M}}$.

\begin{proof}
	The proof is split in two parts.
	In the first part, we show that the inequality holds in the encoded training points $\pointRedAltCoordTrain \vcentcolon= \evalFun[\pointTrainAltCoord]{\redAeCoord}$ with $\pointTrainAltCoord \in \solMnfTrain$.
	Then, we show that the inequality holds for general $\pointRedCoord \in \R^{\dimMnfRed}$ by applying Lipschitz continuity.

	Consider a fixed but arbitrary training point $\pointTrainAltCoord \in \solMnfTrain$. Using \eqref{eq:redLipschitz} and \eqref{eq:msePointwise}, it holds
	\begin{align*}
		\pmetricRedNormCoord{
			{\evalFun[\pointRedAltCoordTrain]{{(\redAeCoord \circ \embAeCoord)}}}
			-
			\pointRedAltCoordTrain
		}
		&=
		\pmetricRedNormCoord{
			{\evalFun[\pointTrainAltCoord]{{(\redAeCoord \circ \embAeCoord \circ \redAeCoord)}}}
			-
			{\evalFun[\pointTrainAltCoord]{{\redAeCoord}}}
		}\\
		&\leq
		\redLipschitzConst
		\pmetricNormCoord{
			{\evalFun[\pointTrainAltCoord]{{(\embAeCoord \circ \redAeCoord)}}}
			-
			\pointTrainAltCoord
			}
		\\
		&\leq
		\redLipschitzConst
		\sqrt{
			\abs{\solMnfTrain}\evalFun[\embAeCoord, \redAeCoord]{\lossMSE}
		}.
	\end{align*}

	This property can be generalized to points $\pointRedCoord \in \R^{\dimMnfRed} \setminus \evalFun[\solMnfTrain]{\redAeCoord}$.
	The idea is that for a general Lipschitz continuous function $\ff: \R^{\dimMnfRed} \to \R^{\dimMnfRed}$ with Lipschitz constant $C \geq 0$ and $\pmetricRedNormCoord{\evalFun[\pointRedAltCoordTrain]{\ff}} \leq C_{B}$ for some $C_{B} \geq 0$,
	it holds for all $\pointRedCoord \in \R^{\dimMnfRed}$ by adding a zero, triangle inequality, and Lipschitz continuity
	\begin{align*}
		\pmetricRedNormCoord{\evalFun[\pointRedCoord]{\ff}}
\leq \pmetricRedNormCoord{\evalFun[\pointRedAltCoordTrain]{\ff}}
+ \pmetricRedNormCoord{
	\evalFun[\pointRedCoord]{\ff} - \evalFun[\pointRedAltCoordTrain]{\ff}
}
\leq C_{B}
+ C \pmetricRedNormCoord{
	\pointRedCoord - \pointRedAltCoordTrain
}
.
	\end{align*}
	For our case, we use $\evalFun[\pointRedCoord]{\ff} = \evalFun[\pointRedCoord]{(\redAeCoord \circ \embAeCoord)} - \pointRedCoord$
	with Lipschitz constant $C = \redLipschitzConst \embLipschitzConst + 1$,
	bound $C_{B}=\redLipschitzConst\sqrt{\abs{\solMnfTrain}\evalFun[\embAeCoord, \redAeCoord]{\lossMSE}}$,
	and points $\pointRedAltCoordTrain = \evalFun[\pointTrainAltCoord]{\redAeCoord}$. Thus, it holds
	\begin{align*}
		\pmetricRedNormCoord{
			{\evalFun[\pointRedCoord]{{(\redAeCoord \circ \embAeCoord)}}}
			-
			\pointRedCoord
		}
		\leq&\;
		\redLipschitzConst
		\sqrt{
			\abs{\solMnfTrain}\evalFun[\embAeCoord, \redAeCoord]{\lossMSE}
		}\\
		&+ (\redLipschitzConst \embLipschitzConst + 1)
		\pmetricRedNormCoord{
			\pointRedCoord - \evalFun[\pointTrainAltCoord]{{\redAeCoord}}
		}
	\end{align*}
	Taking the minimum over all $\pointTrainAltCoord \in \solMnfTrain$ on the right-hand side yields the result.
\end{proof}

\begin{remark}[Constrained autoencoders]
	In \cite{OttMR23}, the authors introduce a novel autoencoder architecture, which aims at fulfilling the point projection property \eqref{eq:projProperty:pointProjection} exactly.
	The architecture of the encoder is chosen to invert the decoder layer-wise based on the assumption that the linear layers of the decoder and encoder are pairwise biorthogonal.
	This biorthogonality is penalized in the loss functional with an additional term. 
\end{remark}

One of the first works to combine autoencoders with projection-based \MOR is \cite{Lee20}.
As discussed in \Cref{subsec:GMG}, the time-continuous formulation in that work considers the \GMG reduction for a state-independent Riemannian metric. 
A structure-preserving formulation for Hamiltonian systems in combination with autoencoders is discussed in \cite{BucGH21}.
As shown in \Cref{thm:special_case_smg},
this work is based on the \SMG reduction.

%------------------------------------------%
% SECTION 7: CONCLUSION %
%------------------------------------------%
\section{Conclusions}
\label{sec:conclusions}

This work proposed a differential geometric framework for \MOR on manifolds in order to analyze two important efforts in \MOR jointly:
(i) The use of nonlinear projections and (ii) structure preservation.
The key ingredient for our framework is an embedding for a low-dimensional submanifold and a compatible reduction map.
The joint abstraction allowed us to derive shared theoretical properties, such as the exact reproduction result.
As two possible reduction mappings, we discussed
(a) the \emph{manifold Petrov--Galerkin} (\MPG) using the differential of the point reduction and
(b) the \emph{generalized manifold Galerkin} (\GMG), which is based on a nondegenerate tensor field.
Moreover, we showed that structure-preserving \MOR on manifolds for Lagrangian and Hamiltonian systems can be accomplished by choosing specific nondegenerate tensor fields in the \GMG reduction map, 
which we refer to as the \emph{Lagrangian manifold Galerkin} (\LMG) and \emph{symplectic manifold Galerkin} (\SMG).
In order to connect our framework to existing work in the field,
we demonstrated how different techniques for data-driven construction of the embedding and point reduction map are reflected in our approach.
We discussed four approximation types (linear, quadratic, nonlinear compressive, and autoencoders) and linked each of these types to multiple existing works in the field.

We believe that our framework can be extended in several regards:
First, other structure-preserving \MOR techniques might be formulated in the framework, such as Poisson systems \cite{HesP18},
port-Hamiltonian (descriptor) systems \cite{SchJ14,MehU23},
or (port-)metriplectic systems \cite{Morrison1986,Gruber2023,Hernandez2023}.
Thereby, all the described nonlinear projections (quadratic, nonlinear compressive, autoencoders) become available for such structured systems.
Second, as the high-dimensional differential equations we consider as \FOM are often obtained from the semi-discretization of systems of partial differential equations (\PDEs), the framework should be extended to the \PDE level. Such a formulation would forward the structures formulated on the \PDE level to the \ODE level.

\textbf{Acknowledgements} PB, BH, and BU are funded by Deutsche Forschungsgemeinschaft (DFG, German Research Foundation) under Germany's Excellence Strategy - EXC 2075 - 390740016 and acknowledge the support by the Stuttgart Center for Simulation Science (SimTech). PB and BH acknowledge the funding of DFG Project No.~314733389. BU acknowledges funding by the BMBF (grant no.~05M22VSA).

%-----------------------------------------------------------------------------%
\bibliographystyle{plain-doi}
\bibliography{journalAbbr,literature}

\appendix
%-----------------------------------------------------------------------------%
\section{Topological spaces and topological manifolds}
\subsection{Fundamentals}
\label{appx:topology}
Consider a set $\calM$.
A \emph{topology on $\calM$} is a collection $\calT$
of subsets of $\calM$ (which are called \emph{open subsets of $\calM$}) that satisfy that
(i) both the empty set $\emptyset$ and the set itself $\calM$ are open,
(ii) each union of open subsets is open,
(iii) each intersection of finitely many open subsets is open.
The pair $(\calM,\calT)$ is called a \emph{topological space}. If the specific topology is clear from the context or not particularly relevant for the discussion, then we simply write $\calM$ instead of $(\calM,\calT)$.

For two topological spaces $\calM$ and $\mnfAlt$, a map $F\colon \calM \to \mnfAlt$ is called \emph{continuous}, if for every open subset $V \subseteq \mnfAlt$, the \emph{preimage} $\left\{ \point \in \calM \mid \evalFun[\point]{F} \in V \right\}$ is open in $\calM$.
We call $F$ a \emph{homeomorphism}, if (i) it is bijective (and thus the \emph{inverse} $\inv{F}\colon \mnfAlt \to \calM$ exists) and (ii) both $F$ and $\inv{F}$ are continuous.
Correspondingly, two topological spaces $\calM$ and $\mnfAlt$ are called \emph{homeomorphic} if there exists a homeomorphism from $\calM$ to $\mnfAlt$.
Moreover, $\calM$ is called \emph{locally homeomorphic to} $\R^{\dimMnf}$ for $\dimMnf\in\N$ if for every point $\point \in \calM$ there exists an open set $\chartdomain \subseteq \calM$ with $\point \in \chartdomain$, which is homeomorphic to an open subset of $\R^{\dimMnf}$.

A topological space $\mnf$ is called a \emph{topological manifold} of \emph{dimension} $\dimMnf$ if it is locally homeomorphic to $\R^{\dimMnf}$ (and additionally Hausdorff and second-countable, see e.g.\ \cite[Cha.~1 and App.~A]{Lee12}).

\subsection{Proof of \texorpdfstring{\Cref{lem:emb_subspaces_from_point_red}}{Lemma 2.1}}
\label{appx:proof_emb_subspaces_from_point_red}
\input{proof_emb_subspaces_from_point_red.tex}

%-----------------------------------------------------------------------------%
\section{Proofs for Lagrangian systems}
\subsection{Derivation of the reduced Euler-Lagrange equations}
\label{appx:proof_red_euler_lagrange}
\input{proof_red_euler_lagrange.tex}

\subsection{Proof of \texorpdfstring{\Cref{thm:equivalence_solve_lmg}}{Theorem~4.19}}
\label{appx:proof_equivalence_solve_lmg}
\input{proof_equivalence_solve_lmg}

%-----------------------------------------------------------------------------%
\end{document}

%% file: def.tex
\newcommand{\Crefpoint}[1]{(\seeSymbol \Cref{#1})}

\newcommand{\N}{\ensuremath\mathbb{N}}

\newcommand{\R}{\ensuremath\mathbb{R}}

\newcommand{\rT}[1]{#1^\top}
\newcommand{\rTb}[1]{\rT{\left( #1 \right)}}

\newcommand{\symplInv}[1]{#1^{+}}

\newcommand{\norm}[1]{\left\lVert #1 \right\rVert}

\newcommand{\abs}[1]{\left|#1\right|}

\newcommand{\sdd}[1]{\ensuremath{\tfrac{\mathrm{d^2}}{\mathrm{d}{#1}^2}}}
\newcommand{\dd}[1]{\ensuremath{\tfrac{\mathrm{d}}{\mathrm{d}{#1}}}}
\newcommand{\deldel}[1]{\ensuremath{\tfrac{\partial}{\partial{#1}}}}
\newcommand{\ddt}{\ensuremath{\dd{t}}}
\newcommand{\sddt}{\ensuremath{\sdd{t}}}

\newcommand{\inv}[1]{#1^{-1}}
\newcommand{\invb}[1]{\inv{\left( #1 \right)}}

\newcommand{\sq}[1]{#1^{1/2}}

\DeclareMathOperator{\id}{id}

\DeclareMathOperator{\real}{Re}

\DeclareMathOperator{\disjointUnion}{\dot\bigcup}
\DeclareMathOperator*{\argmin}{arg\,min}

\newcommand{\calA}{\mathcal{A}}

\newcommand{\calC}{\mathcal{C}}

\newcommand{\calE}{\mathcal{E}}
\newcommand{\calF}{\mathcal{F}}

\newcommand{\calI}{\mathcal{I}}

\newcommand{\calL}{\mathcal{L}}
\newcommand{\calM}{\mathcal{M}}

\newcommand{\calP}{\mathcal{P}}
\newcommand{\calQ}{\mathcal{Q}}

\newcommand{\calT}{\mathcal{T}}

\newcommand{\fd}{\ensuremath{\bm{d}}}
\newcommand{\fe}{\ensuremath{\bm{e}}}
\newcommand{\ff}{\ensuremath{\bm{f}}}
\newcommand{\fg}{\ensuremath{\bm{g}}}

\newcommand{\fm}{\ensuremath{\bm{m}}}

\newcommand{\fq}{\ensuremath{\bm{q}}}

\newcommand{\fv}{\ensuremath{\bm{v}}}

\newcommand{\fx}{\ensuremath{\bm{x}}}

\newcommand{\fA}{\ensuremath{\bm{A}}}
\newcommand{\fB}{\ensuremath{\bm{B}}}

\newcommand{\fD}{\ensuremath{\bm{D}}}
\newcommand{\fE}{\ensuremath{\bm{E}}}
\newcommand{\fF}{\ensuremath{\bm{F}}}

\newcommand{\fI}{\ensuremath{\bm{I}}}

\newcommand{\fL}{\ensuremath{\bm{L}}}

\newcommand{\fR}{\ensuremath{\bm{R}}}

\newcommand{\fV}{\ensuremath{\bm{V}}}
\newcommand{\fW}{\ensuremath{\bm{W}}}
\newcommand{\fX}{\ensuremath{\bm{X}}}
\newcommand{\fY}{\ensuremath{\bm{Y}}}

\newcommand{\ftheta}{\ensuremath{\bm{\theta}}}

\newcommand{\frho}{\ensuremath{\bm{\varrho}}}

\newcommand{\fzero}{\ensuremath{\bm{0}}}

\newcommand{\stepApproximation}{\textcaps{Approximation}}
\newcommand{\stepReduction}{\textcaps{Reduction}}
\newcommand{\stepReconstruction}{\textcaps{Reconstruction}}

\newcommand{\reduce}[1]{\check{#1}}

\newcommand{\restrict}[2]{\left.#1\right|_{#2}}

\newcommand{\evalFun}[2][\point]{#2\left(#1\right)}
\newcommand{\evalField}[2][\point]{\mleft.\kern-\nulldelimiterspace{#2}\mright|_{#1}}

\newcommand{\stack}[2]{\left[ #1 \right]_{#2}}
\newcommand{\unstack}[3]{\left[ #1 \right]^{#2}_{#3}}

\newcommand{\D}{\fD}

\newcommand{\param}{\mu}

\newcommand{\paramSet}{\calP}
\newcommand{\paramDim}{{n_p}}

\newcommand{\point}{m}
\newcommand{\pointAlt}{w}

\newcommand{\pointRed}{\reduce{\point}}

\newcommand{\pointCoord}{\fm}
\newcommand{\pointAltCoord}{\bm{\pointAlt}}
\newcommand{\pointRedCoord}{\reduce{\pointCoord}}
\newcommand{\pointRedAltCoord}{\reduce{\pointAltCoord}}
\newcommand{\pointSub}{{\evalFun[{\pointRed}]{\emb}}}
\newcommand{\pointSubCoord}{{\evalFun[{\pointRedCoord}]{\embCoord}}}
\newcommand{\pointTanSpace}{Y}
\newcommand{\pointTanSpaceCoord}{\fY}

\newcommand{\mnf}{\calM}
\newcommand{\mnfAlt}{\calQ}
\newcommand{\mnfAltAlt}{\calL}

\newcommand{\dimMnf}{N}
\newcommand{\dimMnfAlt}{Q}
\newcommand{\dimMnfAltAlt}{L}
\newcommand{\dimMnfRed}{n}

\newcommand{\chartmap}{x}
\newcommand{\chartmapRed}{\reduce{\chartmap}}
\newcommand{\chartmapAlt}{y}
\newcommand{\chartmapMnfAlt}{y}
\newcommand{\chartmapMnfAltAlt}{z}
\newcommand{\chartdomain}{U}
\newcommand{\chartdomainAlt}{V}
\newcommand{\chartdomainMnfAlt}{V}
\newcommand{\chartdomainMnfAltAlt}{W}
\newcommand{\chartdomainRed}{\reduce{\chartdomain}}

\newcommand{\chart}{(\chartdomain, \chartmap)}
\newcommand{\chartAlt}{(\chartdomainAlt,\chartmapAlt)}
\newcommand{\chartMnfAlt}{(\chartdomainMnfAlt,\chartmapMnfAlt)}
\newcommand{\chartMnfAltAlt}{(\chartdomainMnfAltAlt, \chartmapMnfAltAlt)}

\newcommand{\chartmapTanSpace}{\xi}
\newcommand{\chartmapTanSpaceRed}{\reduce{\chartmapTanSpace}}

\newcommand{\tangentT}{T}
\newcommand{\cotangentT}{\tangentT^{\ast}}
\newcommand{\T}[1]{\tangentT{#1}}
\newcommand{\natProj}{\pi}
\newcommand{\Tp}[2][\point]{\tangentT_{#1}#2}
\newcommand{\TpMnf}[1][\point]{\Tp[#1]{\mnf}}%for lazyness
\newcommand{\TpMnfAlt}[1][\pointMnfAlt]{\Tp[#1]{\mnfAlt}}%for lazyness
\newcommand{\TpMnfRed}[1][\pointRed]{\Tp[#1]{\mnfRed}}%for lazyness
\newcommand{\TpMnfSub}[1][\pointSub]{\Tp[#1]{\big(\mnfSub\big)}}%for lazyness
\newcommand{\velocity}{v}
\newcommand{\velocityCoord}{\bm{\velocity}}
\newcommand{\Finfty}[1][\point]{{F^\infty_{#1}}}

\newcommand{\coT}[1]{\cotangentT{#1}}
\newcommand{\coTp}[2][\point]{\cotangentT_{#1}#2}
\newcommand{\coTpMnf}[1][\point]{\coTp[#1]{\mnf}}%for lazyness
\newcommand{\coTpMnfAlt}[1][\point]{\coTp[#1]{\mnfAlt}}%for lazyness
\newcommand{\coTpMnfRed}[1][\pointRed]{\coTp[#1]{\mnfRed}}%for lazyness
\newcommand{\basisTan}[1]{\deldel{\tidx{\chartmap}{#1}{}}}
\newcommand{\basisTanAt}[2][\point]{\evalField[#1]{\basisTan{#2}}}
\newcommand{\basisCoTan}[1]{\diff{\tidx{\chartmap}{#1}{}}}
\newcommand{\basisCoTanAt}[2][\point]{\evalField[#1]{\basisCoTan{#2}}}
\newcommand{\basisTanMnfAlt}[1]{\deldel{\tidx{\chartmapMnfAlt}{#1}{}}}
\newcommand{\basisCoTanMnfAlt}[1]{\diff{\tidx{\chartmapMnfAlt}{#1}{}}}
\newcommand{\chartmapMnfQPosi}[1]{\tidx{\chartmapTanSpaceMnfQ}{#1}{}}
\newcommand{\chartmapMnfQVeli}[1]{\tidx{\chartmapTanSpaceMnfQ}{\dimMnfQ+#1}{}}

\newcommand{\basisTMnfQTanPos}[1]{\deldel{\chartmapMnfQPosi{#1}}}
\newcommand{\basisTMnfQTanPosAt}[2][\pointMnfQTanSpace]{\evalField[#1]{\basisTMnfQTanPos{#2}}}
\newcommand{\basisTMnfQTanVel}[1]{\deldel{\chartmapMnfQVeli{#1}}}
\newcommand{\basisTMnfQTanVelAt}[2][\pointMnfQTanSpace]{\evalField[#1]{\basisTMnfQTanVel{#2}}}
\newcommand{\basisTMnfQCoTan}[1]{\diff{\tidx{\chartmapTanSpaceMnfQ}{#1}{}}}

\newcommand{\basisTMnfQCoTanPosAt}[2][\pointMnfQTanSpace]{\evalField[#1]{\diff{\chartmapMnfQPosi{#2}}}}
\newcommand{\basisTMnfQCoTanVel}[1]{\diff{\chartmapMnfQVeli{#1}}}
\newcommand{\basisTMnfQCoTanVelAt}[2][\pointMnfQTanSpace]{\evalField[#1]{\basisTMnfQCoTanVel{#2}}}
\newcommand{\chartmapMnfQRedPosi}[1]{\tidx{\chartmapTanSpaceMnfQRed}{#1}{}}
\newcommand{\chartmapMnfQRedVeli}[1]{\tidx{\chartmapTanSpaceMnfQRed}{\dimMnfQRed+#1}{}}
\newcommand{\basisTMnfQRedCoTanPos}[1]{\diff{\chartmapMnfQRedPosi{#1}}}
\newcommand{\basisTMnfQRedCoTanVel}[1]{\diff{\chartmapMnfQRedVeli{#1}}}
\newcommand{\basisTMnfQRedTan}[1]{\deldel{\tidx{\chartmapTanSpaceMnfQRed}{#1}{}}}
\newcommand{\basisTMnfQRedTanPos}[1]{\deldel{\chartmapMnfQRedPosi{#1}}}
\newcommand{\basisTMnfQRedTanVel}[1]{\deldel{\chartmapMnfQRedVeli{#1}}}
\newcommand{\basisTMnfQRedTanPosAt}[2][\pointMnfQTanSpaceRed]{\evalField[#1]{\basisTMnfQRedTanPos{#2}}}
\newcommand{\basisTMnfQRedTanVelAt}[2][\pointMnfQTanSpaceRed]{\evalField[#1]{\basisTMnfQRedTanVel{#2}}}
\newcommand{\basisTMnfQRedCoTan}[1]{\diff{\tidx{\chartmapTanSpaceMnfQRed}{#1}{}}}
\newcommand{\fracdiff}[2]{\tfrac{\partial #1}{\partial #2}}
\newcommand{\diff}[1]{\mathrm{d}{#1}}
\newcommand{\sfracdiff}[3]{\tfrac{\partial^2 #1}{\partial #3 \partial #2}}
\newcommand{\dualdiff}[1]{\mathrm{d}{#1}^\ast}
\newcommand{\pullback}[2]{{#1}^\ast #2}

\newcommand{\gmgPullbackCoordSymbol}{\tname{\fR}{\GMG}}

\newcommand{\gmgReduction}{\tname{R}{\GMG}}
\newcommand{\gmgReductionCoord}{\tname{\fR}{\GMG}}

\newcommand{\mpgReduction}{\tname{R}{\MPG}}
\newcommand{\redMapMPG}{\mpgReduction}
\newcommand{\mpgReductionCoordSymbol}{\tname{\fR}{\MPG}}

\newcommand{\smgReduction}{\tname{R}{\SMG}}
\newcommand{\smgReductionCoord}{\tname{\fR}{\SMG}}

\newcommand{\Trs}[2][r,s]{T^{(#1)}( #2 )}
\newcommand{\smoothSec}[1]{\Gamma\left( #1 \right)}
\newcommand{\vectorspace}{\mathcal{V}}
\newcommand{\trafo}[5]{%
\evalField[#2]{
    \frac{\partial \tidx{#1}{#3}{}}{\partial \tidx{#4}{#5}{}}
}}

\newcommand{\trafoEmb}[3][\pointRed]{%
\trafo{\emb}{#1}{#2}{\chartmapRed}{#3}}

\def \ifempty#1{\def\temp{#1} \ifx\temp\empty }
\newcommand{\tidx}[3]{
    \ifthenelse{\equal{#2}{}}{
        {#1}_{\underline{#3}}
    }{
        \ifthenelse{\equal{#3}{}}{
            {#1}^{\underline{#2}}
        }{
            {#1}^{\underline{#2}}_{\underline{#3}}
        }
    }
}
\newcommand{\TrsField}[2]{\smoothSec{\smash{\Trs[#1]{\T#2}}}}
\newcommand{\TrsFieldMnf}[1][r,s]{\TrsField{#1}{\mnf}}
\newcommand{\TrsFieldMnfAlt}[1][r,s]{\TrsField{#1}{\mnfAlt}}
\newcommand{\TrsFieldMnfRed}[1][r,s]{\TrsField{#1}{\mnfRed}}
\newcommand{\mnfRed}{\reduce{\mnf}}
\newcommand{\mnfSub}{\emb(\mnfRed)}
\newcommand{\vecbunSub}{E_{\mnfSub}}
\newcommand{\emb}{\varphi}
\newcommand{\embCoord}{\bm{\emb}}
\newcommand{\red}{\varrho}
\newcommand{\redCoord}{\bm{\varrho}}

\newcommand{\redMap}{R}
\newcommand{\redMapPoint}{\varrho}
\newcommand{\redMapPointCoord}{\frho}
\newcommand{\redMapCoord}{\bm{R}}
\newcommand{\redMapTan}[1]{\evalField[#1]{\redMap}}

\newcommand{\redMapTanMPGCoord}[1]{\evalField[#1]{\redMapCoord_{\MPG}}}
\newcommand{\redMapTanGMGCoord}[1]{\evalField[#1]{\redMapCoord_{\GMG}}}

\newcommand{\redMapLMGCoord}{\redMapCoord_{\LMG}}
\newcommand{\lmgReduction}{\tname{R}{\LMG}}

\newcommand{\mnfTime}{\calI}
\newcommand{\chartmapMnfTime}{\tname{\chartmap}{\mnfTime}}
\newcommand{\chartMnfTime}{(\mnfTime, \chartmapMnfTime)}
\newcommand{\pointTime}{t}

\newcommand{\pointTimeAlt}{s}

\newcommand{\pointTimeInit}{\tname{\pointTime}{0}}
\newcommand{\pointTimeEnd}{\tname{\pointTime}{\text{f}}}
\newcommand{\pointInit}{\tname{\curve}{0}}
\newcommand{\pointInitCoord}{\tname{\curveCoord}{0}}

\newcommand{\velocityInit}{\tname{\velocity}{0}}
\newcommand{\covec}{\lambda}

\newcommand{\tf}{\tau}%tensor field
\newcommand{\tfRed}{\reduce{\tf}}
\newcommand{\tfCoord}{\bm{\tf}}%tensor field
\newcommand{\tfRedCoord}{\reduce{\tfCoord}}
\newcommand{\vf}{X}% vector field
\newcommand{\vfRed}{\reduce{\vf}}
\newcommand{\vfCoord}{\fX}

\newcommand{\smoothVfs}[1][\mnf]{\mathfrak{X}_{#1}}
\newcommand{\smoothVfsRed}{\smoothVfs[\mnfRed]}

\newcommand{\covf}{\alpha}% covector field

\newcommand{\tensor}{\sigma}

\newcommand{\tensorCoord}{\bm{\tensor}}

\newcommand{\metric}{\tf}
\newcommand{\metricCoord}{\bm{\metric}}
\newcommand{\metricRed}{\reduce{\metric}}
\newcommand{\metricRedCoord}{\reduce{\metricCoord}}

\newcommand{\metricLMG}{\tname{\tf}{\LMG}}
\newcommand{\metricLMGCoord}{\tname{\tfCoord}{\LMG}}
\newcommand{\metricLMGRed}{\tname{\reduce{\tf}}{\LMG}}

\newcommand{\symplForm}{\omega}
\newcommand{\symplFormCoord}{\bm{\symplForm}}
\newcommand{\symplFormRed}{\reduce{\symplForm}}
\newcommand{\symplFormRedCoord}{\reduce{\symplFormCoord}}
\newcommand{\symplFormAlt}{\eta}
\newcommand{\pmetric}{g}
\newcommand{\pmetricCoord}{\bm{\pmetric}}

\newcommand{\pmetricRedCoord}{\reduce{\pmetricCoord}}

\newcommand{\pmetricNormCoord}[1]{\norm{#1}_{\pmetricCoord}}
\newcommand{\pmetricRedNormCoord}[1]{\norm{#1}_{\pmetricRedCoord}}
\newcommand{\isoFlat}[2][\metric]{\flat_{#1}\left( #2 \right)}
\newcommand{\isoSharp}[2][\metric]{\sharp_{#1}\left( #2 \right)}
\newcommand{\isoFlatRed}[2][\metric]{\isoFlat[\reduce{#1}]{#2}}

\newcommand{\curve}{\gamma}
\newcommand{\curveAlt}{\beta}
\newcommand{\curveRed}{\reduce{\curve}}
\newcommand{\curveRedAlt}{\reduce{\curveAlt}}
\newcommand{\curveCoord}{\bm{\curve}}

\newcommand{\curveRedCoord}{\reduce{\curveCoord}}

\newcommand{\lift}[1]{\Gamma_{#1}}
\newcommand{\flow}{\theta}
\newcommand{\flowAt}[1][\pointTime]{\flow_{#1}}
\newcommand{\flowRed}{\reduce{\flow}}
\newcommand{\flowRedAt}[1][\pointTime]{\flowRed_{#1}}
\newcommand{\pointTanSpaceRed}{\reduce{\pointTanSpace}}
\newcommand{\pointTanSpaceRedCoord}{\reduce{\pointTanSpaceCoord}}
\newcommand{\embTanSpace}{\emb}
\newcommand{\embTanSpaceCoord}{\embCoord}
\newcommand{\redTanSpace}{\redMapPoint}
\newcommand{\velocityRed}{\reduce{\velocity}}
\newcommand{\velocityInitRed}{\tname{\velocityRed}{0}}
\newcommand{\velocityRedCoord}{\reduce{\velocityCoord}}
\newcommand{\mnfQ}{\mathcal{Q}}
\newcommand{\energy}{\calE}
\newcommand{\energyRed}{\reduce{\energy}}
\newcommand{\basisMnfQTan}[1]{\deldel{\tidx{\chartmapMnfQ}{#1}{}}}
\newcommand{\basisMnfQTanAt}[2][\pointQ]{\evalField[#1]{\basisMnfQTan{#2}}}
\newcommand{\dimMnfQ}{Q}
\newcommand{\dimMnfQRed}{\reduce{Q}}
\newcommand{\chartmapMnfQ}{\chartmapAlt}
\newcommand{\chartdomainMnfQ}{\chartdomainAlt}
\newcommand{\chartdomainMnfQRed}{\reduce{\chartdomainAlt}}
\newcommand{\chartMnfQ}{(\chartdomainMnfQ, \chartmapMnfQ)}
\newcommand{\chartmapTanSpaceMnfQ}{\chartmapTanSpace}
\newcommand{\chartmapTanSpaceMnfQRed}{\chartmapTanSpaceRed}
\newcommand{\chartdomainTanSpaceMnfQ}{\T\chartdomainMnfQ}
\newcommand{\chartTanSpaceMnfQ}{(\chartdomainTanSpaceMnfQ, \chartmapTanSpaceMnfQ)}
\newcommand{\chartTanSpaceMnfQRed}{(\T\chartdomainMnfQRed, \chartmapTanSpaceMnfQRed)}
\newcommand{\mnfQRed}{\reduce{\mnfQ}}
\newcommand{\basisMnfQRedTan}[1]{\deldel{\tidx{\chartmapMnfQRed}{#1}{}}}
\newcommand{\basisMnfQRedTanAt}[2][\pointQRed]{\evalField[#1]{\basisMnfQRedTan{#2}}}
\newcommand{\mnfQSub}{\embQ(\mnfQRed)}
\newcommand{\pointQ}{q}
\newcommand{\pointQCoord}{\fq}
\newcommand{\pointQRed}{\reduce{\pointQ}}
\newcommand{\pointQRedCoord}{\reduce{\pointQCoord}}
\newcommand{\pointQInit}{\tname{\pointQ}{0}}

\newcommand{\pointQInitRed}{\tname{\pointQRed}{0}}
\newcommand{\pointMnfQTanSpace}{\tname{\pointTanSpace}{\mnfQ}}
\newcommand{\pointMnfQTanSpaceCoord}{\tname{\pointTanSpaceCoord}{\mnfQ}}
\newcommand{\pointMnfQTanSpaceRed}{\pointTanSpaceRed_\mnfQ}
\newcommand{\pointMnfQTanSpaceRedCoord}{\pointTanSpaceRedCoord_\mnfQ}
\newcommand{\curveMnfQ}{\tname{\curve}{\mnfQ}}
\newcommand{\curveMnfQCoord}{\tname{\curveCoord}\mnfQ}
\newcommand{\curveMnfQRed}{\tname{\curveRed}{\mnfQ}}
\newcommand{\curveMnfQRedCoord}{\tname{\curveRedCoord}{\mnfQ}}

\newcommand{\curveQCoord}{\tname{\curveCoord}{\pointQ}}
\newcommand{\curveVCoord}{\tname{\curveCoord}{\velocity}}

\newcommand{\chartmapMnfQRed}{\reduce{\chartmapMnfQ}}
\newcommand{\chartMnfQRed}{(\chartdomainMnfQRed,\chartmapMnfQRed)}
\newcommand{\Lag}{\mathscr{L}}
\newcommand{\LagCoord}{\pmb{\Lag}}
\newcommand{\LagRed}{\reduce{\Lag}}

\newcommand{\vfLag}{\tname{\vf}{\Lag}}
\newcommand{\vfLagRed}{\tname{\reduce{\vf}}{\LagRed}}
\newcommand{\metricV}{\tname{\tf}{\velocity}}
\newcommand{\metricVCoord}{\tname{\metricCoord}{\velocity}}
\newcommand{\metricQ}{\tname{\metric}{\pointQ}}
\newcommand{\metricQCoord}{\tname{\metricCoord}{\pointQ}}

\newcommand{\embQ}{\emb_\mnfQ}
\newcommand{\embQCoord}{\tname{\embCoord}{\mnfQ}}
\newcommand{\redMnfQ}{\tname{\red}{\mnfQ}}
\newcommand{\redMnfQCoord}{\tname{\redCoord}{\mnfQ}}
\newcommand{\metricQRed}{\tname{\metricRed}{\pointQ}}
\newcommand{\metricVRed}{\tname{\metricRed}{\velocity}}
\newcommand{\metricQVRed}{\tname{\metricRed}{\pointQ\velocity}}
\newcommand{\metricQRedCoord}{\tname{\metricRedCoord}{\pointQ}}
\newcommand{\metricVRedCoord}{\tname{\metricRedCoord}{\velocity}}
\newcommand{\metricQVRedCoord}{\tname{\metricRedCoord}{\pointQ\velocity}}

\newcommand{\canonicalJ}[1]{\tname{\mathbb{J}}{#1}}
\newcommand{\JtN}{\canonicalJ{2N}}
\newcommand{\Jtn}{\canonicalJ{2n}}
\newcommand{\Ham}{\mathscr{H}}
\newcommand{\HamCoord}{\pmb{\Ham}}
\newcommand{\HamRed}{\reduce{\Ham}}
\newcommand{\HamRedCoord}{\reduce{\HamCoord}}
\newcommand{\vfHam}{\tname{\vf}{\Ham}}
\newcommand{\vfHamCoord}{\tname{\vfCoord}{\Ham}}

\newcommand{\tname}[2]{{#1}_{{#2}}}

\newcommand{\solMnf}{S}
\newcommand{\solMnfTrain}{\tname{\solMnf}{\mathrm{train}}}

\newcommand{\redEmbFamily}{\tname{\calF}{\emb,\red}}
\newcommand{\loss}{L}
\newcommand{\metricMnfSymbol}{\tname{d}{\mnf}}
\newcommand{\metricMnf}[2]{\evalFun[#1,\,#2]{\metricMnfSymbol}}
\newcommand{\metricMnfCoordSymbol}{\tname{\fd}{\mnf}}
\newcommand{\metricMnfCoord}[2]{\evalFun[#1,\,#2]{\metricMnfCoordSymbol}}

\newcommand{\train}[1]{{#1}_{\mathrm{train}}}
\newcommand{\pointTrain}{{{\train{\point}}}}

\newcommand{\pointTrainCoord}{{{\train{\pointCoord}}}}
\newcommand{\pointTrainAltCoord}{{{\train{\pointAltCoord}}}}
\newcommand{\pointRedAltCoordTrain}{\train{\pointRedAltCoord}}

\newcommand{\classical}[1]{\tname{#1}{\mathrm{lin}}}
\newcommand{\embClassicCoord}{\classical{\embCoord}}
\newcommand{\redClassicCoord}{\classical{\redCoord}}
\newcommand{\redEmbClassicalFamily}{\tname{\calF}{\emb,\red,\mathrm{lin}}}

\newcommand{\quadKron}[1]{{#1}^{\otimes^{\mathrm s} 2}}
\newcommand{\nameQuad}[1]{\tname{#1}{\mathrm{quad}}}
\newcommand{\redQuadCoord}{\nameQuad{\redCoord}}
\newcommand{\embQuadCoord}{\nameQuad{\embCoord}}
\newcommand{\redEmbQuadFamily}{\tname{\calF}{\emb,\red,\mathrm{quad}}}

\newcommand{\dimMnfIntermed}{\widetilde{n}}
\newcommand{\nameNCA}[1]{\tname{#1}{\NCA}}
\newcommand{\redNCACoord}{\nameNCA{\redCoord}}
\newcommand{\embNCACoord}{\nameNCA{\embCoord}}
\newcommand{\redEmbNCAFamily}{\tname{\calF}{\emb,\red,\NCA}}

\newcommand{\networkParam}{\ftheta}
\newcommand{\numNetworkParam}{\tname{n}{\networkParam}}
\newcommand{\redEmbAeFamily}{\tname{\calF}{\emb,\red,\sfAE}}
\newcommand{\nameAe}[1]{\tname{#1}{\sfAE}}
\newcommand{\redAeCoord}{\nameAe{\redCoord}}
\newcommand{\embAeCoord}{\nameAe{\embCoord}}
\newcommand{\lossMSE}{\tname{L}{\MSE}}
\newcommand{\lossMSECoord}{\tname{\fL}{\MSE}}
\newcommand{\Rpos}{\R_{\geq 0}}
\newcommand{\embLipschitzConst}{\tname{C}{\emb}}
\newcommand{\redLipschitzConst}{\tname{C}{\red}}

%% file: figures/differential_geometry/drawing-1.pdf_tex
%% Creator: Inkscape 1.1.2 (0a00cf5339, 2022-02-04), www.inkscape.org
%% PDF/EPS/PS + LaTeX output extension by Johan Engelen, 2010
%% Accompanies image file 'drawing-1.pdf' (pdf, eps, ps)
%%
%% To include the image in your LaTeX document, write
%%   \input{<filename>.pdf_tex}
%%  instead of
%%   \includegraphics{<filename>.pdf}
%% To scale the image, write
%%   \def\svgwidth{<desired width>}
%%   \input{<filename>.pdf_tex}
%%  instead of
%%   \includegraphics[width=<desired width>]{<filename>.pdf}
%%
%% Images with a different path to the parent latex file can
%% be accessed with the `import' package (which may need to be
%% installed) using
%%   \usepackage{import}
%% in the preamble, and then including the image with
%%   \import{<path to file>}{<filename>.pdf_tex}
%% Alternatively, one can specify
%%   \graphicspath{{<path to file>/}}
%% 
%% For more information, please see info/svg-inkscape on CTAN:
%%   http://tug.ctan.org/tex-archive/info/svg-inkscape
%%
\begingroup%
  \makeatletter%
  \providecommand\color[2][]{%
    \errmessage{(Inkscape) Color is used for the text in Inkscape, but the package 'color.sty' is not loaded}%
    \renewcommand\color[2][]{}%
  }%
  \providecommand\transparent[1]{%
    \errmessage{(Inkscape) Transparency is used (non-zero) for the text in Inkscape, but the package 'transparent.sty' is not loaded}%
    \renewcommand\transparent[1]{}%
  }%
  \providecommand\rotatebox[2]{#2}%
  \newcommand*\fsize{\dimexpr\f@size pt\relax}%
  \newcommand*\lineheight[1]{\fontsize{\fsize}{#1\fsize}\selectfont}%
  \ifx\svgwidth\undefined%
    \setlength{\unitlength}{512.10709699bp}%
    \ifx\svgscale\undefined%
      \relax%
    \else%
      \setlength{\unitlength}{\unitlength * \real{\svgscale}}%
    \fi%
  \else%
    \setlength{\unitlength}{\svgwidth}%
  \fi%
  \global\let\svgwidth\undefined%
  \global\let\svgscale\undefined%
  \makeatother%
  \begin{picture}(1,0.62271669)%
    \lineheight{1}%
    \setlength\tabcolsep{0pt}%
    \put(0,0){\includegraphics[width=\unitlength,page=1]{figures/differential_geometry/drawing-1.pdf}}%
    \put(0.18446637,0.46551489){\color[rgb]{0,0,0}\makebox(0,0)[lt]{\lineheight{1.25}\smash{\begin{tabular}[t]{l}$\calM$\end{tabular}}}}%
    \put(0.44597599,0.50504665){\color[rgb]{0,0,0}\makebox(0,0)[lt]{\lineheight{1.25}\smash{\begin{tabular}[t]{l}$U$\end{tabular}}}}%
    \put(0.71396613,0.46378894){\color[rgb]{0,0,0}\makebox(0,0)[rt]{\lineheight{1.25}\smash{\begin{tabular}[t]{r}$V$\end{tabular}}}}%
    \put(0.61671127,0.58793271){\color[rgb]{0,0,0}\makebox(0,0)[t]{\lineheight{1.25}\smash{\begin{tabular}[t]{c}$U\cap V$\end{tabular}}}}%
    \put(0,0){\includegraphics[width=\unitlength,page=2]{figures/differential_geometry/drawing-1.pdf}}%
    \put(0.80920403,0.32741795){\color[rgb]{0,0,0}\makebox(0,0)[lt]{\lineheight{1.25}\smash{\begin{tabular}[t]{l}$y^{-1}$\end{tabular}}}}%
    \put(0.69345441,0.29404409){\color[rgb]{0,0,0}\makebox(0,0)[lt]{\lineheight{1.25}\smash{\begin{tabular}[t]{l}$y$\end{tabular}}}}%
    \put(0.63917808,0.26521618){\color[rgb]{0,0,0}\makebox(0,0)[t]{\lineheight{1.25}\smash{\begin{tabular}[t]{c}$\R^N$\end{tabular}}}}%
    \put(0.08852157,0.26521618){\color[rgb]{0,0,0}\makebox(0,0)[t]{\lineheight{1.25}\smash{\begin{tabular}[t]{c}$\R^N$\end{tabular}}}}%
    \put(0.85052433,0.08763028){\color[rgb]{0,0,0}\makebox(0,0)[rt]{\lineheight{1.25}\smash{\begin{tabular}[t]{r}$y(V)$\end{tabular}}}}%
    \put(0.44729342,0.18168874){\color[rgb]{0,0,0}\makebox(0,0)[lt]{\lineheight{1.25}\smash{\begin{tabular}[t]{l}$y \circ x^{-1}$\end{tabular}}}}%
    \put(0.44455775,0.02220123){\color[rgb]{0,0,0}\makebox(0,0)[lt]{\lineheight{1.25}\smash{\begin{tabular}[t]{l}$x \circ y^{-1}$\end{tabular}}}}%
    \put(0.17721929,0.09357528){\color[rgb]{0,0,0}\makebox(0,0)[lt]{\lineheight{1.25}\smash{\begin{tabular}[t]{l}$x(U)$\end{tabular}}}}%
    \put(0.22085216,0.32189879){\color[rgb]{0,0,0}\makebox(0,0)[lt]{\lineheight{1.25}\smash{\begin{tabular}[t]{l}$x$\end{tabular}}}}%
    \put(0.43262313,0.3018021){\color[rgb]{0,0,0}\makebox(0,0)[lt]{\lineheight{1.25}\smash{\begin{tabular}[t]{l}$x^{-1}$\end{tabular}}}}%
    \put(0,0){\includegraphics[width=\unitlength,page=3]{figures/differential_geometry/drawing-1.pdf}}%
    \put(0.79559124,0.14612684){\color[rgb]{0,0.25490196,0.56862745}\makebox(0,0)[t]{\lineheight{1.25}\smash{\begin{tabular}[t]{c}$y(m)$\end{tabular}}}}%
    \put(0,0){\includegraphics[width=\unitlength,page=4]{figures/differential_geometry/drawing-1.pdf}}%
    \put(0.64935362,0.53302784){\color[rgb]{0,0.25490196,0.56862745}\makebox(0,0)[t]{\lineheight{1.25}\smash{\begin{tabular}[t]{c}$m$\end{tabular}}}}%
    \put(0,0){\includegraphics[width=\unitlength,page=5]{figures/differential_geometry/drawing-1.pdf}}%
    \put(0.23549417,0.16609954){\color[rgb]{0,0.25490196,0.56862745}\makebox(0,0)[t]{\lineheight{1.25}\smash{\begin{tabular}[t]{c}$x(m)$\end{tabular}}}}%
    \put(0,0){\includegraphics[width=\unitlength,page=6]{figures/differential_geometry/drawing-1.pdf}}%
  \end{picture}%
\endgroup%

%% file: figures/differential_geometry/submanifold_w_solMnf.pdf_tex
%% Creator: Inkscape 1.1.2 (0a00cf5339, 2022-02-04), www.inkscape.org
%% PDF/EPS/PS + LaTeX output extension by Johan Engelen, 2010
%% Accompanies image file 'submanifold_w_solMnf.pdf' (pdf, eps, ps)
%%
%% To include the image in your LaTeX document, write
%%   \input{<filename>.pdf_tex}
%%  instead of
%%   \includegraphics{<filename>.pdf}
%% To scale the image, write
%%   \def\svgwidth{<desired width>}
%%   \input{<filename>.pdf_tex}
%%  instead of
%%   \includegraphics[width=<desired width>]{<filename>.pdf}
%%
%% Images with a different path to the parent latex file can
%% be accessed with the `import' package (which may need to be
%% installed) using
%%   \usepackage{import}
%% in the preamble, and then including the image with
%%   \import{<path to file>}{<filename>.pdf_tex}
%% Alternatively, one can specify
%%   \graphicspath{{<path to file>/}}
%% 
%% For more information, please see info/svg-inkscape on CTAN:
%%   http://tug.ctan.org/tex-archive/info/svg-inkscape
%%
\begingroup%
  \makeatletter%
  \providecommand\color[2][]{%
    \errmessage{(Inkscape) Color is used for the text in Inkscape, but the package 'color.sty' is not loaded}%
    \renewcommand\color[2][]{}%
  }%
  \providecommand\transparent[1]{%
    \errmessage{(Inkscape) Transparency is used (non-zero) for the text in Inkscape, but the package 'transparent.sty' is not loaded}%
    \renewcommand\transparent[1]{}%
  }%
  \providecommand\rotatebox[2]{#2}%
  \newcommand*\fsize{\dimexpr\f@size pt\relax}%
  \newcommand*\lineheight[1]{\fontsize{\fsize}{#1\fsize}\selectfont}%
  \ifx\svgwidth\undefined%
    \setlength{\unitlength}{146.15090429bp}%
    \ifx\svgscale\undefined%
      \relax%
    \else%
      \setlength{\unitlength}{\unitlength * \real{\svgscale}}%
    \fi%
  \else%
    \setlength{\unitlength}{\svgwidth}%
  \fi%
  \global\let\svgwidth\undefined%
  \global\let\svgscale\undefined%
  \makeatother%
  \begin{picture}(1,0.37946306)%
    \lineheight{1}%
    \setlength\tabcolsep{0pt}%
    \put(0,0){\includegraphics[width=\unitlength,page=1]{figures/differential_geometry/submanifold_w_solMnf.pdf}}%
    \put(0.26460113,0.15849965){\color[rgb]{1,0,0}\makebox(0,0)[lt]{\lineheight{1.25}\smash{\begin{tabular}[t]{l}$\mnfSub$\end{tabular}}}}%
    \put(0.29032945,0.34438394){\color[rgb]{0,0,0.50196078}\makebox(0,0)[lt]{\lineheight{1.25}\smash{\begin{tabular}[t]{l}$\mnf$\end{tabular}}}}%
    \put(0.50886509,0.15849965){\makebox(0,0)[lt]{\lineheight{1.25}\smash{\begin{tabular}[t]{l}$\solMnf$\end{tabular}}}}%
    \put(0,0){\includegraphics[width=\unitlength,page=2]{figures/differential_geometry/submanifold_w_solMnf.pdf}}%
  \end{picture}%
\endgroup%

%% file: figures/differential_geometry/submanifold.pdf_tex
%% Creator: Inkscape 1.1.2 (0a00cf5339, 2022-02-04), www.inkscape.org
%% PDF/EPS/PS + LaTeX output extension by Johan Engelen, 2010
%% Accompanies image file 'submanifold.pdf' (pdf, eps, ps)
%%
%% To include the image in your LaTeX document, write
%%   \input{<filename>.pdf_tex}
%%  instead of
%%   \includegraphics{<filename>.pdf}
%% To scale the image, write
%%   \def\svgwidth{<desired width>}
%%   \input{<filename>.pdf_tex}
%%  instead of
%%   \includegraphics[width=<desired width>]{<filename>.pdf}
%%
%% Images with a different path to the parent latex file can
%% be accessed with the `import' package (which may need to be
%% installed) using
%%   \usepackage{import}
%% in the preamble, and then including the image with
%%   \import{<path to file>}{<filename>.pdf_tex}
%% Alternatively, one can specify
%%   \graphicspath{{<path to file>/}}
%% 
%% For more information, please see info/svg-inkscape on CTAN:
%%   http://tug.ctan.org/tex-archive/info/svg-inkscape
%%
\begingroup%
  \makeatletter%
  \providecommand\color[2][]{%
    \errmessage{(Inkscape) Color is used for the text in Inkscape, but the package 'color.sty' is not loaded}%
    \renewcommand\color[2][]{}%
  }%
  \providecommand\transparent[1]{%
    \errmessage{(Inkscape) Transparency is used (non-zero) for the text in Inkscape, but the package 'transparent.sty' is not loaded}%
    \renewcommand\transparent[1]{}%
  }%
  \providecommand\rotatebox[2]{#2}%
  \newcommand*\fsize{\dimexpr\f@size pt\relax}%
  \newcommand*\lineheight[1]{\fontsize{\fsize}{#1\fsize}\selectfont}%
  \ifx\svgwidth\undefined%
    \setlength{\unitlength}{315.69023459bp}%
    \ifx\svgscale\undefined%
      \relax%
    \else%
      \setlength{\unitlength}{\unitlength * \real{\svgscale}}%
    \fi%
  \else%
    \setlength{\unitlength}{\svgwidth}%
  \fi%
  \global\let\svgwidth\undefined%
  \global\let\svgscale\undefined%
  \makeatother%
  \begin{picture}(1,0.19571717)%
    \lineheight{1}%
    \setlength\tabcolsep{0pt}%
    \put(0,0){\includegraphics[width=\unitlength,page=1]{figures/differential_geometry/submanifold.pdf}}%
    \put(0.12249886,0.0746831){\color[rgb]{1,0,0}\makebox(0,0)[lt]{\lineheight{1.25}\smash{\begin{tabular}[t]{l}$\pointSub \in \mnfSub$\end{tabular}}}}%
    \put(0.13440996,0.15943485){\color[rgb]{0,0,0.50196078}\makebox(0,0)[lt]{\lineheight{1.25}\smash{\begin{tabular}[t]{l}$\mnf$\end{tabular}}}}%
    \put(0.74226644,0.13248848){\color[rgb]{1,0,0}\makebox(0,0)[lt]{\lineheight{1.25}\smash{\begin{tabular}[t]{l}$\pointRed \in \mnfRed$\end{tabular}}}}%
    \put(0.52961734,0.17947706){\color[rgb]{1,0,0}\makebox(0,0)[lt]{\lineheight{1.25}\smash{\begin{tabular}[t]{l}$\emb$\end{tabular}}}}%
    \put(0.52961734,0.08940377){\color[rgb]{1,0,0}\makebox(0,0)[lt]{\lineheight{1.25}\smash{\begin{tabular}[t]{l}$\inv\emb$\end{tabular}}}}%
    \put(0.52961734,0.00492742){\color[rgb]{0,0,0.50196078}\makebox(0,0)[lt]{\lineheight{1.25}\smash{\begin{tabular}[t]{l}$\redMapPoint$\end{tabular}}}}%
    \put(0,0){\includegraphics[width=\unitlength,page=2]{figures/differential_geometry/submanifold.pdf}}%
    \put(0.13440996,0.15943485){\color[rgb]{0,0,0.50196078}\makebox(0,0)[lt]{\lineheight{1.25}\smash{\begin{tabular}[t]{l}$\mnf$\end{tabular}}}}%
    \put(0,0){\includegraphics[width=\unitlength,page=3]{figures/differential_geometry/submanifold.pdf}}%
  \end{picture}%
\endgroup%

%% file: figures/differential_geometry/submanifold_tangent_space.pdf_tex
%% Creator: Inkscape 1.1.2 (0a00cf5339, 2022-02-04), www.inkscape.org
%% PDF/EPS/PS + LaTeX output extension by Johan Engelen, 2010
%% Accompanies image file 'submanifold_tangent_space.pdf' (pdf, eps, ps)
%%
%% To include the image in your LaTeX document, write
%%   \input{<filename>.pdf_tex}
%%  instead of
%%   \includegraphics{<filename>.pdf}
%% To scale the image, write
%%   \def\svgwidth{<desired width>}
%%   \input{<filename>.pdf_tex}
%%  instead of
%%   \includegraphics[width=<desired width>]{<filename>.pdf}
%%
%% Images with a different path to the parent latex file can
%% be accessed with the `import' package (which may need to be
%% installed) using
%%   \usepackage{import}
%% in the preamble, and then including the image with
%%   \import{<path to file>}{<filename>.pdf_tex}
%% Alternatively, one can specify
%%   \graphicspath{{<path to file>/}}
%% 
%% For more information, please see info/svg-inkscape on CTAN:
%%   http://tug.ctan.org/tex-archive/info/svg-inkscape
%%
\begingroup%
  \makeatletter%
  \providecommand\color[2][]{%
    \errmessage{(Inkscape) Color is used for the text in Inkscape, but the package 'color.sty' is not loaded}%
    \renewcommand\color[2][]{}%
  }%
  \providecommand\transparent[1]{%
    \errmessage{(Inkscape) Transparency is used (non-zero) for the text in Inkscape, but the package 'transparent.sty' is not loaded}%
    \renewcommand\transparent[1]{}%
  }%
  \providecommand\rotatebox[2]{#2}%
  \newcommand*\fsize{\dimexpr\f@size pt\relax}%
  \newcommand*\lineheight[1]{\fontsize{\fsize}{#1\fsize}\selectfont}%
  \ifx\svgwidth\undefined%
    \setlength{\unitlength}{349.66445089bp}%
    \ifx\svgscale\undefined%
      \relax%
    \else%
      \setlength{\unitlength}{\unitlength * \real{\svgscale}}%
    \fi%
  \else%
    \setlength{\unitlength}{\svgwidth}%
  \fi%
  \global\let\svgwidth\undefined%
  \global\let\svgscale\undefined%
  \makeatother%
  \begin{picture}(1,0.26048809)%
    \lineheight{1}%
    \setlength\tabcolsep{0pt}%
    \put(0,0){\includegraphics[width=\unitlength,page=1]{figures/differential_geometry/submanifold_tangent_space.pdf}}%
    \put(0.20775893,0.0921197){\color[rgb]{0.6,0.6,0.6}\makebox(0,0)[lt]{\lineheight{1.25}\smash{\begin{tabular}[t]{l}$\pointSub$\end{tabular}}}}%
    \put(0,0){\includegraphics[width=\unitlength,page=2]{figures/differential_geometry/submanifold_tangent_space.pdf}}%
    \put(0.76730844,0.14430845){\color[rgb]{0.70196078,0.70196078,0.70196078}\makebox(0,0)[lt]{\lineheight{1.25}\smash{\begin{tabular}[t]{l}$\pointRed \in \mnfRed$\end{tabular}}}}%
    \put(0.53382535,0.22676136){\color[rgb]{1,0,0}\makebox(0,0)[lt]{\lineheight{1.25}\smash{\begin{tabular}[t]{l}$\evalField[\pointRed]{\diff{\emb}}$\end{tabular}}}}%
    \put(0.53382535,0.15797884){\color[rgb]{1,0,0}\makebox(0,0)[lt]{\lineheight{1.25}\smash{\begin{tabular}[t]{l}$\evalField[\pointSub]{\diff{\inv\emb}}$\end{tabular}}}}%
    \put(0.53382535,0.0334313){\color[rgb]{0,0,0.50196078}\makebox(0,0)[lt]{\lineheight{1.25}\smash{\begin{tabular}[t]{l}$\redMapTan{\pointSub}$\end{tabular}}}}%
    \put(0,0){\includegraphics[width=\unitlength,page=3]{figures/differential_geometry/submanifold_tangent_space.pdf}}%
    \put(0.22114024,0.22264279){\color[rgb]{1,0,0}\makebox(0,0)[lt]{\lineheight{1.25}\smash{\begin{tabular}[t]{l}$\TpMnfSub$\end{tabular}}}}%
    \put(0.76693827,0.22561652){\color[rgb]{1,0,0}\makebox(0,0)[lt]{\lineheight{1.25}\smash{\begin{tabular}[t]{l}$\TpMnfRed$\end{tabular}}}}%
    \put(0.0672055,0.17185234){\color[rgb]{0,0,0.50196078}\makebox(0,0)[lt]{\lineheight{1.25}\smash{\begin{tabular}[t]{l}$\TpMnf[\pointSub]$\end{tabular}}}}%
    \put(0,0){\includegraphics[width=\unitlength,page=4]{figures/differential_geometry/submanifold_tangent_space.pdf}}%
  \end{picture}%
\endgroup%

%% file: proof_emb_subspaces_from_point_red.tex
By assumption,
$\emb \in \calC^{\infty}(\mnfRed, \mnf)$
and
$\redMapPoint \in \calC^{\infty}(\mnf, \mnfRed)$ are smooth maps.
Then, the restrictions to $\mnfSub \subseteq \mnf$ are smooth maps in the subspace topology, i.e.,
$\emb \in \calC^{\infty}(\mnf, \mnfSub)$
and
$\restrict{\redMapPoint}{\mnfSub} \in \calC^{\infty}(\mnfSub, \mnfRed)$.
Thus, $\emb$ is a smooth diffeomorphism onto its image in the subspace topology.
By \cite[Prop.~4.8.\ (a)]{Lee12},
$\emb$ is a smooth immersion and thus a smooth embedding.

%% file: proof_red_euler_lagrange.tex
The derivatives of $\LagRed$ can be computed for $\pointMnfQTanSpaceRed = \left( \pointQRed, \tidx{\velocityRed}{i}{} \basisMnfQRedTanAt[\pointQRed]{i} \right) \in \T\mnfQRed$ with $1 \leq j,k \leq \dimMnfQ$ and $1\leq i\leq \dimMnfQRed$ as
\begin{align*}
	\evalField[\pointMnfQTanSpaceRed]{
		\fracdiff{\LagRed}{\tidx{\chartmapTanSpaceMnfQRed}{i}{}}
	}
&=
\evalField[
	{\evalFun[
		\pointMnfQTanSpaceRed
	]{\embTanSpace}}
]{
	\fracdiff{\Lag}{\tidx{\chartmapTanSpaceMnfQ}{j}{}}
}
\evalField[\pointMnfQTanSpaceRed]{
	\fracdiff{\tidx{\embTanSpace}{j}{}}{\tidx{\chartmapTanSpaceMnfQRed}{i}{}}
}
+
\evalField[
	{\evalFun[
		\pointMnfQTanSpaceRed
	]{\embTanSpace}}
]{
	\fracdiff{\Lag}{\tidx{\chartmapTanSpaceMnfQ}{\dimMnfQ + j}{}}
}
\evalField[\pointMnfQTanSpaceRed]{
	\fracdiff{\tidx{\embTanSpace}{\dimMnfQ + j}{}}{\tidx{\chartmapTanSpaceMnfQRed}{i}{}}
}\\
	&=
\evalField[
	{\evalFun[
		\pointMnfQTanSpaceRed
	]{\embTanSpace}}
]{
	\fracdiff{\Lag}{\tidx{\chartmapTanSpaceMnfQ}{j}{}}
}
\evalField[{\pointQRed}]{
	\fracdiff{\tidx{\embQ}{j}{}}
		{\tidx{\chartmapMnfQRed}{i}{}}
}
+
\evalField[
	{\evalFun[
		\pointMnfQTanSpaceRed
	]{\embTanSpace}}
]{
	\fracdiff{\Lag}{\tidx{\chartmapTanSpaceMnfQ}{\dimMnfQ + j}{}}
}
\evalField[{\pointQRed}]{
	\sfracdiff{\tidx{\embQ}{j}{}}
		{\tidx{\chartmapMnfQRed}{k}{}}
		{\tidx{\chartmapMnfQRed}{i}{}\,}
} {\tidx{\velocityRed}{k}{}}\\
% d LagRed / d xi^{n+i}
	\evalField[\pointMnfQTanSpaceRed]{
		\fracdiff{\LagRed}{\tidx{\chartmapTanSpaceMnfQRed}{\dimMnfQRed +i}{}}
	}
&= 
\evalField[
	{\evalFun[
		\pointMnfQTanSpaceRed
	]{\embTanSpace}}
]{
	\fracdiff{\Lag}{\tidx{\chartmapTanSpaceMnfQ}{j}{}}
}
\evalField[\pointMnfQTanSpaceRed]{
	\fracdiff{\tidx{\embTanSpace}{j}{}}{\tidx{\chartmapTanSpaceMnfQRed}{\dimMnfQRed + i}{}}
}
+
\evalField[
	{\evalFun[
		\pointMnfQTanSpaceRed
	]{\embTanSpace}}
]{
	\fracdiff{\Lag}{\tidx{\chartmapTanSpaceMnfQ}{\dimMnfQ + j}{}}
}
\evalField[\pointMnfQTanSpaceRed]{
	\fracdiff{\tidx{\embTanSpace}{\dimMnfQ + j}{}}{\tidx{\chartmapTanSpaceMnfQRed}{\dimMnfQRed + i}{}}
}\\
	&=
\evalField[
	{\evalFun[
		\pointMnfQTanSpaceRed
	]{\embTanSpace}}
]{
	\fracdiff{\Lag}{\tidx{\chartmapTanSpaceMnfQ}{\dimMnfQ + j}{}}
}
\evalField[{\pointQRed}]{
	\fracdiff{\tidx{\embQ}{j}{}}
		{\tidx{\chartmapMnfQRed}{i}{}}
}.
\end{align*}
Evaluation for the lifted curve $\lift{\curveMnfQRed} \in \calC^{\infty}(\mnfTime, \T\mnfQRed)$
and derivation with respect to the time, yields for $1 \leq j, k \leq \dimMnfQ$ and $1 \leq i, \ell, p \leq \dimMnfQRed$ 
\begin{align*}
	\evalField[\pointTime]{\ddt\left(
	\evalField[\lift{\evalFun[\cdot]{\curveMnfQRed}}]{
		\fracdiff{\LagRed}{\tidx{\chartmapTanSpaceMnfQRed}{\dimMnfQRed+i}{}}
	}
	\right)}
=&\;
\evalField[\pointTime]{
\ddt\left(
	\evalField[
		{\evalFun[
			{\lift{\evalFun[\cdot]{\curveMnfQRed}}}
		]{\embTanSpace}}
	]{
		\fracdiff{\Lag}{\tidx{\chartmapTanSpaceMnfQ}{\dimMnfQ + j}{}}
	}
	\evalField[{\evalFun[\cdot]{\curveMnfQRed}}]{
		\fracdiff{\tidx{\embQ}{j}{}}
			{\tidx{\chartmapMnfQRed}{i}{}}
	}
\right)
}\\
	=&\;
% first term: k, l, l
\evalField[
	{\evalFun[
		{\lift{\evalFun[\pointTime]{\curveMnfQRed}}}
	]{\embTanSpace}}
]{
	\sfracdiff{\Lag}
		{\tidx{\chartmapTanSpaceMnfQ}{\dimMnfQ + j}{}}
		{\tidx{\chartmapTanSpaceMnfQ}{k}{}\,}
}
\evalField[{\evalFun[\pointTime]{\curveMnfQRed}}]{
	\fracdiff{\tidx{\embQ}{j}{}}
		{\tidx{\chartmapMnfQRed}{i}{}}
}
\evalField[{
	{\lift{\evalFun[\pointTime]{\curveMnfQRed}}}
}]{
	\fracdiff{\tidx{\embTanSpace}{k}{}}
		{\tidx{\chartmapTanSpaceMnfQRed}{\ell}{}}
}
\evalField[\pointTime]{\ddt\left(
	\tidx{\lift{\evalFun[\cdot]{\curveMnfQRed}}}{\ell}{}
\right)}\\
% second term: k, n+l, n+l
	&+
	\evalField[
		{\evalFun[
			{\lift{\evalFun[\pointTime]{\curveMnfQRed}}}
		]{\embTanSpace}}
	]{
		\sfracdiff{\Lag}
			{\tidx{\chartmapTanSpaceMnfQ}{\dimMnfQ + j}{}}
			{\tidx{\chartmapTanSpaceMnfQ}{k}{}\,}
	}
	\evalField[{\evalFun[\pointTime]{\curveMnfQRed}}]{
		\fracdiff{\tidx{\embQ}{j}{}}
			{\tidx{\chartmapMnfQRed}{i}{}}
	}
	\evalField[{
		{\lift{\evalFun[\pointTime]{\curveMnfQRed}}}
	}]{
		\fracdiff{\tidx{\embTanSpace}{k}{}}
			{\tidx{\chartmapTanSpaceMnfQRed}{\dimMnfQRed + \ell}{}}
	}
	\evalField[\pointTime]{
		\ddt\left(
			\tidx{\lift{\evalFun[\cdot]{\curveMnfQRed}}}{\dimMnfQRed + \ell}{}
		\right)
	}\\
		&+
% third term: N+k, l, l
\evalField[
	{\evalFun[
		{\lift{\evalFun[\pointTime]{\curveMnfQRed}}}
	]{\embTanSpace}}
]{
	\sfracdiff{\Lag}
		{\tidx{\chartmapTanSpaceMnfQ}{\dimMnfQ + j}{}}
		{\tidx{\chartmapTanSpaceMnfQ}{\dimMnfQ + k}{}\,}
}
\evalField[{\evalFun[\pointTime]{\curveMnfQRed}}]{
	\fracdiff{\tidx{\embQ}{j}{}}
		{\tidx{\chartmapMnfQRed}{i}{}}
}
\evalField[{
	{\lift{\evalFun[\pointTime]{\curveMnfQRed}}}
}]{
	\fracdiff{\tidx{\embTanSpace}{\dimMnfQ + k}{}}
		{\tidx{\chartmapTanSpaceMnfQRed}{\ell}{}}
}
\evalField[\pointTime]{
	\ddt\left(
		\tidx{\lift{\evalFun[\cdot]{\curveMnfQRed}}}{\ell}{}
	\right)
}\\
% fourth term: N+k, n+l, n+l
	&+
	\evalField[
		{\evalFun[
			{\lift{\evalFun[\pointTime]{\curveMnfQRed}}}
		]{\embTanSpace}}
	]{
		\sfracdiff{\Lag}
			{\tidx{\chartmapTanSpaceMnfQ}{\dimMnfQ + j}{}}
			{\tidx{\chartmapTanSpaceMnfQ}{\dimMnfQ + k}{}\,}
	}
	\evalField[{\evalFun[\pointTime]{\curveMnfQRed}}]{
		\fracdiff{\tidx{\embQ}{j}{}}
			{\tidx{\chartmapMnfQRed}{i}{}}
	}
	\evalField[{
		{\lift{\evalFun[\pointTime]{\curveMnfQRed}}}
	}]{
		\fracdiff{\tidx{\embTanSpace}{\dimMnfQ + k}{}}
			{\tidx{\chartmapTanSpaceMnfQRed}{\dimMnfQRed + \ell}{}}
	}
	\evalField[\pointTime]{
		\ddt\left(
			\tidx{\lift{\evalFun[\cdot]{\curveMnfQRed}}}{\dimMnfQRed + \ell}{}
		\right)
	}\\
		&+
% fifth term
\evalField[
	{\evalFun[
		{\lift{\evalFun[\pointTime]{\curveMnfQRed}}}
	]{\embTanSpace}}
]{
	\fracdiff{\Lag}{\tidx{\chartmapTanSpaceMnfQ}{\dimMnfQ + j}{}}
}
\evalField[{\evalFun[\pointTime]{\curveMnfQRed}}]{
	\sfracdiff{\tidx{\embQ}{j}{}}
		{\tidx{\chartmapMnfQRed}{i}{}}
		{\tidx{\chartmapMnfQRed}{k}{}\,}
}
\evalField[\pointTime]{
\ddt\left(
	\tidx{\lift{\evalFun[\cdot]{\curveMnfQRed}}}{k}{}
\right)
}\\
	=&\;
% first term: k, l, l
\evalField[
	{\evalFun[
		{\lift{\evalFun[\pointTime]{\curveMnfQRed}}}
	]{\embTanSpace}}
]{
	\sfracdiff{\Lag}
		{\tidx{\chartmapTanSpaceMnfQ}{\dimMnfQ + j}{}}
		{\tidx{\chartmapTanSpaceMnfQ}{k}{}\,}
}
\evalField[{\evalFun[\pointTime]{\curveMnfQRed}}]{
	\fracdiff{\tidx{\embQ}{j}{}}
		{\tidx{\chartmapMnfQRed}{i}{}}
}
\evalField[{\evalFun[\pointTime]{\curveMnfQRed}}]{
	\fracdiff{\tidx{\embQ}{k}{}}
		{\tidx{\chartmapMnfQRed}{\ell}{}}
}
\evalField[\pointTime]{\ddt\tidx{\curveMnfQRed}{\ell}{}}
\\
	&+
% third term: N+k, l, l
\evalField[
	{\evalFun[
		{\lift{\evalFun[\pointTime]{\curveMnfQRed}}}
	]{\embTanSpace}}
]{
	\sfracdiff{\Lag}
		{\tidx{\chartmapTanSpaceMnfQ}{\dimMnfQ + j}{}}
		{\tidx{\chartmapTanSpaceMnfQ}{\dimMnfQ + k}{}\,}
}
\evalField[{\evalFun[\pointTime]{\curveMnfQRed}}]{
	\fracdiff{\tidx{\embQ}{j}{}}
		{\tidx{\chartmapMnfQRed}{i}{}}
}
\evalField[{\pointQRed}]{
	\sfracdiff{\tidx{\embQ}{k}{}}
		{\tidx{\chartmapMnfQRed}{p}{}}
		{\tidx{\chartmapMnfQRed}{\ell}{}\,}
}
\evalField[\pointTime]{\ddt\tidx{\curveMnfQRed}{p}{}}
\evalField[\pointTime]{\ddt\tidx{\curveMnfQRed}{\ell}{}}\\
% fourth term: N+k, n+l, n+l
	&+
\evalField[
	{\evalFun[
		{\lift{\evalFun[\pointTime]{\curveMnfQRed}}}
	]{\embTanSpace}}
]{
	\sfracdiff{\Lag}
		{\tidx{\chartmapTanSpaceMnfQ}{\dimMnfQ + j}{}}
		{\tidx{\chartmapTanSpaceMnfQ}{\dimMnfQ + k}{}\,}
}
\evalField[{\evalFun[\pointTime]{\curveMnfQRed}}]{
	\fracdiff{\tidx{\embQ}{j}{}}
		{\tidx{\chartmapMnfQRed}{i}{}}
}
\evalField[{\evalFun[\pointTime]{\curveMnfQRed}}]{
	\fracdiff{\tidx{\embQ}{k}{}}
		{\tidx{\chartmapMnfQRed}{\ell}{}}
}
\evalField[\pointTime]{\sddt\tidx{\curveMnfQRed}{\ell}{}}\\
	&+
% fifth term
\evalField[
	{\evalFun[
		{\lift{\evalFun[\pointTime]{\curveMnfQRed}}}
	]{\embTanSpace}}
]{
	\fracdiff{\Lag}{\tidx{\chartmapTanSpaceMnfQ}{\dimMnfQ + j}{}}
}
\evalField[{\evalFun[\pointTime]{\curveMnfQRed}}]{
	\sfracdiff{\tidx{\embQ}{j}{}}
		{\tidx{\chartmapMnfQRed}{i}{}}
		{\tidx{\chartmapMnfQRed}{\ell}{}\,}
}
\evalField[\pointTime]{\ddt\tidx{\curveMnfQRed}{\ell}{}}
\end{align*}
In total, it holds for $1 \leq j, k \leq \dimMnfQ$ and $1 \leq i, \ell, p \leq \dimMnfQRed$ in
\begin{align*}
	0 &= \evalField[\lift{\evalFun[\pointTime]{\curveMnfQRed}}]{
	\fracdiff{\LagRed}{\tidx{\chartmapTanSpaceMnfQRed}{i}{}}
}
-
\evalField[\pointTime]{\ddt\left(
	\evalField[\lift{\evalFun[\cdot]{\curveMnfQRed}}]{
		\fracdiff{\LagRed}{\tidx{\chartmapTanSpaceMnfQRed}{\dimMnfQRed+i}{}}
	}
\right)
}\\
	&=
\evalField[{\evalFun[\pointTime]{\curveMnfQRed}}]{
	\fracdiff{\tidx{\embQ}{j}{}}
		{\tidx{\chartmapMnfQRed}{i}{}}
}
\Bigg(
\evalField[
	{\evalFun[
		{\lift{\evalFun[\pointTime]{\curveMnfQRed}}}
	]{\embTanSpace}}
]{
	\fracdiff{\Lag}{\tidx{\chartmapTanSpaceMnfQ}{j}{}}
}
-
% first term: k, l, l
\evalField[
	{\evalFun[
		{\lift{\evalFun[\pointTime]{\curveMnfQRed}}}
	]{\embTanSpace}}
]{
	\sfracdiff{\Lag}
		{\tidx{\chartmapTanSpaceMnfQ}{\dimMnfQ + j}{}}
		{\tidx{\chartmapTanSpaceMnfQ}{k}{}\,}
}
\evalField[{\evalFun[\pointTime]{\curveMnfQRed}}]{
	\fracdiff{\tidx{\embQ}{k}{}}
		{\tidx{\chartmapMnfQRed}{\ell}{}}
}
\evalField[\pointTime]{\ddt\tidx{\curveMnfQRed}{\ell}{}}\\
	&\phantom{\evalField[{\evalFun[\pointTime]{\curveMnfQRed}}]{
		\fracdiff{\tidx{\embQ}{j}{}}
			{\tidx{\chartmapMnfQRed}{i}{}}
	}
	\Bigg(}
-
% third term: N+k, l, l
\evalField[
	{\evalFun[
		{\lift{\evalFun[\pointTime]{\curveMnfQRed}}}
	]{\embTanSpace}}
]{
	\sfracdiff{\Lag}
		{\tidx{\chartmapTanSpaceMnfQ}{\dimMnfQ + j}{}}
		{\tidx{\chartmapTanSpaceMnfQ}{\dimMnfQ + k}{}\,}
}
\evalField[{\evalField[\pointTime]{{\curveMnfQRed}}}]{
	\sfracdiff{\tidx{\embQ}{k}{}}
		{\tidx{\chartmapMnfQRed}{p}{}}
		{\tidx{\chartmapMnfQRed}{\ell}{}\,}
}
\evalField[\pointTime]{\ddt\tidx{\curveMnfQRed}{p}{}}
\evalField[\pointTime]{\ddt\tidx{\curveMnfQRed}{\ell}{}}\\
% fourth term: N+k, n+l, n+l
&\phantom{\evalField[{\evalFun[\pointTime]{\curveMnfQRed}}]{
	\fracdiff{\tidx{\embQ}{j}{}}
		{\tidx{\chartmapMnfQRed}{i}{}}
}
\Bigg(}
-
\evalField[
	{\evalFun[
		{\lift{\evalFun[\pointTime]{\curveMnfQRed}}}
	]{\embQ}}
]{
	\sfracdiff{\Lag}
		{\tidx{\chartmapTanSpaceMnfQ}{\dimMnfQ + j}{}}
		{\tidx{\chartmapTanSpaceMnfQ}{\dimMnfQ + k}{}\,}
}
\evalField[{\evalFun[\pointTime]{\curveMnfQRed}}]{
	\fracdiff{\tidx{\embQ}{k}{}}
		{\tidx{\chartmapMnfQRed}{\ell}{}}
}
\evalField[\pointTime]{\sddt\tidx{\curveMnfQRed}{\ell}{}}
\Bigg).
\end{align*}

%% file: proof_equivalence_solve_lmg.tex
	In order to show that the systems are equivalent, we show that the underlying vector fields are identical, i.e., we show that the \LMG reduction~\eqref{eqn:lmg_reduction} of the Euler--Lagrange vector field~\eqref{eq:def_vfLag} results in the reduced Euler--Lagrangian vector field~\eqref{eq:red_vf_euler_lagrange}. To simplify the notation, we use $\metric = \metricLMG$ and $\metricRed = \metricLMGRed$ in the following, with $\metricLMG$ as in~\eqref{eqn:LMGtensorField}.
	Let $\pointMnfQTanSpaceRed = (\pointQRed, \velocityRed) = \big(\pointQRed, \tidx{\velocityRed}{i}{} \basisMnfQRedTanAt{i}\big) \in \T\mnfQRed$.
	The reduced tensor field $\metricRed
	= \pullback{\diff{\embTanSpace}} \metric$ reads for $1 \leq \alpha, \beta \leq 2\dimMnfQ$
	\begin{align*}
		\evalField[\pointMnfQTanSpaceRed]{\metricRed}
		&=
\evalField[\pointMnfQTanSpaceRed]{
	\fracdiff{\tidx{\embTanSpace}{\gamma}{}}{\tidx{\chartmapTanSpaceMnfQ}{\alpha}{}}
}
\tidx{\left(\evalField[
	{\evalFun[
		\pointMnfQTanSpaceRed
	]{\embTanSpace}}
]{\metric}\right)}{}{\gamma\delta}\;
\evalField[\pointMnfQTanSpaceRed]{
	\fracdiff{\tidx{\embTanSpace}{\delta}{}}{\tidx{\chartmapTanSpaceMnfQ}{\beta}{}}
}
\evalField[\pointMnfQTanSpaceRed]{
	\basisTMnfQCoTan{\alpha}
}
\otimes
\evalField[\pointMnfQTanSpaceRed]{
	\basisTMnfQCoTan{\beta}
}\\
&=
\tidx{\left(\evalField[
	{\evalFun[
		\pointMnfQTanSpaceRed
	]{\embTanSpace}}
]{\metricQ}\right)}{}{k\ell}
\left(
	\evalField[\pointMnfQTanSpaceRed]{
		\sfracdiff{\tidx{\embQ}{k}{}}
			{\tidx{\chartmapMnfQRed}{i}{}}
			{\tidx{\chartmapMnfQRed}{p}{}\,}
	}
	\tidx{\velocityRed}{p}{}\,
	\evalField[\pointMnfQTanSpaceRed]{
		\basisTMnfQRedCoTanPos{i}
	}
	+
	\evalField[\pointMnfQTanSpaceRed]{
		\fracdiff{\tidx{\embQ}{k}{}}
			{\tidx{\chartmapMnfQRed}{i}{}}
	}
	\evalField[\pointMnfQTanSpaceRed]{
		\basisTMnfQRedCoTanVel{i}
	}
\right)
\otimes
\left(
	\evalField[\pointMnfQTanSpaceRed]{
		\fracdiff{\tidx{\embQ}{\ell}{}}
			{\tidx{\chartmapMnfQRed}{j}{}}
	}
	\evalField[\pointMnfQTanSpaceRed]{
		\basisTMnfQRedCoTanPos{j}
	}
\right)\\
&\phantom{=} +
\tidx{\left(\evalField[
	{\evalFun[
		\pointMnfQTanSpaceRed
	]{\embTanSpace}}
]{\metricV}\right)}{}{k\ell}
\left(
	\evalField[\pointMnfQTanSpaceRed]{
		\fracdiff{\tidx{\embQ}{k}{}}
			{\tidx{\chartmapMnfQRed}{i}{}}
	}
	\evalField[\pointMnfQTanSpaceRed]{
		\basisTMnfQRedCoTanPos{i}
	}
\right)
\otimes
\left(
	\evalField[\pointMnfQTanSpaceRed]{
		\sfracdiff{\tidx{\embQ}{\ell}{}}
			{\tidx{\chartmapMnfQRed}{j}{}}
			{\tidx{\chartmapMnfQRed}{p}{}\,}
	}
	\tidx{\velocityRed}{p}{}\,
	\evalField[\pointMnfQTanSpaceRed]{
		\basisTMnfQRedCoTanPos{j}
	}
	+
	\evalField[\pointMnfQTanSpaceRed]{
		\fracdiff{\tidx{\embQ}{\ell}{}}
			{\tidx{\chartmapMnfQRed}{j}{}}
	}
	\evalField[\pointMnfQTanSpaceRed]{
		\basisTMnfQRedCoTanVel{j}
	}
\right)\\
&=
\tidx{\left(\evalField[\pointMnfQTanSpaceRed]{
	\metricQRed}\right)}{}{ij}
\evalField[\pointMnfQTanSpaceRed]{
	\basisTMnfQRedCoTanVel{i}
}
\otimes
\evalField[\pointMnfQTanSpaceRed]{
	\basisTMnfQRedCoTanPos{j}
}
+
\tidx{\left(\evalField[\pointMnfQTanSpaceRed]{
	\metricQVRed}\right)}{}{ij}
\evalField[\pointMnfQTanSpaceRed]{
	\basisTMnfQRedCoTanPos{i}
}
\otimes
\evalField[\pointMnfQTanSpaceRed]{
	\basisTMnfQRedCoTanPos{j}
}\\
&\phantom{=} +
\tidx{\left(\evalField[\pointMnfQTanSpaceRed]{
	\metricVRed}\right)}{}{ij}
\evalField[\pointMnfQTanSpaceRed]{
	\basisTMnfQRedCoTanPos{i}
}
\otimes
\evalField[\pointMnfQTanSpaceRed]{
	\basisTMnfQRedCoTanVel{j}
}
	\end{align*}
	with the abbreviations
	\begin{align*}
		\tidx{\left(\evalField[\pointMnfQTanSpaceRed]{
			\metricQRed}\right)}{}{ij}
&\vcentcolon=
\evalField[\pointMnfQTanSpaceRed]{
	\fracdiff{\tidx{\embQ}{k}{}}
		{\tidx{\chartmapMnfQRed}{i}{}}
}
\tidx{\left(\evalField[
	{\evalFun[
		\pointMnfQTanSpaceRed
	]{\embTanSpace}}
]{
	\metricQ}\right)}{}{k\ell}
\evalField[\pointMnfQTanSpaceRed]{
	\fracdiff{\tidx{\embQ}{\ell}{}}
		{\tidx{\chartmapMnfQRed}{j}{}}
},\\
		\tidx{\left(\evalField[\pointMnfQTanSpaceRed]{
			\metricVRed}\right)}{}{ij}
&\vcentcolon=
\evalField[\pointMnfQTanSpaceRed]{
	\fracdiff{\tidx{\embQ}{k}{}}
		{\tidx{\chartmapMnfQRed}{i}{}}
}
\tidx{\left(\evalField[
	{\evalFun[
		\pointMnfQTanSpaceRed
	]{\embTanSpace}}
]{
	\metricV}\right)}{}{k\ell}
\evalField[\pointMnfQTanSpaceRed]{
	\fracdiff{\tidx{\embQ}{\ell}{}}
		{\tidx{\chartmapMnfQRed}{j}{}}
},\\
		\tidx{\left(\evalField[\pointMnfQTanSpaceRed]{
			\metricQVRed}\right)}{}{ij}
&\vcentcolon=
\evalField[\pointMnfQTanSpaceRed]{
	\sfracdiff{\tidx{\embQ}{k}{}}
		{\tidx{\chartmapMnfQRed}{i}{}}
		{\tidx{\chartmapMnfQRed}{p}{}\,}
}
	\tidx{\velocityRed}{p}{}\,
\tidx{\left(\evalField[
	{\evalFun[
		\pointMnfQTanSpaceRed
	]{\embTanSpace}}
]{
	\metricQ}\right)}{}{k\ell}
\evalField[\pointMnfQTanSpaceRed]{
	\fracdiff{\tidx{\embQ}{\ell}{}}
		{\tidx{\chartmapMnfQRed}{j}{}}
}
+
\evalField[\pointMnfQTanSpaceRed]{
	\fracdiff{\tidx{\embQ}{k}{}}
		{\tidx{\chartmapMnfQRed}{i}{}}
}
\tidx{\left(\evalField[
	{\evalFun[
		\pointMnfQTanSpaceRed
	]{\embTanSpace}}
]{
	\metricV}\right)}{}{k\ell}\,
\evalField[\pointMnfQTanSpaceRed]{
	\sfracdiff{\tidx{\embQ}{\ell}{}}
		{\tidx{\chartmapMnfQRed}{j}{}}
		{\tidx{\chartmapMnfQRed}{p}{}\,}
}
	\tidx{\velocityRed}{p}{}
	\end{align*}
	for $1 \leq k,\ell \leq \dimMnfQ$ and $1 \leq i,j,p \leq \dimMnfQRed$.
	It is easy to verify that the inverse of $\metricRed$ is given by
	\begin{equation}\label{eq:proof_lagrange_inv_metricRed}
	\begin{aligned}
		\inv{\evalField[\pointMnfQTanSpaceRed]{\metricRed}} &=
\tidx{(\evalField[\pointMnfQTanSpaceRed]{ \metricQRed })}{ij}{}
\basisTMnfQRedTanPosAt{i}
\otimes
\basisTMnfQRedTanVelAt{j}\\
	&\phantom{=}-
\tidx{(\evalField[\pointMnfQTanSpaceRed]{ \metricVRed })}{ik}{}\,
\tidx{(\evalField[\pointMnfQTanSpaceRed]{ \metricQVRed })}{}{k\ell}\,
\tidx{(\evalField[\pointMnfQTanSpaceRed]{ \metricQRed })}{\ell j}{}
\basisTMnfQRedTanVelAt{i}
\otimes
\basisTMnfQRedTanVelAt{j}\\
	&\phantom{=}+
\tidx{(\evalField[\pointMnfQTanSpaceRed]{ \metricVRed })}{ij}{}
\basisTMnfQRedTanVelAt{i}
\otimes
\basisTMnfQRedTanPosAt{j}
	\end{aligned}
	\end{equation}
	with $1 \leq i,j,k,\ell \leq \dimMnfQRed$, which in bold notation reads
	\begin{align*}
		&\evalField[\pointMnfQTanSpaceRedCoord]{\metricRedCoord}
= \begin{pmatrix}
	\evalField[\pointMnfQTanSpaceRedCoord]{\metricQVRedCoord} &
	\evalField[\pointMnfQTanSpaceRedCoord]{\metricVRedCoord}\\
	\evalField[\pointMnfQTanSpaceRedCoord]{\metricQRedCoord}  & \fzero
\end{pmatrix}&
		&\text{and}&
		&\inv{\evalField[\pointMnfQTanSpaceRedCoord]{\metricRedCoord}}
= \begin{pmatrix}
	\fzero & \inv{\evalField[\pointMnfQTanSpaceRedCoord]{\metricQRedCoord}}\\
	\inv{\evalField[\pointMnfQTanSpaceRedCoord]{\metricVRedCoord}} &
	-\inv{\evalField[\pointMnfQTanSpaceRedCoord]{\metricVRedCoord}}
		\evalField[\pointMnfQTanSpaceRedCoord]{\metricQVRedCoord}
		\inv{\evalField[\pointMnfQTanSpaceRedCoord]{\metricQRedCoord}}
\end{pmatrix}.
	\end{align*}
	Moreover, we obtain for the indices $1 \leq \beta, \gamma \leq 2\dimMnfQ$ and $1 \leq \alpha \leq 2\dimMnfQRed$ and $1 \leq j,k \leq \dimMnfQ$ and $1 \leq \ell,p,r \leq \dimMnfQRed$
	\begin{equation}\label{eq:proof_lagrange_term_without_inv_metric}
	\begin{aligned}
		\evalField[\pointMnfQTanSpaceRed]{
			\isoFlatRed{\lmgReduction (\vfLag)}
		}
&= \tidx{\left(
	\evalField[{{{\pointMnfQTanSpaceRed}}}]{\dualdiff{\embTanSpace}} \left(
		\isoFlat{\evalField[{{
			\evalFun[{
				\pointMnfQTanSpaceRed
			}]{\embTanSpace}
		}}]{\vfLag}}
	\right)
\right)}{}{\alpha}
\evalField[\pointMnfQTanSpaceRed]{
	\basisTMnfQRedCoTan{\alpha}
}\\
&=
\evalField[\pointMnfQTanSpaceRed]{
	\fracdiff{\tidx{\embTanSpace}{\beta}{}}
		{\tidx{\chartmapTanSpaceMnfQRed}{\alpha}{}}
}
\tidx{\left(\evalField[
	{\evalFun[
		\pointMnfQTanSpaceRed
	]{\embTanSpace}}
]{
	\metric}\right)}{}{\beta\gamma}
\tidx{\left(\evalField[
	{\evalFun[\pointMnfQTanSpaceRed]{\embTanSpace}}
]{
	\vfLag}\right)}{\gamma}{}
\evalField[\pointMnfQTanSpaceRed]{
	\basisTMnfQRedCoTan{\alpha}
}
\\
		&=
\bigg(
\evalField[\pointMnfQTanSpaceRed]{
	\fracdiff{\tidx{\embTanSpace}{j}{}}{\tidx{\chartmapTanSpaceMnfQRed}{\alpha}{}}
}
\tidx{\left(\evalField[
	{\evalFun[
		\pointMnfQTanSpaceRed
	]{\embTanSpace}}
]{
	\metricV}\right)}{}{jk}
\tidx{\left(\evalField[{
	{\evalFun[\pointMnfQTanSpaceRed]{\embTanSpace}}
}]{\vfLag} \right)}{\dimMnfQ + k}{}\;
\\
	&\phantom{=}\qquad +
\evalField[\pointMnfQTanSpaceRed]{
	\fracdiff{\tidx{\embTanSpace}{\dimMnfQRed + j}{}}
		{\tidx{\chartmapTanSpaceMnfQRed}{\alpha}{}}
}
\tidx{\left(\evalField[
	{\evalFun[
		\pointMnfQTanSpaceRed
	]{\embTanSpace}}
]{
	\metric}\right)}{}{jk}
\evalField[{\pointQRed}]{
	\fracdiff{\tidx{\embQ}{k}{}}
		{\tidx{\chartmapMnfQRed}{\ell}{}}
}
\tidx{\velocityRed}{\ell}{}
\bigg)
\evalField[\pointMnfQTanSpaceRed]{
	\basisTMnfQRedCoTan{\alpha}
}
\\
		&=
\bigg(
\evalField[{\pointQRed}]{
	\fracdiff{\tidx{\embQ}{j}{}}
		{\tidx{\chartmapMnfQRed}{r}{}}
}
\tidx{\left(\evalField[
	{\evalFun[
		\pointMnfQTanSpaceRed
	]{\embTanSpace}}
]{
	\metricV}\right)}{}{jk}
\tidx{\left( \evalField[{
	{\evalFun[\pointMnfQTanSpaceRed]{\embTanSpace}}
}]{\vfLag} \right)}{\dimMnfQ + k}{}
\\
		&\phantom{=}\qquad +
\evalField[{\pointQRed}]{
	\sfracdiff{\tidx{\embQ}{j}{}}
		{\tidx{\chartmapMnfQRed}{p}{}}
		{\tidx{\chartmapMnfQRed}{r}{}\,}
}
\tidx{\velocityRed}{p}{}\,
\tidx{\left(\evalField[
	{\evalFun[
		\pointMnfQTanSpaceRed
	]{\embTanSpace}}
]{
	\metricQ}\right)}{}{jk}
\evalField[{\pointQRed}]{
	\fracdiff{\tidx{\embQ}{k}{}}
		{\tidx{\chartmapMnfQRed}{\ell}{}}
}
\tidx{\velocityRed}{\ell}{}
\bigg)
\evalField[\pointMnfQTanSpaceRed]{
	\basisTMnfQRedCoTanPos{r}
}
\\
		&\phantom{=}+
\evalField[{\pointQRed}]{
	\fracdiff{\tidx{\embQ}{j}{}}
		{\tidx{\chartmapMnfQRed}{r}{}}
}
\tidx{\left(\evalField[
	{\evalFun[
		\pointMnfQTanSpaceRed
	]{\embTanSpace}}
]{
	\metricQ}\right)}{}{jk}
\evalField[{\pointQRed}]{
	\fracdiff{\tidx{\embQ}{k}{}}
		{\tidx{\chartmapMnfQRed}{\ell}{}}
}
\tidx{\velocityRed}{\ell}{}\,
\evalField[\pointMnfQTanSpaceRed]{
	\basisTMnfQRedCoTanVel{r}
}
	\end{aligned}
	\end{equation}
	and observe that the last term equals $\tidx{\left(\evalField[\pointMnfQTanSpaceRed]{
		\metricQRed}\right)}{}{r\ell}
	$.
	Combining \eqref{eq:proof_lagrange_term_without_inv_metric} with \eqref{eq:proof_lagrange_inv_metricRed}, the \LMG reduction~\eqref{eqn:lmg_reduction} of the Euler--Lagrange vector field~\eqref{eq:def_vfLag} can be written (with the indices $1 \leq \beta, \gamma \leq 2\dimMnfQ$ and $1 \leq \alpha \leq 2\dimMnfQRed$ and $1 \leq j,k \leq \dimMnfQ$ and $1 \leq i,\ell,p,r \leq \dimMnfQRed$) as
	\begin{align*}
		\evalField[\pointMnfQTanSpaceRed]{
			\lmgReduction (\vfLag)
		}
		&=
		\tidx{\left(
			\evalFun[{
				\evalField[{{
					\evalFun[\pointMnfQTanSpaceRed]{\embTanSpace}
				}}]{\vfLag}
			}]{\left(
				\sharp_{\metricRed}
				\circ
				\evalField[\pointMnfQTanSpaceRed]{\dualdiff{\embTanSpace}}
				\circ
				\flat_{\metric}
			\right)}
		\right)}{\alpha}{}
		\evalField[\pointMnfQTanSpaceRed]{
			\basisTMnfQRedTan{\alpha}
		}\\
		&=
\tidx{\left(\evalField[\pointMnfQTanSpaceRed]{
	\metricQRed}\right)}{ir}{}
\tidx{\left(\evalField[\pointMnfQTanSpaceRed]{
	\metricQRed}\right)}{}{r\ell}
\tidx{\velocityRed}{\ell}{}\,
\evalField[\pointMnfQTanSpaceRed]{
	\basisTMnfQRedTanPos{i}
}
\\
		&\phantom{=} +
\tidx{\left(\evalField[\pointMnfQTanSpaceRed]{
	\metricVRed}\right)}{ir}{}
\bigg(
	\evalField[{\pointQRed}]{
		\fracdiff{\tidx{\embQ}{j}{}}
			{\tidx{\chartmapMnfQRed}{r}{}}
	}
	\tidx{\left(\evalField[
		{\evalFun[
			\pointMnfQTanSpaceRed
		]{\embTanSpace}}
	]{
		\metricV}\right)}{}{jk}
	\tidx{\left(\evalField[{
		{\evalFun[
			{\pointMnfQTanSpaceRed}
		]{\embTanSpace}}
	}]{\vfLag}\right)}{\dimMnfQ + k}{}
	\\
		&\phantom{= +\bigg(\evalFun[\pointMnfQTanSpaceRed]{\tidx{(\metricVRed)}{ik}{}}}
	+
	\evalField[{\pointQRed}]{
		\sfracdiff{\tidx{\embQ}{j}{}}
			{\tidx{\chartmapMnfQRed}{p}{}}
			{\tidx{\chartmapMnfQRed}{r}{}\,}
	}
	\tidx{\velocityRed}{p}{}\,
	\tidx{\left(\evalField[
		{\evalFun[
			\pointMnfQTanSpaceRed
		]{\embTanSpace}}
	]{
		\metricQ}\right)}{}{jk}
	\evalField[{\pointQRed}]{
		\fracdiff{\tidx{\embQ}{k}{}}
			{\tidx{\chartmapMnfQRed}{\ell}{}}
	}
	\tidx{\velocityRed}{\ell}{}
	\\
		&\phantom{= +\bigg(\evalFun[\pointMnfQTanSpaceRed]{\tidx{(\metricVRed)}{ik}{}}}
	-
	\tidx{\left(\evalField[\pointMnfQTanSpaceRed]{
		\metricQVRed}\right)}{}{r\ell}
	\tidx{\left(\evalField[\pointMnfQTanSpaceRed]{
		\metricQRed}\right)}{\ell p}{}
	\tidx{\left(\evalField[\pointMnfQTanSpaceRed]{
		\metricQRed}\right)}{}{ps}\,
	\tidx{\velocityRed}{s}{}
\bigg)
\evalField[\pointMnfQTanSpaceRed]{
	\basisTMnfQRedTanVel{i}
}
\\
&=
\tidx{\velocityRed}{i}{}\,
\evalField[\pointMnfQTanSpaceRed]{
	\basisTMnfQRedTanPos{i}
}
+
\tidx{\left(\evalField[\pointMnfQTanSpaceRed]{
	\metricVRed}\right)}{ir}{}
\bigg(
	\evalField[{\pointQRed}]{
		\fracdiff{\tidx{\embQ}{j}{}}
			{\tidx{\chartmapMnfQRed}{r}{}}
	}
	\tidx{\left(\evalField[
		{\evalFun[
			\pointMnfQTanSpaceRed
		]{\embTanSpace}}
	]{
		\metricV}\right)}{}{jk}
	\tidx{\left( \evalField[{
		{\evalFun[
			{\pointMnfQTanSpaceRed}
		]{\embTanSpace}}
	}]{\vfLag} \right)}{\dimMnfQ + k}{}
	\\
		&\phantom{= \tidx{\velocityRed}{i}{}\,\evalField[\pointMnfQTanSpaceRed]{\basisTMnfQRedTanPos{i}} +\bigg(\evalFun[\pointMnfQTanSpaceRed]{\tidx{(\metricVRed)}{ik}{}}}
	+
	\evalField[{\pointQRed}]{
		\sfracdiff{\tidx{\embQ}{j}{}}
			{\tidx{\chartmapMnfQRed}{p}{}}
			{\tidx{\chartmapMnfQRed}{r}{}\,}
	}
	\tidx{\velocityRed}{p}{}\,
	\tidx{\left(\evalField[
		{\evalFun[
			\pointMnfQTanSpaceRed
		]{\embTanSpace}}
	]{
		\metricQ}\right)}{}{jk}
	\evalField[{\pointQRed}]{
		\fracdiff{\tidx{\embQ}{k}{}}
			{\tidx{\chartmapMnfQRed}{\ell}{}}
	}
	\tidx{\velocityRed}{\ell}{}
	\\
		&\phantom{= \tidx{\velocityRed}{i}{}\,\evalField[\pointMnfQTanSpaceRed]{\basisTMnfQRedTanPos{i}} +\bigg(\evalFun[\pointMnfQTanSpaceRed]{\tidx{(\metricVRed)}{ik}{}}}
	-
	\evalField[\pointQRed]{
		\sfracdiff{\tidx{\embQ}{j}{}}
			{\tidx{\chartmapMnfQRed}{r}{}}
			{\tidx{\chartmapMnfQRed}{p}{}\,}
	}
	\tidx{\velocityRed}{p}{}\,
	\tidx{\left(\evalField[
		{\evalFun[
			\pointMnfQTanSpaceRed
		]{\embTanSpace}}
	]{
		\metricQ}\right)}{}{jk}
	\evalField[\pointQRed]{
		\fracdiff{\tidx{\embQ}{k}{}}
			{\tidx{\chartmapMnfQRed}{\ell}{}}
	}
	\tidx{\velocityRed}{\ell}{}\\
		&\phantom{= \tidx{\velocityRed}{i}{}\,\evalField[\pointMnfQTanSpaceRed]{\basisTMnfQRedTanPos{i}} +\bigg(\evalFun[\pointMnfQTanSpaceRed]{\tidx{(\metricVRed)}{ik}{}}}
	-
	\evalField[\pointQRed]{
		\fracdiff{\tidx{\embQ}{j}{}}
			{\tidx{\chartmapMnfQRed}{r}{}}
	}
	\tidx{\left(\evalField[
		{\evalFun[
			\pointMnfQTanSpaceRed
		]{\embTanSpace}}
	]{
		\metricV}\right)}{}{jk}
	\evalField[\pointQRed]{
		\sfracdiff{\tidx{\embQ}{k}{}}
			{\tidx{\chartmapMnfQRed}{p}{}}
			{\tidx{\chartmapRed}{\ell}{}\,}
	}
	\tidx{\velocityRed}{\ell}{}\,
	\tidx{\velocityRed}{p}{}
\bigg)
\evalField[\pointMnfQTanSpaceRed]{
	\basisTMnfQRedTanVel{i}
}
\\
		&=
\evalField[\pointMnfQTanSpaceRed]{\vfLagRed}
	\end{align*}
	Thus, the vector field obtained with the \LMG reduction~\eqref{eqn:lmg_reduction} with the \LMG tensor field~$\metricLMG$ from~\eqref{eqn:LMGtensorField} results in the reduced Euler--Lagrange vector field \eqref{eq:red_vf_euler_lagrange}, which is equivalent to solving the reduced Euler--Lagrange equations~\eqref{eq:red_euler_lagrange} by construction.

%% file: BucGHU23.bbl
\begin{thebibliography}{10}

\bibitem{Abraham1987}
\textsc{R.~Abraham and J.~E. Marsden}.
\newblock \href{https://resolver.caltech.edu/CaltechBOOK:1987.001}{{\em
  Foundations of Mechanics}}.
\newblock Addison-Wesley Publishing Company, Redwood City, CA, second edition,
  1987.

\bibitem{AmsCCF09}
\textsc{D.~Amsallem, J.~Cortial, K.~Carlberg, and C.~Farhat}.
\newblock \href{https://doi.org/10.1002/nme.2681}{A method for interpolating on
  manifolds structural dynamics reduced-order models}.
\newblock {\em Internat. J. Numer. Methods Engrg.}, 80(9), 2009.

\bibitem{AntBG20}
\textsc{A.~C. Antoulas, C.~Beattie, and S.~Gugercin}.
\newblock \href{https://doi.org/10.1137/1.9781611976083}{{\em Interpolatory
  Methods for Model Reduction}}.
\newblock SIAM, Philadelphia, PA, USA, 2020.

\bibitem{Barnett2022}
\textsc{J.~Barnett and C.~Farhat}.
\newblock \href{https://doi.org/10.1016/j.jcp.2022.111348}{Quadratic
  approximation manifold for mitigating the {Kolmogorov} barrier in nonlinear
  projection-based model order reduction}.
\newblock {\em J. Comput. Phys.}, 464:111348, 2022.

\bibitem{Barnett2023}
\textsc{J.~L. Barnett, C.~Farhat, and Y.~Maday}.
\newblock \href{https://doi.org/10.2514/6.2023-0535}{Mitigating the
  {Kolmogorov} barrier for the reduction of aerodynamic models using
  neural-network-augmented reduced-order models}.
\newblock In {\em AIAA SCITECH 2023 Forum}, 2023.

\bibitem{BenCOW17}
\textsc{P.~Benner, A.~Cohen, M.~Ohlberger, and K.~Willcox}.
\newblock \href{https://doi.org/10.1137/1.9781611974829}{{\em Model Reduction
  and Approximation}}.
\newblock Advances in Design and Control. SIAM, Philadelphia, PA, USA, 2017.

\bibitem{Benner2022}
\textsc{P.~Benner, P.~Goyal, J.~Heiland, and I.~Pontes~Duff}.
\newblock \href{https://doi.org/10.1002/pamm.202200049}{A quadratic decoder
  approach to nonintrusive reduced-order modeling of nonlinear dynamical
  systems}.
\newblock {\em PAMM}, 23(1):e202200049, 2023.

\bibitem{BeyT04}
\textsc{W.~J. Beyn and V.~Th{\"{u}}mmler}.
\newblock \href{https://doi.org/10.1137/030600515}{Freezing solutions of
  equivariant evolution equations}.
\newblock {\em {SIAM} J. Appl. Dyn. Syst.}, 3(2):85--116, 2004.

\bibitem{Bishop1968}
\textsc{R.~Bishop and S.~Goldberg}.
\newblock \href{https://store.doverpublications.com/0486640396.html}{{\em
  Tensor analysis on manifolds}}.
\newblock Macmillan, New York, 1968.

\bibitem{BlaSU20}
\textsc{F.~Black, P.~Schulze, and B.~Unger}.
\newblock \href{https://doi.org/10.1051/m2an/2020046}{Projection-based model
  reduction with dynamically transformed modes}.
\newblock {\em ESAIM: Math. Model. Numer. Anal.}, 54(6):2011--2043, 2020.

\bibitem{BlaSU21b}
\textsc{F.~Black, P.~Schulze, and B.~Unger}.
\newblock \href{https://doi.org/10.3390/fluids6080280}{Efficient wildland fire
  simulation via nonlinear model order reduction}.
\newblock {\em Fluids}, 6(8):280, 2021.

\bibitem{BlaSU22}
\textsc{F.~Black, P.~Schulze, and B.~Unger}.
\newblock \href{https://doi.org/10.1007/978-3-030-90727-3_13}{Modal
  decomposition of flow data via gradient-based transport optimization}.
\newblock In \textsc{R.~King and D.~Peitsch}, editors, {\em Active Flow and
  Combustion Control 2021}, pages 203--224, Cham, 2022. Springer International
  Publishing.

\bibitem{BonF21}
\textsc{G.~Boncoraglio and C.~Farhat}.
\newblock \href{https://doi.org/10.2514/1.J060581}{Active manifold and
  model-order reduction to accelerate multidisciplinary analysis and
  optimization}.
\newblock {\em AIAA Journal}, 59(11):4739--4753, 2021.

\bibitem{Buchfink2019}
\textsc{P.~Buchfink, A.~Bhatt, and B.~Haasdonk}.
\newblock \href{https://doi.org/10.3390/mca24020043}{Symplectic model order
  reduction with non-orthonormal bases}.
\newblock {\em Math. Comput. Appl.}, 24(2), 2019.

\bibitem{BucGH23}
\textsc{P.~Buchfink, S.~Glas, and B.~Haasdonk}.
\newblock \href{https://doi.org/10.48550/arXiv.2312.00724}{Approximation bounds
  for model reduction on polynomially mapped manifolds}.
\newblock {\em ArXiv e-print 2312.00724}, 2023.

\bibitem{BucGH21}
\textsc{P.~Buchfink, S.~Glas, and B.~Haasdonk}.
\newblock \href{https://doi.org/10.1137/21M1466657}{Symplectic model reduction
  of {Hamiltonian} systems on nonlinear manifolds and approximation with weakly
  symplectic autoencoder}.
\newblock {\em {SIAM} J. Sci. Comput.}, 45(2):A289--A311, 2023.

\bibitem{CagMS19}
\textsc{N.~Cagniart, Y.~Maday, and B.~Stamm}.
\newblock \href{https://doi.org/10.1007/978-3-319-78325-3}{Model order
  reduction for problems with large convection effects}.
\newblock In \textsc{B.~N. Chetverushkin, W.~Fitzgibbon, Y.~A. Kuznetsov,
  P.~Neittaanm{\"a}ki, J.~Periaux, and O.~Pironneau}, editors, {\em
  Contributions to Partial Differential Equations and Applications},
  Computational Methods in Applied Sciences, pages 131--150. Springer-Verlag,
  Cham, Switzerland, 2019.

\bibitem{Carlberg2015}
\textsc{K.~Carlberg, R.~Tuminaro, and P.~Boggs}.
\newblock \href{https://doi.org/10.1137/140959602}{Preserving {L}agrangian
  structure in nonlinear model reduction with application to structural
  dynamics}.
\newblock {\em {SIAM} J. Sci. Comput.}, 37(2):B153--B184, 2015.

\bibitem{Cohen2023}
\textsc{A.~Cohen, C.~Farhat, Y.~Maday, and A.~Somacal}.
\newblock \href{https://doi.org/10.5802/crmeca.191}{Nonlinear compressive
  reduced basis approximation for {PDE{\textquoteright}s}}.
\newblock {\em Comptes Rendus. M\'ecanique}, pages 1--8, 2023.
\newblock Online first.

\bibitem{Con15}
\textsc{P.~G. Constantine}.
\newblock \href{https://doi.org/10.1137/1.9781611973860}{{\em Active Subspaces:
  Emerging Ideas for Dimension Reduction in Parameter Studies}}.
\newblock SIAM Spotlights. SIAM, 2015.

\bibitem{CopCHC22}
\textsc{D.~M. Copeland, S.~W. Cheung, K.~Huynh, and Y.~Choi}.
\newblock \href{https://doi.org/10.1016/j.cma.2021.114259}{Reduced order models
  for {Lagrangian} hydrodynamics}.
\newblock {\em Comput. Meth. Appl. Mech. Eng.}, 388:114259, 2022.

\bibitem{FerTZ22}
\textsc{A.~Ferrero, T.~Taddei, and L.~Zhang}.
\newblock \href{https://doi.org/10.1016/j.jcp.2022.111068}{Registration-based
  model reduction of parameterized two-dimensional conservation laws}.
\newblock {\em J. Comput. Phys.}, 457:111068, 2022.

\bibitem{Geelen2022}
\textsc{R.~Geelen, S.~Wright, and K.~Willcox}.
\newblock \href{https://doi.org/10.1016/j.cma.2022.115717}{Operator inference
  for non-intrusive model reduction with quadratic manifolds}.
\newblock {\em Comput. Meth. Appl. Mech. Eng.}, 403:115717, 2023.

\bibitem{GonWW2017}
\textsc{Y.~Gong, Q.~Wang, and Z.~Wang}.
\newblock \href{https://doi.org/10.1016/j.cma.2016.11.016}{Structure-preserving
  {G}alerkin {POD} reduced-order modeling of {H}amiltonian systems}.
\newblock {\em Comput. Meth. Appl. Mech. Eng.}, 315:780--798, 2017.

\bibitem{Goodfellow2016}
\textsc{I.~Goodfellow, Y.~Bengio, and A.~Courville}.
\newblock \href{http://www.deeplearningbook.org}{{\em Deep Learning}}.
\newblock MIT Press, 2016.

\bibitem{GreU19}
\textsc{C.~Greif and K.~Urban}.
\newblock \href{https://doi.org/10.1016/j.aml.2019.05.013}{Decay of the
  {Kolmogorov} $n$-width for wave problems}.
\newblock {\em Appl. Math. Lett.}, 96:216--222, 2019.

\bibitem{GreDY19}
\textsc{S.~Greydanus, M.~Dzamba, and J.~Yosinski}.
\newblock \href{https://doi.org/10.5555/3454287.3455665}{{Hamiltonian} neural
  networks}.
\newblock In {\em NIPS'19: Proceedings of the 33rd International Conference on
  Neural Information Processing Systems}, pages 15379--15389, 2019.

\bibitem{Gruber2023}
\textsc{A.~Gruber, M.~Gunzburger, L.~Ju, and Z.~Wang}.
\newblock \href{https://doi.org/10.1016/j.cma.2022.115709}{Energetically
  consistent model reduction for metriplectic systems}.
\newblock {\em Computer Methods in Applied Mechanics and Engineering},
  404:115709, 2023.

\bibitem{GruT23}
\textsc{A.~Gruber and I.~Tezaur}.
\newblock \href{https://doi.org/10.1016/j.cma.2023.116334}{Canonical and
  noncanonical {Hamiltonian} operator inference}.
\newblock {\em Comput. Meth. Appl. Mech. Eng.}, 416:116334, 2023.

\bibitem{Gu11}
\textsc{C.~Gu}.
\newblock \href{https://doi.org/10.1109/TCAD.2011.2142184}{{QLMOR}: {A}
  projection-based nonlinear model order reduction approach using
  quadratic-linear representation of nonlinear systems}.
\newblock {\em IEEE Trans. Comput.-Aided Des. Integr.}, 30(9):1307--1320, 2011.

\bibitem{Haasdonk2017}
\textsc{B.~Haasdonk}.
\newblock \href{https://doi.org/10.1137/1.9781611974829.ch2}{Reduced basis
  methods for parametrized {PDE}s—a tutorial introduction for stationary and
  instationary problems}.
\newblock In \textsc{P.~Benner, M.~Ohlberger, A.~Cohen, and K.~Willcox},
  editors, {\em Model Reduction and Approximation}, pages 65--136. SIAM, 2017.

\bibitem{ShaWK22}
\textsc{H.~Harma, Z.~Wang, and B.~Kramer}.
\newblock \href{https://doi.org/10.1016/j.physd.2021.133122}{{Hamiltonian}
  operator inference: Physics-preserving learning of reduced-order models for
  canonical {Hamiltonian} systems}.
\newblock {\em Phys. D}, 431, 2022.

\bibitem{HarM17}
\textsc{D.~Hartmann and L.~K. Mestha}.
\newblock \href{https://doi.org/10.1109/CCTA.2017.8062736}{A deep learning
  framework for model reduction of dynamical systems}.
\newblock In {\em Proc. IEEE Conference on Control Technology and Applications
  (CCTA), Kohala Coast, USA}, pages 1917--1922, 2017.

\bibitem{Hernandez2023}
\textsc{Q.~Hern{\'{a}}ndez, A.~Bad{\'{\i}}as, F.~Chinesta, and E.~Cueto}.
\newblock \href{https://doi.org/10.1007/s00466-023-02296-w}{Port-metriplectic
  neural networks: thermodynamics-informed machine learning of complex physical
  systems}.
\newblock {\em Comput. Mech.}, 72:553--561, 2023.

\bibitem{HesRS16}
\textsc{J.~S. Hesthaven, G.~Rozza, and B.~Stamm}.
\newblock \href{https://doi.org/10.1007/978-3-319-22470-1}{{\em Certified
  Reduced Basis Methods for Parametrized Partial Differential Equations}}.
\newblock SpringerBriefs in Mathematics. Springer-Verlag, Cham, Switzerland,
  2016.

\bibitem{HesP18}
\textsc{J.~S. Hesthaven and C.~Pagliantini}.
\newblock \href{https://doi.org/10.1090/mcom/3618}{Structure-preserving reduced
  basis methods for {Poisson} systems}.
\newblock {\em Math. Comp.}, 90(330):1701--1740, 2021.

\bibitem{Issan2023}
\textsc{O.~Issan and B.~Kramer}.
\newblock \href{https://doi.org/10.1016/j.jcp.2022.111689}{Predicting solar
  wind streams from the inner-heliosphere to earth via shifted operator
  inference}.
\newblock {\em J. Comput. Phys.}, 473:111689, 2023.

\bibitem{Kas16}
\textsc{K.~Kashima}.
\newblock \href{https://doi.org/10.1109/CDC. 2016.7799153}{Nonlinear model
  reduction by deep autoencoder of noise response data}.
\newblock In {\em Proc.~55th IEEE Conference on Decision and Control (CDC), Las
  Vegas, USA}, pages 5750--5755, 2016.

\bibitem{KimCWZ22}
\textsc{Y.~Kim, Y.~Choi, D.~Widemann, and T.~Zohdi}.
\newblock \href{https://doi.org/10.1016/j.jcp.2021.110841}{A fast and accurate
  physics-informed neural network reduced order model with shallow masked
  autoencoder}.
\newblock {\em J. Comput. Phys.}, 451:110841, 2022.

\bibitem{KirA92}
\textsc{M.~Kirby and D.~Armbruster}.
\newblock \href{https://doi.org/10.1007/BF00916425}{Reconstructing phase space
  from {PDE} simulations}.
\newblock {\em Z. Angew. Math. Phys.}, 42:999--1022, 1992.

\bibitem{Kol36}
\textsc{A.~Kolmogoroff}.
\newblock \href{https://doi.org/10.2307/1968691}{{\"U}ber die beste
  {A}nn{\"a}herung von {F}unktionen einer gegebenen {F}unktionenklasse}.
\newblock {\em Ann. of Math. (2)}, 37(1):107--110, 1936.

\bibitem{KraBHR22}
\textsc{P.~Krah, S.~B\"uchholz, M.~H\"aringer, and J.~Reiss}.
\newblock \href{https://doi.org/10.1007/s10915-023-02210-9}{Front transport
  reduction for complex moving fronts}.
\newblock {\em J. Sci. Comput.}, 96:28, 2023.

\bibitem{KraW19}
\textsc{B.~Kramer and K.~E. Willcox}.
\newblock \href{https://doi.org/10.2514/1.J057791}{Nonlinear model order
  reduction via lifting transformations and proper orthogonal decomposition}.
\newblock {\em AIAA Journal}, 57(6):2297--2307, 2019.

\bibitem{Lall2003}
\textsc{S.~Lall, P.~Krysl, and J.~E. Marsden}.
\newblock
  \href{https://doi.org/10.1016/S0167-2789(03)00227-6}{Structure-preserving
  model reduction for mechanical systems}.
\newblock {\em Phys. D}, 184(1):304--318, 2003.

\bibitem{Lee12}
\textsc{J.~M. Lee}.
\newblock \href{https://doi.org/10.1007/978-1-4419-9982-5}{{\em Introduction to
  Smooth Manifolds}}.
\newblock Graduate Texts in Mathematics. Springer-Verlag, New York, NY, 2nd
  edition, 2012.

\bibitem{Lee20}
\textsc{K.~Lee and K.~T. Carlberg}.
\newblock \href{https://doi.org/10.1016/j.jcp.2019.108973}{Model reduction of
  dynamical systems on nonlinear manifolds using deep convolutional
  autoencoders}.
\newblock {\em J. Comput. Phys.}, 404:108973, 2020.

\bibitem{AfkH17}
\textsc{B.~{Maboudi Afkham} and J.~S. Hesthaven}.
\newblock \href{https://doi.org/10.1137/17M1111991}{Structure preserving model
  reduction of parametric {H}amiltonian systems}.
\newblock {\em {SIAM} J. Sci. Comput.}, 39(6):A2616--A2644, 2017.

\bibitem{MaboudiAfkham2018a}
\textsc{B.~{Maboudi Afkham}, A.~{Bhatt}, B.~{Haasdonk}, and J.~S. {Hesthaven}}.
\newblock \href{https://doi.org/10.48550/arXiv.1803.07799}{Symplectic
  model-reduction with a weighted inner product}.
\newblock {\em ArXiv e-print 1803.07799}, 2018.

\bibitem{MarR99}
\textsc{J.~E. Marsden and T.~S. Ratiu}.
\newblock \href{https://doi.org/10.1007/978-0-387-21792-5}{{\em Introduction to
  Mechanics and Symmetry: A Basic Exposition of Classical Mechanical Systems}}.
\newblock Springer New York, 1999.

\bibitem{MasGSSA19}
\textsc{E.~Massart, P.-Y. Gousenbourger, N.~T. Son, T.~Stykel, and P.-A.
  Absil}.
\newblock
  \href{https://www.esann.org/sites/default/files/proceedings/legacy/es2019-164.pdf}{Interpolation
  on the manifold of fixed-rank positive-semidefinite matrices for parametric
  model order reduction: preliminary results}.
\newblock In {\em Proceedings of the 27th European Symposium on Artifical
  Neural Networks, Computational Intelligence and Machine Learning
  (ESANN2019)}, pages 281--286, 2019.

\bibitem{MehU23}
\textsc{V.~Mehrmann and B.~Unger}.
\newblock \href{https://doi.org/10.1017/S0962492922000083}{Control of
  port-{Hamiltonian} differential-algebraic systems and applications}.
\newblock {\em Acta Numer.}, 32:395--515, 2023.

\bibitem{MenBALK20}
\textsc{A.~Mendible, S.~L. Brunton, A.~Y. Aravkin, W.~Lowrie, and J.~N. Kutz}.
\newblock \href{https://doi.org/10.1007/s00162-020-00529-9}{Dimensionality
  reduction and reduced-order modeling for traveling wave physics}.
\newblock {\em Theor. Comput. Fluid Dyn.}, 34:385--400, 2020.

\bibitem{MojB17}
\textsc{R.~Mojgani and M.~Balajewicz}.
\newblock \href{https://doi.org/10.48550/arXiv.1701.04343}{{Lagrangian} basis
  method for dimensionality reduction of convection dominated nonlinear flows}.
\newblock {\em ArXiv e-print 1701.04343v1}, 2017.

\bibitem{Morrison1986}
\textsc{P.~J. Morrison}.
\newblock \href{https://doi.org/10.1016/0167-2789(86)90209-5}{A paradigm for
  joined {Hamiltonian} and dissipative systems}.
\newblock {\em Phys. D}, 18(1):410--419, 1986.

\bibitem{OhlR13}
\textsc{M.~Ohlberger and S.~Rave}.
\newblock \href{https://doi.org/10.1016/j.crma.2013.10.028}{Nonlinear reduced
  basis approximation of parameterized evolution equations via the method of
  freezing}.
\newblock {\em C. R. Math. Acad. Sci. Paris}, 351(23--24):901--906, 2013.

\bibitem{OhlR16}
\textsc{M.~Ohlberger and S.~Rave}.
\newblock
  \href{http://www.iam.fmph.uniba.sk/amuc/ojs/index.php/algoritmy/article/view/389}{Reduced
  basis methods: Success, limitations and future challenges}.
\newblock {\em Proceedings of the Conference Algoritmy}, pages 1--12, 2016.

\bibitem{OttMR23}
\textsc{S.~E. Otto, G.~R. Macchio, and C.~W. Rowley}.
\newblock \href{https://doi.org/10.1063/5.0169688}{{Learning nonlinear
  projections for reduced-order modeling of dynamical systems using constrained
  autoencoders}}.
\newblock {\em Chaos: An Interdisciplinary Journal of Nonlinear Science},
  33(11):113130, 2023.

\bibitem{OttPR22}
\textsc{S.~E. Otto, A.~Padovan, and C.~W. Rowley}.
\newblock \href{https://doi.org/10.1137/21M1425815}{Optimizing oblique
  projections for nonlinear systems using trajectories}.
\newblock {\em {SIAM} J. Sci. Comput.}, 44(3):A1681--A1702, 2022.

\bibitem{OulA21}
\textsc{M.~Oulghelou and C.~Allery}.
\newblock \href{https://doi.org/10.1016/j.jcp.2020.109924}{Non intrusive method
  for parametric model order reduction using a bi-calibrated interpolation on
  the {Grassmann} manifold}.
\newblock {\em J. Comput. Phys.}, 426:109924, 2021.

\bibitem{PenM2016}
\textsc{L.~Peng and K.~Mohseni}.
\newblock \href{https://doi.org/10.1137/140978922}{Symplectic model reduction
  of {H}amiltonian systems}.
\newblock {\em {SIAM} J. Sci. Comput.}, 38(1):A1--A27, 2016.

\bibitem{Pin85}
\textsc{A.~Pinkus}.
\newblock \href{https://doi.org/10.1007/978-3-642-69894-1}{{\em N-widths in
  approximation theory}}.
\newblock Ergebnisse der Mathematik und ihrer Grenzgebiete. Springer-Verlag,
  Heidelberg, Germany, 1985.

\bibitem{QiaKPW20}
\textsc{E.~Qian, B.~Kramer, B.~Peherstorfer, and K.~Willcox}.
\newblock \href{https://doi.org/10.1016/j.physd.2020.132401}{Lift \& {L}earn:
  {P}hysics-informed machine learning for large-scale nonlinear dynamical
  systems}.
\newblock {\em Phys. D}, 406:132401, 2020.

\bibitem{QuaMN16}
\textsc{A.~Quarteroni, A.~Manzoni, and F.~Negri}.
\newblock \href{https://doi.org/10.1007/978-3-319-15431-2}{{\em Reduced Basis
  Methods for Partial Differential Equations: An Introduction}}.
\newblock UNITEXT. Springer-Verlag, Cham, Switzerland, 2016.

\bibitem{QuaR14}
\textsc{A.~Quarteroni and G.~Rozza}, editors.
\newblock \href{https://doi.org/10.1007/978-3-319-02090-7}{{\em Reduced Order
  Methods for Modeling and Computational Reduction}}.
\newblock Number~9 in {MS}\&{A} - {Modeling}, simulation and applications.
  Springer, Cham, Switzerland, 2014.

\bibitem{ReiSSM18}
\textsc{J.~Reiss, P.~Schulze, J.~Sesterhenn, and V.~Mehrmann}.
\newblock \href{https://doi.org/10.1137/17M1140571}{The shifted proper
  orthogonal decomposition: A mode decomposition for multiple transport
  phenomena}.
\newblock {\em {SIAM} J. Sci. Comput.}, 40(3):A1322--A1344, 2018.

\bibitem{RimML18}
\textsc{D.~Rim, S.~Moe, and R.~J. LeVeque}.
\newblock \href{https://doi.org/10.1137/17M1113679}{Transport reversal for
  model reduction of hyperbolic partial differential equations}.
\newblock {\em {SIAM/ASA} J. Uncertain. Quantif.}, 6(1):118--150, 2018.

\bibitem{RowM00}
\textsc{C.~W. Rowley and J.~E. Marsden}.
\newblock \href{https://doi.org/10.1016/S0167-2789(00)00042-7}{Reconstruction
  equations and the {K}arhunen--{L}o{\`e}ve expansion for systems with
  symmetry}.
\newblock {\em Phys. D}, 142(1--2):1--19, 2000.

\bibitem{Ruiner2012}
\textsc{T.~Ruiner, J.~Fehr, B.~Haasdonk, and P.~Eberhard}.
\newblock \href{https://doi.org/10.1007/s10409-012-0114-7}{A-posteriori error
  estimation for second order mechanical systems}.
\newblock {\em {Acta Mech. Sin.}}, 28(3):854--862, 2012.

\bibitem{Sch23}
\textsc{P.~Schulze}.
\newblock \href{https://doi.org/10.14279/depositonce-17843}{{\em Energy-based
  model reduction of transport-dominated phenomena}}.
\newblock Dissertation, Technische Universit{\"a}t Berlin, Institut f{\"u}r
  Mathematik, 2023.

\bibitem{SchRM19}
\textsc{P.~Schulze, J.~Reiss, and V.~Mehrmann}.
\newblock \href{https://doi.org/10.1007/978-3-319-98177-2_17}{Model reduction
  for a pulsed detonation combuster via shifted proper orthogonal
  decomposition}.
\newblock In \textsc{R.~King}, editor, {\em Active Flow and Combustion Control
  2018}, pages 271--286. Springer, Cham, Switzerland, 2019.

\bibitem{SchSBP23}
\textsc{P.~Schwerdtner, P.~Schulze, J.~Berman, and B.~Peherstorfer}.
\newblock \href{https://doi.org/10.48550/arXiv.2310.07485}{Nonlinear embeddings
  for conserving {Hamiltonians} and other quantities with {Neural} {Galerkin}
  schemes}.
\newblock {\em ArXiv e-print 2310.07485}, 2023.

\bibitem{Sharma2023}
\textsc{H.~Sharma, H.~Mu, P.~Buchfink, R.~Geelen, S.~Glas, and B.~Kramer}.
\newblock \href{https://doi.org/10.1016/j.cma.2023.116402}{Symplectic model
  reduction of {H}amiltonian systems using data-driven quadratic manifolds}.
\newblock {\em Comput. Meth. Appl. Mech. Eng.}, 417:116402, 2023.

\bibitem{Son12}
\textsc{N.~T. Son}.
\newblock \href{https://doi.org/10.1002/nme.4408}{A real time procedure for
  affinely dependent parametric model order reduction using interpolation on
  {Grassmann} manifolds}.
\newblock {\em Internat. J. Numer. Methods Engrg.}, 93(8):818--833, 2012.

\bibitem{Tad20}
\textsc{T.~Taddei}.
\newblock \href{https://doi.org/10.1137/19M127127}{A registration method for
  model order reduction: Data compression and geometry reduction}.
\newblock {\em {SIAM} J. Sci. Comput.}, 42(2):A997--A1027, 2020.

\bibitem{UngG19}
\textsc{B.~{Unger} and S.~{Gugercin}}.
\newblock \href{https://doi.org/10.1007/s10444-019-09701-0}{{K}olmogorov
  $n$-widths for linear dynamical systems}.
\newblock {\em Adv. Comput. Math.}, 45(5--6):2273--2286, 2019.

\bibitem{SchJ14}
\textsc{A.~{van der S}chaft and D.~Jeltsema}.
\newblock \href{https://doi.org/10.1561/2600000002}{Port-{H}amiltonian systems
  theory: {A}n introductory overview}.
\newblock {\em Foundations and Trends in Systems and Control}, 1(2-3):173--378,
  2014.

\bibitem{Volkwein2013}
\textsc{S.~Volkwein}.
\newblock
  \href{https://www.math.uni-konstanz.de/numerik/personen/volkwein/teaching/POD-Book.pdf}{Proper
  orthogonal decomposition: Theory and reduced-order modelling}, August 2013.
\newblock Lecture Notes, University of Konstanz.

\bibitem{YilGBB23}
\textsc{S.~Yildiz, P.~Goyal, T.~Bendokat, and P.~Benner}.
\newblock \href{https://doi.org/10.48550/arXiv.2308.01084}{Data-driven
  identification of quadratic symplectic representations of nonlinear
  {Hamiltonian} systems}.
\newblock {\em ArXiv e-print 2308.01084}, 2023.

\bibitem{YilGB23}
\textsc{S.~Yildiz, P.~Goyal, and P.~Benner}.
\newblock \href{https://doi.org/10.48550/arXiv.2308.02625}{Linearly implicit
  global energy preserving reduced-order models for cubic {Hamiltonian}
  systems}.
\newblock {\em ArXiv e-print 2308.02625}, 2023.

\bibitem{Zim21}
\textsc{R.~Zimmermann}.
\newblock \href{https://doi.org/10.1515/9783110498967-007}{Manifold
  interpolation}.
\newblock In \textsc{P.~Benner, S.~Grivet-Talocia, A.~Quarteroni, G.~Rozza,
  W.~Schilders, and L.~M. Silveira}, editors, {\em Model Order Reduction -
  Volume 1: System- and Data-Driven Methods and Algorithms}, pages 229--274. De
  Gruyter, Berlin, Boston, 2021.

\end{thebibliography}
